\documentclass[11pt,oneside,reqno]{amsart}
\usepackage{geometry}
\usepackage[utf8]{inputenc}
\usepackage{amstext,latexsym,amsbsy,amsmath,amssymb,amsthm,mathtools,relsize,geometry,enumerate}
\usepackage{hyperref}
\hypersetup{
  colorlinks   = true, 
  urlcolor     = blue, 
  linkcolor    = red, 
  citecolor   = red 
}
\usepackage{mathtools,enumerate,enumitem}

\usepackage{mathrsfs}

\usepackage{comment}
\usepackage{graphicx}
\usepackage{epstopdf}
\setkeys{Gin}{width=\linewidth,totalheight=\textheight,keepaspectratio}
\graphicspath{{./images/}}

\usepackage[normalem]{ulem}

\usepackage{multicol}
\usepackage{booktabs}
\usepackage{mathrsfs}
\usepackage{color}

\newcommand{\nbb}{\mathbb{N}}
\newcommand{\rbb}{\mathbb{R}}

\newcommand{\cbb}{\mathbb{C}}

\renewcommand{\L}{\mathcal{L}}

\newcommand{\Pcal}{\mathcal{P}}
\newcommand{\B}{\mathcal{B}}
\newcommand{\la}{\langle}
\newcommand{\ra}{\rangle}

\newcommand{\ug}{u^\gamma}

\renewcommand{\i}{\textup{i}}

\newcommand{\mi}{\wedge}

\newcommand{\HS}{\textup{HS}}

\renewcommand{\d}{\textup{d}}

\newcommand{\grad}{\nabla}

\newcommand{\Fcal}{\mathcal{F}}
\newcommand{\E}{\mathbb{E}}

\renewcommand{\P}{\mathbb{P}}

\renewcommand{\Re}{\text{Re}}

\theoremstyle{plain}
\newtheorem{theorem}{Theorem}[section]

\newtheorem{lemma}[theorem]{Lemma}
\newtheorem{assumption}[theorem]{Assumption}
\newtheorem{proposition}[theorem]{Proposition}
\newtheorem{definition}[theorem]{Definition}

\theoremstyle{definition}

\newtheorem{remark}[theorem]{Remark}

\numberwithin{equation}{section}
 \title{Polynomial mixing for the stochastic Schr\"odinger equation with large damping in the whole space}

\author{Hung D.~Nguyen$^1$ and Kihoon Seong$^2$}

\address{$^1$ Department of Mathematics, University of Tennessee, Knoxville, Tennessee, USA}

\address{$^2$ Department of Mathematics, Cornell University, Ithaca, New York, USA, and Simons Laufer Mathematical Sciences Institute (SLMath), 17 Gauss Way, Berkeley, CA 94720, USA}

\begin{document}

\begin{abstract}
We study the long-time mixing behavior of the stochastic nonlinear Schr\"odinger equation in $\rbb^d$, $d\le 3$.  It is well known that, under a sufficiently strong damping force, the system admits unique ergodicity, although the rate of convergence toward equilibrium has remained unknown. In this work, we address the mixing property in the regime of large damping and establish that solutions are attracted toward the unique invariant probability measure at polynomial rates of arbitrary order. Our approach is based on a coupling strategy with pathwise Strichartz estimates.

\vspace{5mm}

\end{abstract}

\vspace{5mm}

\maketitle

\section{Introduction} \label{sec:intro}
In this paper, we study the mixing behavior of the nonlinear Schr\"odinger equation,
\begin{equation} \label{eqn:Schrodinger:original}
\i \d u (t) +  \triangle u(t)\d t +  \alpha  |u(t)|^{2\sigma}u(t)\d t+\lambda u(t)\d t = Q \d W(t),
\end{equation}

\noindent
where  $u:[0,\infty)\times\rbb^d\to \cbb$, $d \le 3$. On the left-hand side of \eqref{eqn:Schrodinger:original}, $\lambda>0$ denotes the damping constant, measuring the strength of the damping force,  $\sigma>0$ represents the nonlinear effect, $\alpha\in\{ -1,1\}$ where $\alpha=1$ and $\alpha=-1$ respectively correspond to the focusing and defocusing equations, $QW(t)$ is a white-in-time, color-in-space Wiener process defined on some Hilbert space $U$ and whose spatial covariance operator is given by the mapping $Q:U\to L^2(\rbb^d)$.

It is well-known that the Schr\"odinger equation posed on one-dimensional compact intervals with cubic nonlinearity possesses a unique invariant probability measure, which is polynomially attractive \cite{debussche2005ergodicity,nguyen2024inviscid}. While there are several existence and unique ergodic results for equation \eqref{eqn:Schrodinger:original} in $\rbb^d$ \cite{brzezniak2023ergodic, brzezniak2023invariant, ekren2017existence, ferrario2025stationary, kim2008stochastic}, to the best of the authors' knowledge, the issue of mixing rate of \eqref{eqn:Schrodinger:original} in $\rbb^d$ is still poorly understood. Our goal is to make progress in this direction with the aim of bridging the gap between bounded domains and the whole space. More specifically, the main result of the paper establishes an algebraic convergence rate of \eqref{eqn:Schrodinger:original} toward equilibrium, provided a sufficiently strong damping effect. This is summarized below through Theorem \ref{thm:pseudo:poly-mixing}, whose rigorous statement can be found in Theorem \ref{thm:unique-ergodicity} and Theorem \ref{thm:poly-mixing}.

\begin{theorem} \label{thm:pseudo:poly-mixing}
In dimension $d\le 3$, let $u(t;u_0)$ be the solution of \eqref{eqn:Schrodinger:original} with initial condition $u_0$. 

1. Under appropriate assumptions on $\sigma$ and $QW$, for all $\lambda$ large enough, there exists a unique invariant probability measure for \eqref{eqn:Schrodinger:original}.  

2. Suppose $\sigma$ satisfies further restriction. Then, the following holds for all suitable observables $f:\cbb\to\rbb$ and initial conditions $u_0^1$ and $u_0^2$
\begin{align} \label{ineq:poly-mixing:pseudo}
    \big|\E f \big(u(t;u_0^1)\big) - \E f \big(u(t;u_0^2)\big)\big| \le \frac{C}{(1+t)^q},\quad t\ge 0,\, q>0.
\end{align}
In the above, $C$ is a positive constant independent of $t$.

\end{theorem}

Historically, the well-posedness of \eqref{eqn:Schrodinger:original} in the presence of external stochastic forcing has been investigated through a vast literature. Research in this direction appeared as early as in the work of \cite{de1999stochastic} and was later explored in \cite{de2003stochastic} to account for both additive and multiplicative noise. Since then, there has been an extensive study on the existence and uniqueness of the solutions of \eqref{eqn:Schrodinger:original} under different assumptions on the regularity of the randomness, the nonlinearity as well as the domains' structures \cite{barbu2014stochastic,barbu2016stochastic,barbu2017stochastic,barbu2017stochasticLOG,
 brzezniak2025global,brzezniak2019martingale,  cui2019global,
de2010nonlinear,debussche20111d, ferrario2025stationary,
hornung2018nonlinear, roy2025large}. Despite the rich history on the well-posedness theory, the statistically steady states of \eqref{eqn:Schrodinger:original} and the convergence rate of \eqref{eqn:Schrodinger:original} seem to receive less attention. To mention a few examples, the existence of invariant probability measures was established in \cite{kim2006invariant} for the defocusing case ($\alpha=-1$). The results therein was later extended in \cite{ekren2017existence} to treat the focusing scenario. We note that due to the setting of $\rbb^d$ which precludes access to compactness in a natural way, the work of \cite{ekren2017existence} has to employ a generalized version of the Krylov-Bogoliubov procedure, so as to overcome the lack of Sobolev embedding. This is also found to be the main challenge in the related two-dimensional stochastic Euler equation \cite{bessaih2020invariant}. While the conditions for the existence of steady states in \cite{ekren2017existence} are compatible with those for the well-posedness in \cite{de2003stochastic}, the unique ergodicity of \eqref{eqn:Schrodinger:original} requires a strong damping effect as well as restrictions on the parameter $\sigma$ and the spatial dimension $d\le 3$ \cite{brzezniak2023ergodic}. Notably, it was previously shown in \cite{brzezniak2023ergodic} that there is only one invariant measure $\nu$ provided that $\lambda$ is sufficiently large and that
\begin{align*}
\begin{cases}
    0<\sigma<\infty,& d=1,2,\\
    0<\sigma\le\frac{1+\sqrt{17}}{4},&d=3.
\end{cases}
\end{align*}
Here, the upper bound $(1+\sqrt{17})/4$ in dimension $d=3$ is a technicality resulting from a series of delicate interpolations between Sobolev embeddings and the classical Strichartz inequalities. One of the contributions of the present article is the extension of this threshold. More specifically, we are able to improve the pathwise control on the nonlinearity established in \cite{brzezniak2023ergodic} by facilitating the damping effect and dispersive estimates on the Schr\"odinger semigroup, cf. Lemma \ref{lem:Strichartz:S(t)u}. In turn, this allows us to conclude the uniqueness of $\nu$ when
\begin{align*}
   \lambda\gg 1,\quad \text{and}\quad \begin{cases}
    0<\sigma<\infty,& d=1,2,\\
    0<\sigma<\frac{3}{2},&d=3.
\end{cases}
\end{align*}
We refer the reader to Theorem \ref{thm:unique-ergodicity} for the precise statement and to Section \ref{sec:poly-mixing:proof-unique-ergodicity} for its proof. See also related work for the stochastic Schr\"odinger equation posed in manifolds \cite{brzezniak2023invariant} or in the presence of multiplicative noise \cite{brzezniak2025global}. 

With regard to the problem of mixing rate, we remark that in the literature of SPDEs, geometric ergodicity is significantly more popular compared to the situation of sub-exponential rate. In particular, there are many dispersive dynamics whose invariant measures are exponentially attractive, e.g., the KdV equations with large damping \cite{glatt2021long}, the Ginzburg-Landau equations \cite{nersesyan2024exponential, odasso2006ergodicity}, as well as the wave equations \cite{cerrai2020convergence, martirosyan2014exponential,nguyen2023small}. The common feature of these settings and other parabolic systems is that they possess a strong dissipative structure thanks to the appearance of the Laplacian operator or the damping. See also related systems with a localized noise structure \cite{chen2025exponential,liu2024exponential}. In contrast, sub-exponential rates are more difficult to handle, owing to a weak mechanism for balancing energy. To mention several examples, we refer the reader to the Schr\"odinger equations in bounded domains \cite{debussche2005ergodicity,nguyen2024inviscid}, the Ginzburg-Landau equations perturbed by special noise at random times \cite{nersesyan2008polynomial}, the Navier-Stokes in $\rbb^2$ \cite{nersesyan2024polynomial}, the Kuramoto-Sivashinsky equations \cite{gao2024polynomial}, the wave equations \cite{gao2024polynomialWAVE,nguyen2024polynomial} and a stochastic conservation law \cite{dong2023ergodicity}. In literature, a Foias-Prodi typed estimate is typically employed to tackle the issue of mixing rates. Roughly speaking, starting from two initial conditions, this argument provides a means to 
control the high modes (in Fourier frequencies) of the two solutions in short distance from each other, assuming the low modes are synchronized. Particularly, the technique has proven to play a fundamental role in \cite{debussche2005ergodicity, ferrario2023uniqueness, glatt2021long, nguyen2024inviscid,odasso2006ergodicity} in the settings of bounded domains, allowing for the convenience of working with the discrete spectrum of the Laplacian, which are otherwise unavailable in $\rbb^d$ scenario. More recently, this challenge was addressed in \cite{nersesyan2024exponential, nersesyan2024polynomial} where a novel generalized Foias-Prodi agument is developed to circumvent the lack of compactness properties. See also related work of \cite{gao2024polynomialWAVE,gao2024polynomial}.  

Turning back to \eqref{eqn:Schrodinger:original}, we note that since the results from \cite{nersesyan2024exponential, nersesyan2024polynomial} rely on the local nature of the nonlinearity of the concerning dynamics, it is not clear how to adapt the technique therein to \eqref{eqn:Schrodinger:original}. Nevertheless, in the large damping regime, we resort to the coupling framework of \cite{debussche2005ergodicity,nguyen2024inviscid,odasso2006ergodicity} and successfully overcome the issue of the whole space by harnessing the Lyapunov functions and the Strichartz inequalities. More specifically, motivated by \cite{brzezniak2023ergodic, glatt2021long}, the proof of Theorem \ref{thm:pseudo:poly-mixing} is based on the following path-wise estimate
\begin{align} \label{ineq:pseudo:u_1-u_2}
    \|u_1(t)-u_2(t)\|^2_{L^2(\rbb^d)} \le C\exp\Big\{ t\Big(-\lambda  +\frac{1}{t} \int_0^t \|u_1(s)\|^{2\sigma}_{L^\infty(\rbb^d)}+\|u_2(s)\|^{2\sigma}_{L^\infty(\rbb^d)}\d s\Big)  \Big\}.
\end{align}
In the above, $u_1$ and $u_2$ are two solutions' trajectories corresponding to two distinct initial conditions. Regarding the right-hand side of \eqref{ineq:pseudo:u_1-u_2}, we can leverage Sobolev embeddings together with Strichartz estimates to the extent that the time-average of the $L^\infty$ norm can be approximated by suitable moments of the invariant measures. In turn, one may pick $\lambda$ sufficiently large so as to deduce the unique ergodicity, albeit without a mixing rate \cite{brzezniak2023ergodic}. See Section \ref{sec:poly-mixing:proof-unique-ergodicity} for a more detailed discussion of this point. To navigate the mixing difficulty, it is crucial to ensure that one can keep track of the growth of the $L^\infty$ norm over time. In our work, we tackle the problem using a two-fold strategy as follows: firstly, we exploit the damping effect from $\lambda>0$ to strengthen the Strichartz-typed inequalities established in \cite{brzezniak2023ergodic}, cf. Lemma \ref{lem:Strichart:|u|_infty} and Remark \ref{rem:S(t)F}, which provides us a control on the $L^\infty$ norm via Lyapunov functions and the noise trajectories, cf. Lemma \ref{lem:Strichart:int_0^t|u|_infty^2sigma}. We remark that in order to obtain such bound, we have to further restrict the range of $\sigma$, namely,
\begin{align*}
    \begin{cases}
    0<\sigma,& d=1,2,\\
    \frac{1}{6}<\sigma<\frac{3}{2},&d=3.
\end{cases}
\end{align*}
Then, under the assumption $\lambda$ is large enough, we facilitate the coupling technique from \cite{nguyen2024inviscid} while making use of the auxiliary Strichartz estimates to ultimately deduce an ergodic rate, cf. Theorem \ref{thm:poly-mixing}. In particular, this amounts to showing that the likeliness two solutions enter a ball and stay in close proximity is high, whereas the probability they become decoupled leaving the ball is small. We remark that in literature, geometric ergodicity usually relies on exponential moment bounds on the solutions \cite{nersesyan2024exponential,odasso2006ergodicity}, which can be obtained via the exponential Martingale inequality. However, such a result is currently not available in the settings of \eqref{eqn:Schrodinger:original}, for which one can only derive polynomial moment bounds, hence the algebraic convergence rate \eqref{ineq:poly-mixing:pseudo}. It is also worthwhile to mention that 
unlike related dynamics \cite{debussche2005ergodicity, nersesyan2024exponential, nguyen2024inviscid, odasso2006ergodicity} that typically invoke the Girsanov Theorem to ensure the validity of changing measures, our coupling strategy is able to circumnavigate the issue thanks to the large damping nature. This significantly simplifies the mixing argument while still achieving the same coupling effect. We refer the reader to Theorem \ref{thm:poly-mixing} for a precise statement of Theorem \ref{thm:pseudo:poly-mixing}, part 2, and to Section \ref{sec:poly-mixing} where we present the proof of Theorem \ref{thm:poly-mixing}.

The rest of the paper is organized as follows. In Section \ref{sec:main-result}, we introduce the functional settings and the assumptions needed for the analysis. We also state our main results through Theorem \ref{thm:unique-ergodicity} and Theorem \ref{thm:poly-mixing} concerning respectively the unique ergodicity and polynomial mixing of \eqref{eqn:Schrodinger:original}. In Section \ref{sec:moment-bound}, we collect useful energy estimates on the solutions of \eqref{eqn:Schrodinger:original}, including their moment bounds as well as irreducibility conditions. In Section \ref{sec:poly-mixing}, we discuss the pathwise argument and the coupling strategy while making use of the estimates from Section \ref{sec:moment-bound} to establish Theorem \ref{thm:unique-ergodicity} and Theorem \ref{thm:poly-mixing}. The paper concludes with Appendices \ref{sec:Strichartz} and \ref{sec:aux-results}. Particularly, we derive specific Strichartz inequalities in Appendix \ref{sec:Strichartz} whereas we supply several auxiliary results in Appendix \ref{sec:aux-results}, both of which are employed to prove the main theorems.

\section{Assumptions and main results} \label{sec:main-result}

\subsection{The functional settings}\label{sec:main-result:setting}

For $p\ge 1$, let $L^p:=L^p(\rbb^d;\cbb)$ denote the usual Lebesgue space of complex-valued functions on $\rbb^d$. In particular, when $p=2$, we denote by $H=L^2$ the Hilbert space endowed with the inner product
\begin{align*}
\la u,v\ra_H =  \Re \int_{\rbb^d} u(x) \bar{v}(x)\d x,
\end{align*}
and the norm $\|u\|^2_H = \la u,u\ra_H$. Moreover, for $s\in \rbb$, let $H^s:=H^s(\rbb^d;\cbb)$ be the Sobolev space with the norm
\begin{align*}
\|u\|^2_{H^s}=\int_{\rbb^d}(1+|\xi|^2)^s |\mathscr{F} u(\xi)|^2\d \xi,
\end{align*}
where $\mathscr{F} u$ denotes the Fourier transform of $u$. In general, for $p\in[1,\infty]$, we define $H^{s,p}$ with the norm
\begin{align*}
    \|u\|_{H^{s,p}} = \big\|\mathscr{F}^{-1}\big[(1+|\xi|^2)^{\frac{s}{2}}\mathscr{F} u\big]  \big\|_{L^p}.
\end{align*}
We recall the following Sobolev embeddings that will be useful throughout the analysis
\begin{align*}
    H^{s,p}\subset L^{\frac{pd}{d-sp} }, \quad 1\le p <\infty,\, 0<s<d/p,
\end{align*}
\begin{align*}
    H^{s,p}\subset L^\infty ,\quad p>1,\, s>d/p,
\end{align*}
\begin{align*}
 \textup{and}\quad   H^{s,p}\subset H^{s_1,p_1},\quad s-\frac{d}{p}=s_1-\frac{d}{p_1},\, 1< p\le p_1<\infty,\, s,s_1\in\rbb.
\end{align*}
See \cite[Section 1.4]{cazenave2003semilinear}, \cite[Theorem 6.5.1]{bergh2012interpolation} and \cite{triebel1978interpolation}. In particular, the last embedding will play a crucial role in the Strichartz estimates collected in Lemma \ref{lem:Strichart:|u|_infty}.

With regard to the noise structure $Q\d W(t)$, following \cite{brzezniak2023ergodic,brzezniak2023invariant}, let $U$ be a real Hilbert space with an orthonormal basis $\{e_j\}_{j\ge1}$. Then, $W$ is a cylindrical Wiener process defined on $U$ with the representation
\begin{align*}
W(t)=\sum_{k\ge 1}e_k B_k(t),
\end{align*}
where $\{B_k\}_{k\ge 1}$ is a sequence of i.i.d. standard Brownian motions defined on the same stochastic basis $\{\Omega,\Fcal,\Fcal_t,\P\}$ satisfying the usual conditions \cite{karatzas2012brownian}. Concerning the covariance operator $Q$, we will make the following assumption.

\begin{assumption} \label{cond:Q:well-posed}
The map $Q:U\to H^1$ is a Hilbert-Schmidt operator, i.e.,
\begin{align} \label{cond:Q:H1}
\|Q\|^2_{L_{\textup{HS}}(U,H^1)}=  \sum_{k\ge 1}\|Qe_k\|^2_{H^1}<\infty.
\end{align}

\end{assumption}

Turning to the nonlinearity of \eqref{eqn:Schrodinger:original}, we denote
\begin{align*}
    F_\alpha(u) = -\i\alpha|u|^{2\sigma}u.
\end{align*}
Following \cite{brzezniak2023ergodic,de2003stochastic}, we will impose different conditions on the parameter $\sigma$, depending on the values of $\alpha$ and dimension $d$.

\begin{assumption} \label{cond:sigma} 1. When $\alpha=1$ (focusing), $\sigma$ satisfies
\begin{align*} 
0< \sigma <\frac{2}{d},\quad d\le 3.
\end{align*}

2. When $\alpha=-1$ (defocusing), $\sigma$ satisfies
\begin{align*} 
\begin{cases} 
0<\sigma,& d=1,2,\\
0< \sigma< 2 ,& d=3.
\end{cases}
\end{align*}

\end{assumption}

Having introduced assumptions on the noise and the nonlinearity, in what follows, we briefly review the well-posedness of \eqref{eqn:Schrodinger:original}. Let $S(t)=e^{-it\triangle}$, $t\in\rbb$ denote the semigroup associated with the equation
\begin{align*}
    \i\frac{\d}{\d t}u- \triangle u =0.
\end{align*}
It is a classical result that equation \eqref{eqn:Schrodinger:original} admits a unique global mild solution for all initial conditions $u_0\in H^1$. Specifically, the well-posedness of \eqref{eqn:Schrodinger:original} is summarized in the following proposition.
\begin{proposition}{\cite[Theorem 3.4]{de2003stochastic}} \label{prop:well-posed}
Under Assumptions \ref{cond:Q:well-posed} and \ref{cond:sigma}, for every initial condition $u_0\in H^1$, there exists a unique global solution $u(t)=u(t;u_0)$ given by the mild formula
 \begin{align*}
        u(t)=e^{-\lambda t}S(t)u_0 + \int_0^t e^{-\lambda(t-s)}S(t-s)F_\alpha(u(s))\d s+\int_0^t e^{-\lambda(t-s)}S(t-s)Q\d W(s),\quad t\ge 0.
    \end{align*}
Furthermore, the solution $u(t;u_0)$
is continuous in $H^1$ with respect to $u_0$, for all fixed $t\ge 0$, i.e.,
\begin{align*}
    \E \|u(t;u^n_0) -u(t;u_0)\|_{H^1}\to 0,
\end{align*}
whenever $\|u_0^n-u_0\|_{H^1}\to 0$ as $n\to\infty$.
\end{proposition}
We refer the reader to \cite[Section 3]{de2003stochastic} for a further discussion of the above well-posedness statement. See also \cite[Theorem 2.5]{brzezniak2023ergodic}.

\subsection{Unique ergodicity} \label{sub:main-result:unique-ergodicity}

As a consequence of Proposition \ref{prop:well-posed}, we may define the Markov transition probabilities corresponding to the process $u(t;u_0)$ satisfying \eqref{eqn:Schrodinger:original} given by 
\begin{align*}
P_t(u_0,A):=\P(u(t;u_0)\in A),
\end{align*}
for each $t\geq 0$, $u_0\in H^1$, and Borel sets $A\subseteq H^1$. We let $\B_b(H^1)$ denote the set of bounded Borel measurable functions defined on $H^1$. The Markov semigroup associated to \eqref{eqn:Schrodinger:original} is the operator $P_t:\B_b(H^1)\to\B_b(H^1)$ defined by
\begin{align*}
P_t f(u_0)=\E[f(u(t;u_0))], \quad f\in \B_b(H^1).
\end{align*}
Recall that a probability measure $\nu\in \Pcal r(H^1)$ is said to be \emph{invariant} for the semigroup $P_t$ if for every $f\in \B_b(H^1)$
\begin{align*}
\int_{H^1} f(u_0) (P_t)^*\nu(\d u_0)=\int_{H^1} f(u_0)\nu(\d u_0),
\end{align*}
where $(P_t^{\mu})^*\nu$ denotes the measure obtained from $\nu$ by the action of $P_t^{\mu}$, i.e., 
$$\int_{H^1}f(u)(P_t)^*\nu(\d u)=\int_{H^1}P_t f(u)\nu(\d u).$$
We record the following existence result of an invariant probability measure of $P_t$ and refer the reader to \cite[Theorem 3.4]{ekren2017existence} for its proof.
\begin{proposition}{\cite[Theorem 3.4]{ekren2017existence}} \label{prop:unique-ergodicity}
  Under Assumptions \ref{cond:Q:well-posed} and \ref{cond:sigma}, $P_t$ admits at least one invariant probability measure $\nu$.

\end{proposition}

We note that the conditions for the existence of $\nu$ is compatible with those for the well-posedness results from Proposition \ref{prop:well-posed}. We now state the first main result of the paper below through Theorem \ref{thm:unique-ergodicity} establishing the uniqueness of $\nu$ under slightly stronger assumptions on $\sigma$.

\begin{theorem} \label{thm:unique-ergodicity}
Suppose that 
 
 1. in dimension $d=1,2$, Assumptions \ref{cond:Q:well-posed} and \ref{cond:sigma} hold;

 2. in dimension $d=3$, Assumptions \ref{cond:Q:well-posed} and \ref{cond:sigma} and the following extra condition hold
 \begin{align} \label{cond:sigma:unique-ergodicity}
     0<\sigma<\frac{3}{2}.
 \end{align} 
Then, for all $\lambda$ sufficiently large, $\nu$ is unique.
\end{theorem}

 The proof of Theorem \ref{thm:unique-ergodicity} follows closely the approach of \cite{brzezniak2023ergodic, glatt2021long} dealing with the same issue in dispersive dynamics with large damping. In particular, the argument for the uniqueness of $\nu$ relies on pathwise estimates on the difference between two solutions starting from two distinct initial conditions. We demonstrate that their distance  can be approximated by the regularity of invariant measures, which have uniform moment bounds independent of $\lambda$. In turn, taking $\lambda$ sufficiently large allows us to conclude that two solutions must eventually synchronize, thereby establishing the uniqueness of $\nu$. All of this will be explained in detail in Section \ref{sec:poly-mixing:proof-unique-ergodicity} where we provide the proof of Theorem \ref{thm:unique-ergodicity}.  

\begin{remark} \label{rem:unique-ergodicity} We remark that while the result of Theorem \ref{thm:unique-ergodicity} in dimension $d=1,2$ is already established in \cite{brzezniak2023ergodic}, it extends the work of \cite{brzezniak2023ergodic} in dimension $d=3$. More specifically, the unique ergodicity of \cite[Theorem 5.1]{brzezniak2023ergodic} requires that \begin{align*}
    0<\sigma\le \frac{1+\sqrt{17}}{4}, \quad d=3,
\end{align*}
    whereas in Theorem \ref{thm:unique-ergodicity}, $\sigma\in (0,3/2)$. Although this condition does not include every $\sigma\in(0,2)$ when $\alpha=-1$, cf. Assumption \ref{cond:sigma}, it is optimal in the sense that it is also required in Lemma \ref{lem:F_alpha} concerning the existence of some $p\ge1$ such that
    \begin{align*}
        \|F_\alpha(u)\|_{H^{1,p}}\lesssim \|u\|_{H^1}^{1+2\sigma},\quad u\in H^1.
    \end{align*}
    which is one of the main ingredients used to establish Theorem \ref{thm:unique-ergodicity} in Section \ref{sec:poly-mixing:proof-unique-ergodicity}.
\end{remark}

\subsection{Polynomial mixing} \label{sec:main-result:poly-mixing}
We now turn to the second main topic of the paper on the mixing property of \eqref{eqn:Schrodinger:original}. 
In order to measure the convergence rate, we will work with suitable distances and Lyapunov estimates in $H^1$. Specifically, let $d_1:H^1\times H^1\to[0,1]$ denote the distance on $H^1$ defined as
\begin{align} \label{form:d_1}
    d_1(u,v) = \|u-v\|_H\wedge 1,\quad u,\,v \in H^1.
\end{align}
A function $f:H^1\to\cbb$ is called Lipschitz with respect to $d_1$ if
\begin{align*}
    [f]_{\textup{Lip},d_1} := \sup_{u\neq v}\frac{|f(u)-f(v)|}{d_1(u,v)}<\infty.
\end{align*}
Turning to Lyapunov functions, it is well-known that 
\begin{align*}
    \varphi_\alpha(u)=\|\grad u\|^2_H-\frac{\alpha}{1+\sigma}\|u\|^{2+2\sigma}_{L^{2+2\sigma}}
\end{align*}
is invariant for the deterministic equation
\begin{align*}
    \frac{\d}{\d t}u+\i\triangle u+ \i F_\alpha(u)=0,
\end{align*}
in the sense that $$\frac{\d}{\d t}\varphi_\alpha(u)=0.$$ 
Since $\varphi_\alpha$ is not sign-definite, we will have to interpolate between $L^2$, $L^{2+2\sigma}$ and $H^1$ to guarantee the positivity of $\varphi_\alpha$. To this end, we recall that given $\sigma  d< 2(\sigma+1)$, the following Gagliardo-Nirenberg inequality holds
\begin{align*}
\|u\|_{L^{2+2\sigma}}\le C \|u\|_H^{1-\frac{\sigma d}{2+2\sigma}}\|\grad u \|_H^{\frac{\sigma d}{2+2\sigma}}.
\end{align*}
In view of Assumptions \ref{cond:sigma}, the above inequality is valid for the focusing case ($\alpha=1$) when $d=1,2,3$. Moreover, if $\sigma d<2$, using Young's inequality, we can bound $L^{2+2\sigma}$ norm as follows:
\begin{align} \label{ineq:Gagliardo}
100\|u\|_{L^{2+2\sigma}}^{2+2\sigma} \le C \|u\|_H^{2(1+\sigma)-\sigma d}\|\grad u \|_H^{\sigma d}\le  \frac{1}{2}\|\grad u\|^2_H + \frac{1}{2}\kappa\|u\|_H^{2+\frac{4\sigma}{2-\sigma d}}.
\end{align}
So, we introduce the functional $\Phi_\alpha:H^1\to \rbb$ defined as
\begin{align}\label{form:Phi}
\Phi_\alpha(u) =\begin{cases}
\|\grad u\|^2_H+\frac{1}{1+\sigma}\|u\|^{2+2\sigma}_{L^{2+2\sigma}}+\|u\|^2_H,& \alpha=-1,\\
\noalign{\vskip 7pt} \|\grad u\|^2_H-\frac{1}{1+\sigma}\|u\|^{2+2\sigma}_{L^{2+2\sigma}}+ \kappa\|u\|_H^{2\sigma_d },& \alpha=1,\,\sigma_d = 1+\frac{2\sigma}{2-\sigma d}.\\
\end{cases}
\end{align}
In particular, in the focusing case $\alpha=1$, the choice of $\Phi_\alpha$ satisfies the estimates
\begin{align} \label{ineq:Phi(u)}
c_\Phi\big(\|\grad u\|^2_H+\|u\|^{2+2\sigma}_{L^{2+2\sigma}}+\|u\|_H^{2\sigma_d} \big)\le \Phi_\alpha(u)\le C_\Phi\big(\|\grad u\|^2_H+\|u\|_H^{2\sigma_d}\big).
\end{align}

Concerning the noise, we require that $QW$ possess higher regularity compared with the condition imposed in Assumption \ref{cond:Q:well-posed}.

\begin{assumption} \label{cond:Q:poly-mixing}
The map $Q:U\to H^3$ is a Hilbert-Schmidt operator, i.e.,
\begin{align} \label{cond:Q:H3:poly-mixing}
\|Q\|^2_{L_{\textup{HS}}(U,H^3)}=  \sum_{k\ge 1}\|Qe_k\|^2_{H^3}<\infty.
\end{align}

\end{assumption}

We note that condition \eqref{cond:Q:H3:poly-mixing} is equivalent to 
\begin{align*}
 \sum_{k\ge 1}\|Qe_k\|^2_{H}+  \sum_{k\ge 1}\|\grad \triangle Qe_k\|^2_{H}<\infty . 
\end{align*}
In particular, this means that the noise term $QW(t)$ has $H^3$ sample path. It is important to point out that unique ergodicity of \eqref{eqn:Schrodinger:original} only requires $Q$ be a Hilbert-Schmidt operator in $H^1$ as stated in Theorem \ref{thm:unique-ergodicity}. Nevertheless, in order to extract the convergence rate, we will have to assume that the noise has higher regularity than $H^1$. More specifically, Assumption \ref{cond:Q:poly-mixing} is needed later for the proof of the irreducibility stated in Lemma \ref{lem:irreducibility}, which will be employed to establish the mixing result.

We now state the main result of the paper through Theorem \ref{thm:poly-mixing} below establishing the algebraic convergence rate of $P_t$ toward equilibrium.

\begin{theorem} \label{thm:poly-mixing}
Suppose that 
 
1. in dimension $d=1,2$, Assumptions \ref{cond:sigma} and \ref{cond:Q:poly-mixing} hold;

2. in dimension $d=3$, Assumptions \ref{cond:sigma} and \ref{cond:Q:poly-mixing} and the following extra condition hold
    \begin{align} \label{cond:sigma:poly-mixing}
        \frac{1}{6}<\sigma<\frac{3}{2}.
        \end{align}
    Then, for all $\lambda>0$ sufficiently large, $q>0$, $u_0^1$, $u_0^2\in H^1$, and functions $f$ that are Lipschitz with respect to distance $d_1$ defined in \eqref{form:d_1}, the following holds
    \begin{align} \label{ineq:poly-mixing}
        \big|  \E f(u(t;u_0^1)) -\E f(u(t;u_0^2)) \big|\le \frac{C_q}{(1+t)^q}\big[1+ \Phi_\alpha(u_0^1)+\Phi_\alpha(u_0^2)  \big],\quad t\ge 0,
    \end{align}
for some positive constant $C_q$ independent of $t$, $u_0^1$ and $u_0^2$. In the above, $\Phi_\alpha$ is the function defined in \eqref{form:Phi}.
\end{theorem}

\begin{remark} \label{rem:sigma:defocusing}
1. We note that while condition $\sigma>0$, $d=1,2$, is optimal for the mixing property \eqref{ineq:poly-mixing}, the restriction $\sigma\in(1/6,3/2)$ when $d=3$ is perhaps far from optimality. In particular, the mixing rate issue for $\sigma\in(0,1/6]$ when $d=3$, $\alpha\in\{-1,1\}$ remains an open problem.

2. We also would like to mention that Theorem \ref{thm:poly-mixing} does not include the linear instance $\sigma=0$. This is thanks to the fact that when $\sigma=0$, the convergence rate is simply exponentially fast regardless of the value of $\lambda>0$. Indeed, this can be verified by a routine calculation
\begin{align*}
    \frac{\d}{\d t} \| u(t;u_0^1)-u(t;u_0^2) \|^2_H +2\lambda\| u(t;u_0^1)-u(t;u_0^2) \|^2_H= 0,
\end{align*}
which implies that
\begin{align*}
    \E \| u(t;u_0^1)-u(t;u_0^2) \|^2_H = e^{-2\lambda t} \| u_0^1-u_0^2 \|^2_H,\quad t\ge 0.
\end{align*}
\end{remark}

The proof of Theorem \ref{thm:poly-mixing} makes use of the coupling framework developed in \cite{debussche2005ergodicity,odasso2006ergodicity} tailored to our settings of unbounded domains. The strategy consists of three crucial ingredients: the high probability the solutions become coupled in a ball, the likeliness they stay in close proximity of one another, and lastly, the small probability of decoupling. Notably, the arguments of \cite{debussche2005ergodicity,odasso2006ergodicity,nguyen2024inviscid} rely on a stochastic Foias-Prodi estimate controlling the high frequencies by the low modes, so as to achieve the coupling effect. Since we are dealing with unbounded domains, which is not convenient dealing with Fourier decompositions, the main difference of our approach from these works is the Strichartz estimates, allowing for controlling the distance between two solutions. All of this will be explained in detail in Section \ref{sec:poly-mixing:proof-main-result} where the proof of Theorem \ref{thm:poly-mixing} is supplied.

\section{Moment estimates} \label{sec:moment-bound}

Throughout the rest of the paper, $c$ and $C$ denote generic positive constants that may change from line to line. The main parameters that they depend on will appear between parenthesis, e.g., $c(T,q)$ is a function of $T$ and $q$.

In this section, we collect useful estimates on the solutions of \eqref{eqn:Schrodinger:original} that will be employed later in Section \ref{sec:poly-mixing} to establish the main theorems. We start with two results below in Lemma \ref{lem:moment-bound:H} and Lemma \ref{lem:moment-bound:H1} respectively giving moment bounds in $H$ and $H^1$ norms.

\begin{lemma} \label{lem:moment-bound:H}
Under Assumptions \ref{cond:Q:well-posed} and \ref{cond:sigma}, for all $u_0\in L^2(\Omega;H^1)$ and $n\ge 1$, the following holds:
\begin{align} \label{ineq:moment-bound:H}
\E \|u(t)\|^{2n}_{H} \le e^{-n \lambda t}\E\|u_0\|^{2n}_H + C_{0,n}\frac{\|Q\|^{2n}_{L_{\HS}(U;H)}}{\lambda^{n}},\quad t\ge 0,
\end{align}
for some positive constants $C_{0,n}$ independent of $u_0$, $t$, $\lambda$ and $Q$.

\end{lemma}
\begin{proof}
Firstly, we apply It\^o's formula to $\|u\|^2_H$ and obtain
\begin{align*}
\d \|u\|^2_H & = -2\lambda\|u\|^2_H \d t +\sum_{k\ge 1}\|Qe_k\|^2_H\d t +2\Re \big( \la u,Q\d W\ra_H\big).
\end{align*}
Likewise, for $n>1$, we have
\begin{align} \label{eqn:d|u|^2n_H}
\d \|u\|^{2n}_H&= n\|u\|^{2(n-1)}\d \|u\|^2_H+\frac{1}{2}n(n-1)\|u\|^{2(n-2)}_H\la \d \|u\|^2_H,\d \|u\|^2_H\ra \notag \\
& = -2n\lambda\|u\|^{2n}_H\d t + n\|u\|^{2(n-1)}_H\|Q\|_{L_{\HS}(U;H)}^2\d t + 2n\|u\|^{2(n-1)}_H\Re\big(\la u,Q\d W\ra_H\big) \notag \\
&\qquad +2n(n-1)\|u\|^{2(n-2)}_H\sum_{k\ge 1}\big|\Re\big(\la u,Qe_k\ra_H\big) \big|^2\d t.
\end{align}
Using Holder's and Young's inequalities, we infer
\begin{align*}
\sum_{k\ge 1}\big|\Re\big(\la u,Qe_k\ra_H\big) \big|^2 \le \|u\|^2_H \sum_{k\ge 1}\| Qe_k\|_H^2 = \|u\|_H^2\|Q\|^2_{L_{\HS}(U;H)},
\end{align*}
and that
\begin{align*}
&n\|u\|^{2(n-1)}_H\|Q\|_{L_{\HS}(U,H)}^2+ 2n(n-1)\|u\|^{2(n-2)}_H\sum_{k\ge 1}\big|\Re\big(\la u,Qe_k\ra_H\big) \big|^2\\
&\qquad\le \frac{1}{100}n\lambda\|u\|_H^{2n}+ c\frac{\|Q\|^{2n}_{L_{\HS}(U;H)}}{\lambda^{n-1}}.
\end{align*}
In the last estimate, $c=c(n)$ is a positive constant independent of $\lambda$. It follows from \eqref{eqn:d|u|^2n_H} that
\begin{align} \label{ineq:d|u|^2n_H}
\d \|u\|^{2n}_H \le  -n\lambda \|u\|^{2n}_H\d t+ c\frac{\|Q\|^{2n}_{L_{\HS}(U;H)}}{\lambda^{n-1}}\d t + 2n\|u\|^{2(n-1)}_H\Re\big(\la u,Q\d W\ra_H\big).
\end{align}
As a consequence, we take expectation on both sides of \eqref{ineq:d|u|^2n_H} and obtain
\begin{align*}
\frac{\d}{\d t}\E\|u(t)\|^{2n}_H \le -n\lambda \E\|u(t)\|^{2n}_H+ c\frac{\|Q\|^{2n}_{L_{\HS}(U;H)}}{\lambda^{n-1}}.
\end{align*}
This produces \eqref{ineq:moment-bound:H}, as claimed.

\end{proof}

\begin{lemma} \label{lem:moment-bound:H1}

Under Assumptions \ref{cond:Q:well-posed} and \ref{cond:sigma}, for all $u_0\in L^2(\Omega;H^1)$ and $n\ge 1$, there exist positive constants $c_{1,n}$ and $C_{1,n}$ independent of $u_0$, $t$, $\lambda$ and $Q$ such that the followings hold:

1. Defocusing case $\alpha=-1$:
\begin{align} \label{ineq:moment-bound:E.Phi(u(t))^n<e^(-lambda.t)E.Phi(u_0)+C:defocusing}
\E \Phi _{-1} (u(t))^{n} \le e^{-c_{1,n} \lambda t}\E \Phi_{-1}(u_0)^{n} + \frac{C_{1,n}}{\lambda^{n}}\|Q\|^{2n}_{L_{\HS}(U;H^1)} + \frac{C_{1,n}}{\lambda^{ n(1+\sigma)}}\|Q\|^{2n(1+\sigma)}_{L_{\HS}(U;H^1)},\quad t\ge 0,
\end{align}
where $\Phi_{-1}$ is the functional defined in \eqref{form:Phi} with $\alpha=-1$.

2. Focusing case $\alpha=1$:
\begin{align} \label{ineq:moment-bound:E.Phi(u(t))^n<e^(-lambda.t)E.Phi(u_0)+C:focusing}
\E \Phi_1(u(t))^{n} &\le e^{-c_{1,n} \lambda t}\E \Phi_1(u_0)^{n} + \frac{C_{1,n}}{\lambda^{n}}\|Q\|^{2n}_{L_{\HS}(U;H^1)} + \frac{C_{1,n}}{\lambda^{ n(1+\sigma)}}\|Q\|^{2n(1+\sigma)}_{L_{\HS}(U;H^1)}\notag \\
&\qquad+\frac{C_{1,n}}{\lambda^{n\sigma_d}}\|Q\|^{2n\sigma_d}_{L_{\HS}(U;H)},\quad t\ge 0.
\end{align}
In the above, $\Phi_1$ and $\sigma_d$ are defined in \eqref{form:Phi} with $\alpha=1$.

\end{lemma}

\begin{proof}
We first consider the defocusing case $\alpha=-1$, as it is relatively simpler than the focusing case. 

By It\^o's formula, we have
\begin{align} \label{eqn:d.Phi(u):defocusing}
\d \Phi_{-1}(u)& =  -2\lambda \big[\|\grad u\|^2_{H}+\|u\|^{2+2\sigma}_{L^{2+2\sigma}}+\|u\|^2_H \big]\d t +\sum_{k\ge 1}\|\grad Qe_k\|^2_H\d t+\sum_{k\ge 1}\|Qe_k\|^2_H\d t+ \d M_{-1} \notag  \\
&\qquad + \sum_{k\ge 1}\la |u|^{2\sigma},|Qe_k|^2\ra_H\d t + 2\sigma\sum_{k\ge 1}\big\la |u|^{2\sigma-2},\big|\Re\big(u \overline{Qe_k}\big)\big|^2\big\ra_H\d t,
\end{align}
where the semi Martingale process $M_{-1}$ (corresponding to $\alpha=-1$) is given by
\begin{align} \label{form:dM_(-1)}
\d M_{-1} = 2\Re\big( \la \grad u,\grad Q\d W\ra_H\big)\d t + 2\Re\big(\la |u|^{2\sigma}u,Q\d W\ra_H\big)\d t+2\Re \big( \la u,Q\d W\ra_H\big)\d t.
\end{align}
Using Holder's and Young's inequalities and the embedding $H^1\subset L^{2+2\sigma}$, we estimate
\begin{align*}
&\sum_{k\ge 1} \la |u|^{2\sigma},|Qe_k|^2\ra_H + 2\sigma\sum_{k\ge 1} \big\la |u|^{2\sigma-2},\big|\Re\big(u \overline{Qe_k}\big)\big|^2\big\ra_H \\
&\le   (1+2\sigma) \sum_{k\ge 1}  \la |u|^{2\sigma},|Qe_k|^2\ra_H\\
& \le  (1+2\sigma) \|u\|^{2\sigma}_{L^{2+2\sigma}}  \sum_{k\ge 1}  \|Qe_k\|^{2}_{L^{2+2\sigma}}.\\
&\le c \|u\|^{2\sigma}_{L^{2+2\sigma}} \|Q\|_{L_{\HS}(U;H^1)}^2,
\end{align*}
whence by Young's inequality
\begin{align} \label{ineq:<u^(2sigma),Qe_k>}
&\sum_{k\ge 1} \la |u|^{2\sigma},|Qe_k|^2\ra_H + 2\sigma\sum_{k\ge 1} \big\la |u|^{2\sigma-2},\big|\Re\big(u \overline{Qe_k}\big)\big|^2\big\ra_H \notag \\
&\le \frac{1}{100}\lambda\|u\|^{2+2\sigma}_{L^{2+2\sigma}} + \frac{c}{\lambda^\sigma}\|Q\|_{L_{\HS}(U;H^1)}^{2+2\sigma}.
\end{align}
In the above, $c$ is a positive constant independent of $\lambda$, $u$ and $Q$. It follows that
\begin{align} \label{ineq:dPhi(u):defocusing}
\d \Phi_{-1}(u) 
& \le  -2\lambda \|\grad u\|^2_{H}\d t-\lambda\|u\|^{2+2\sigma}_{L^{2+2\sigma}}\d t -2\lambda \| u\|^2_{H}\d t+ \|Q\|^2_{L_{\HS}(U;H^1)}\d t \notag \\
&\qquad+\frac{c}{\lambda^\sigma}\|Q\|_{L_{\HS}(U;H^1)}^{2+2\sigma}\d t+ \d M_{-1} \notag \\
&\le -\lambda \Phi_{-1}(u) \d t + \|Q\|^2_{L_{\HS}(U;H^1)}\d t+\frac{c}{\lambda^\sigma}\|Q\|_{L_{\HS}(U;H^1)}^{2+2\sigma}\d t+ \d M_{-1}.
\end{align}

Next, for every $n> 1$, we compute
\begin{align*}
\d \Phi_{-1}(u)^n & = n \Phi_{-1}(u)^{n-1}\d \Phi_{-1}(u) + \frac{1}{2}n(n-1)\Phi_{-1}(u)^{n-2} \la \d \Phi_{-1}(u),\d\Phi_{-1}(u)\ra\\
&= n \Phi_{-1}(u)^{n-1}\d \Phi(u) + \frac{1}{2}n(n-1)\Phi_{-1}(u)^{n-2} \d\la M_{-1}\ra,
\end{align*}
where $\la M_{-1}\ra(t)$ is the quaratic variation process associated with $M_{-1}(t)$ defined in \eqref{form:dM_(-1)}. In particular, we have
\begin{align*}
\d \la M_{-1}\ra& =4 \sum_{k\ge 1}\big|\Re\big(\la \grad u, \grad Qe_k\ra_H+\la |u|^{2\sigma}u,Qe_k\ra_H+\la u,Qe_k\ra_H\big)\big|^2 \d t\\
&\le 8 \sum_{k\ge 1} \Big(\|\grad u\|^2_H\|\grad Qe_k\|^2_H+\|u\|^{4\sigma +2}_{L^{2+2\sigma}}\|Qe_k\|^{2}_{L^{2+2\sigma}} +\|u\|^2_H\|Qe_k\|^2_H \Big)\d t\\
&\le c\big( \|\grad u\|^2_H+\|u\|^{4\sigma +2}_{L^{2+2\sigma}}+\|u\|^2_H \big)\|Q\|^2_{L_{\HS}(U;H^1)}\d t.
\end{align*}
From expression \eqref{form:Phi} in the defocusing case $\alpha=-1$, we infer
\begin{align} \label{ineq:d<M_(-1)>}
\d \la M_{-1}\ra &\le c \big(\Phi_{-1}(u)+\Phi_{-1}(u)^{2-\frac{1}{1+\sigma}} \big)\|Q\|^2_{L_{\HS}(U;H^1)}\d t.
\end{align}
This together with \eqref{ineq:dPhi(u):defocusing} produces the estimate
\begin{align*}
\d \Phi_{-1}(u)^n & \le  n\Phi_{-1}(u)^{n-1} \Big(-\lambda \Phi_{-1}(u) \d t + \|Q\|^2_{L_{\HS}(U;H^1)}\d t+\frac{c}{\lambda^\sigma}\|Q\|_{L_{\HS}(U;H^1)}^{2+2\sigma}\d t+ \d M_{-1}\Big)\\
&\qquad + c\,n(n-1)\big(\Phi_{-1}(u)^{n-1}+\Phi_{-1}(u)^{n-\frac{1}{1+\sigma}}\big)\|Q\|^2_{L_{\HS}(U;H^1)}\d t.
\end{align*}
We repeatedly employ Young's inequality to infer
\begin{align*}
&\big[n + c\,n(n-1)\big]\Phi_{-1}(u)^{n-1}\|Q\|^2_{L_{\HS}(U;H^1)}  + cn\Phi_{-1}(u)^{n-1}\cdot\frac{1}{\lambda^\sigma}\|Q\|^{2+2\sigma}_{L_{\HS}(U;H^1)}\\
&\qquad\qquad +c\,n(n-1)\Phi_{-1}(u)^{n-\frac{1}{1+\sigma}}\|Q\|^2_{L_{\HS}(U;H^1)}\\
 & \le \frac{1}{100}\lambda\Phi_{-1}(u)^n + \frac{c}{\lambda^{n-1}}\|Q\|^{2n}_{L_{\HS}(U;H^1)}+ \frac{c}{\lambda^{ n\sigma+n-1}}\|Q\|^{2n(1+\sigma)}_{L_{\HS}(U;H^1)}.
\end{align*}
It follows that
\begin{align} \label{ineq:d.Phi(u)^n:defocusing}
\d \Phi_{-1}(u)^n & \le -\frac{1}{2}n\lambda \Phi_{-1}(u)^n\d t+ n\Phi_{-1}(u)^n \d M_{-1}(t) +\frac{c}{\lambda^{n-1}}\|Q\|^{2n}_{L_{\HS}(U;H^1)}\d t \notag \\
&\qquad+ \frac{c}{\lambda^{ n\sigma+n-1}}\|Q\|^{2n(1+\sigma)}_{L_{\HS}(U;H^1)}\d t .
\end{align}
We take expectation on both sides to deduce the bound
\begin{align*}
\frac{\d}{\d t}\E\, \Phi_{-1}(u(t))^n \le -\frac{1}{2}n\lambda \E\,\Phi_{-1}(u)^n+\frac{c}{\lambda^{n-1}}\|Q\|^{2n}_{L_{\HS}(U;H^1)} + \frac{c}{\lambda^{ n\sigma+n-1}}\|Q\|^{2n(1+\sigma)}_{L_{\HS}(U;H^1)},
\end{align*}
whence
\begin{align*}
\E\, \Phi_{-1}(u(t))^n  \le e^{-\frac{1}{2}n\lambda t}\E\, \Phi_{-1}(u_0)^n + \frac{c}{\lambda^{n}}\|Q\|^{2n}_{L_{\HS}(U;H^1)} + \frac{c}{\lambda^{ n(1+\sigma)}}\|Q\|^{2n(1+\sigma)}_{L_{\HS}(U;H^1)}.
\end{align*}
This produces \eqref{ineq:moment-bound:E.Phi(u(t))^n<e^(-lambda.t)E.Phi(u_0)+C:defocusing} for the defocusing case $\alpha=-1$.

Turning to the focusing case $\alpha=1$, similar to \eqref{eqn:d.Phi(u):defocusing}, it holds that
\begin{align*}
\d \Phi_1(u)& =  -2\lambda \big[\|\grad u\|^2_{H}-\|u\|^{2+2\sigma}_{L^{2+2\sigma}}+\kappa\sigma_d\|u\|^{2\sigma_d}_H \big]\d t +\sum_{k\ge 1}\|\grad Qe_k\|^2_H\d t  \\
&\qquad + \sum_{k\ge 1}\la |u|^{2\sigma},|Qe_k|^2\ra_H\d t + 2\sigma\sum_{k\ge 1}\big\la |u|^{2\sigma-2},\big|\Re\big(u \overline{Qe_k}\big)\big|^2\big\ra_H\d t\\
&\qquad +2\sigma_d(\sigma_d-1)\kappa\|u\|^{2(\sigma_d-2)}_H\sum_{k\ge 1}\big|\Re\big(\la u,Qe_k\ra_H\big) \big|^2\d t\\
&\qquad+ \kappa\sigma_d\|u\|^{2(\sigma_d-1)}_H\|Q\|_{L_{\HS}(U;H)}^2\d t  +\d M_1.
\end{align*}
where the semi Martingale $M_1$ (corresponding to $\alpha=1$) is given by
\begin{align} \label{form:M_1}
\d M_1 &=  2\Re\big( \la \grad u,\grad Q\d W\ra_H\big)\d t+ 2\Re\big(\la |u|^{2\sigma}u,Q\d W\ra_H\big) \d t \notag \\
&\qquad+ 2\sigma_d\kappa\|u\|^{2(\sigma_d-1)}_H\Re\big(\la u,Q\d W\ra_H\big)\d t.
\end{align}
We combine the estimate from \eqref{ineq:d|u|^2n_H} (with $n=\sigma_d$) and \eqref{ineq:<u^(2sigma),Qe_k>} to deduce the bound
\begin{align*}
\d \Phi_1(u)& \le -2\lambda \|\grad u\|^2_{H}\d t+2\frac{1}{100}\lambda\|u\|^{2+2\sigma}_{L^{2+2\sigma}}\d t +\|Q\|^2_{L_{\HS}(U;H^1)}\d t+ \frac{C}{\lambda^\sigma}\|Q\|_{L_{\HS}(U;H^1)}^{2+2\sigma}\d t \notag  \\
&\qquad -\sigma_d\lambda\kappa \|u\|^{2\sigma_d}_H\d t+ C\frac{\|Q\|^{2\sigma_d}_{L_{\HS}(U;H)}}{\lambda^{\sigma_d-1}}\d t +\d M_1,
\end{align*}
In view of \eqref{ineq:Gagliardo} and \eqref{ineq:Phi(u)}, we infer
\begin{align*}
-2\lambda \|\grad u\|^2_{H}\d t+2\frac{1}{100}\lambda\|u\|^{2+2\sigma}_{L^{2+2\sigma}}\d t  -\sigma_d\lambda\kappa \|u\|^{2\sigma_d}_H\le - c\,\lambda  \Phi_1(u),
\end{align*}
whence
\begin{align} \label{ineq:d.Phi(u):focusing}
\d \Phi_1(u) &\le -c \lambda \Phi_1(u)\d t +\|Q\|^2_{L_{\HS}(U;H^1)}\d t+ \frac{C}{\lambda^\sigma}\|Q\|_{L_{\HS}(U;H^1)}^{2+2\sigma}\d t+ C\frac{\|Q\|^{2\sigma_d}_{L_{\HS}(U;H)}}{\lambda^{\sigma_d-1}}\d t +\d M_1.
\end{align}
Next, considering $n>1$, we have
\begin{align} \label{eqn:d.Phi(u)^n:focusing}
\d \Phi_1(u)^n & =n \Phi_1(u)^{n-1}\d \Phi_1(u) + \frac{1}{2}n(n-1)\Phi_1(u)^{n-2} \d\la M_1\ra.
\end{align}
Concerning the first term on the above right-hand side, in view of \eqref{ineq:d.Phi(u):focusing}, the following holds
\begin{align*}
&n \Phi_1(u)^{n-1}\d \Phi_1(u)  \le -c n \lambda \Phi_1(u)^{n}\d t+ n \Phi_1(u)^{n-1}\d M_1\\
&\qquad +n \Phi_1(u)^{n-1}\Big(\|Q\|^2_{L_{\HS}(U;H^1)}+ \frac{C}{\lambda^\sigma}\|Q\|_{L_{\HS}(U;H^1)}^{2+2\sigma}+ C\frac{\|Q\|^{2\sigma_d}_{L_{\HS}(U;H)}}{\lambda^{\sigma_d-1}}\Big)\d t.
\end{align*}
We employ Young's inequality again to infer
\begin{align*}
&n \Phi_1(u)^{n-1}\Big(\|Q\|^2_{L_{\HS}(U;H^1)}+ \frac{C}{\lambda^\sigma}\|Q\|_{L_{\HS}(U;H^1)}^{2+2\sigma}+ C\frac{\|Q\|^{2\sigma_d}_{L_{\HS}(U;H)}}{\lambda^{\sigma_d-1}}\Big)\\
&\le \frac{cn}{100}\lambda\Phi_1(u)^n+ C\frac{\|Q\|^{2n}_{L_{\HS}(U;H^1)}}{\lambda^{n-1}}+  C\frac{\|Q\|_{L_{\HS}(U;H^1)}^{2(1+\sigma)n}}{\lambda^{(1+\sigma)n-1}}+C\frac{\|Q\|^{2n\sigma_d}_{L_{\HS}(U;H)}}{\lambda^{n\sigma_d-1}}.
\end{align*}
With regard to the second term on the right-hand side of \eqref{eqn:d.Phi(u)^n:focusing}, we note that the quaratic variation process $\la M_1\ra$ is given by
\begin{align*}
\d \la M_1\ra = 4 \sum_{k\ge 1}\big|\Re\big(\la \grad u, \grad Qe_k\ra_H+\la |u|^{2\sigma}u,Qe_k\ra_H+ \sigma_d\kappa\|u\|_H^{2(\sigma_d-1)} \la u,Qe_k\ra_H  \big)\big|^2 \d t.
\end{align*}
Similar to estimate \eqref{ineq:d<M_(-1)>}, we have
\begin{align*}
& \big|\Re\big(\la \grad u, \grad Qe_k\ra_H+\la |u|^{2\sigma}u,Qe_k\ra_H+ \sigma_d\kappa\|u\|_H^{2(\sigma_d-1)} \la u,Qe_k\ra_H  \big)\big|^2 \\
&\le   c\big( \|\grad u\|^2_H+\|u\|^{4\sigma +2}_{L^{2+2\sigma}} \big)\|Q\|^2_{L_{\HS}(U;H^1)}+ \|u\|^{4\sigma_d-2}_H\|Q\|^2_{L_{\HS}(U;H)}.
\end{align*}
We invoke \eqref{ineq:Phi(u)} once again to deduce
\begin{align} \label{ineq:d.<M_1>}
\d \la M_1\ra \le c \big(\Phi_1(u)+\Phi_1(u)^{2-\frac{1}{1+\sigma}} \big)\|Q\|^2_{L_{\HS}(U;H^1)}\d t+ c \Phi_1(u)^{2-\frac{1}{\sigma_d}}\|Q\|^2_{L_{\HS}(U;H)}\d t.
\end{align}
It follows that
\begin{align*}
&\frac{1}{2}n(n-1)\Phi_1(u)^{n-2} \d\la M_1\ra \\
& \le c \big(\Phi_1(u)^{n-1}+\Phi_1(u)^{n-\frac{1}{1+\sigma}} \big)\|Q\|^2_{L_{\HS}(U;H^1)}\d t+ c \Phi_1(u)^{n-\frac{1}{\sigma_d}}\|Q\|^2_{L_{\HS}(U;H)}\d t\\
&\le \frac{c}{100}\lambda\Phi_1(u)^n \d t + C\frac{\|Q\|^{2n}_{L_{\HS}(U;H^1)}}{\lambda^{n-1}}\d t+C \frac{\|Q\|^{2n(1+\sigma)}_{L_{\HS}(U;H^1)}}{\lambda^{n(1+\sigma)-1}}\d t+ C\frac{\|Q\|^{2n\sigma_d}_{L_{\HS}(U;H)}}{\lambda^{n\sigma_d-1}}\d t.
\end{align*}
Altogether, from \eqref{eqn:d.Phi(u)^n:focusing}, we obtain the bound
\begin{align} \label{ineq:d.Phi(u)^n:focusing}
\d \Phi_1(u)^n& \le -c_{1,n}\lambda\Phi_1(u)^n \d t+ n\Phi_1(u)^{n-1}\d M_1 + C\frac{\|Q\|^{2n}_{L_{\HS}(U;H^1)}}{\lambda^{n-1}}\d t \notag \\
&\qquad+C \frac{\|Q\|^{2n(1+\sigma)}_{L_{\HS}(U;H^1)}}{\lambda^{n(1+\sigma)-1}}\d t+ C\frac{\|Q\|^{2n\sigma_d}_{L_{\HS}(U;H)}}{\lambda^{n\sigma_d-1}}\d t.
\end{align}
In particular, we also deduce the following estimate in expectation
\begin{align*}
\frac{\d}{\d t} \E\Phi_1(u(t))^n& \le -c_{1,n}\lambda\E\Phi_1(u(t))^n +  C\frac{\|Q\|^{2n}_{L_{\HS}(U;H^1)}}{\lambda^{n-1}}+C \frac{\|Q\|^{2n(1+\sigma)}_{L_{\HS}(U;H^1)}}{\lambda^{n(1+\sigma)-1}}+ C\frac{\|Q\|^{2n\sigma_d}_{L_{\HS}(U;H)}}{\lambda^{n\sigma_d-1}}.
\end{align*}
In turn, Gronwall's inequality implies estimate \eqref{ineq:moment-bound:E.Phi(u(t))^n<e^(-lambda.t)E.Phi(u_0)+C:focusing}, as claimed. The proof is thus finished.

\end{proof}

As a consequence of Lemma \ref{lem:moment-bound:H1}, we establish Lemma \ref{lem:P(sup[Phi+int.Phi])} concerning an estimate on the growth rate of the solutions over time. Notably, the result of Lemma \ref{lem:P(sup[Phi+int.Phi])} will be invoked in the proofs of Lemma \ref{lem:ergodicity:P(ell(k+1)=k+1|l(k)=infty)>epsilon} and Lemma \ref{lem:ergodicity:ell(k+1).neq.l|ell(k)=l)<1/2(1+(k-l)T)^-q}, both of which are employed to establish Theorem \ref{thm:poly-mixing}.

\begin{lemma} \label{lem:P(sup[Phi+int.Phi])} Under Assumptions \ref{cond:sigma} and \ref{cond:Q:poly-mixing}, let $u_0\in L^2(\Omega;H^1)$ and $u(t)$ be the solution of \eqref{eqn:Schrodinger:original} with initial condition $u_0$. Then, for all $n\ge 1$, $q> 2$ and $\rho>0$, the followings hold for some positive constants $K_n$, $K_{1,n,q}$ and $K_{2,n,q}$ independent of $u_0$, $T$, $\sigma$, $\rho$ and $\lambda$:
\begin{align}   \label{ineq:P(sup_[0,T][Phi+int.Phi])}
\P\Big( \sup_{t\in[0,T]}\Big[ & \Phi_\alpha(u(t))^n  + \lambda  \int_0^t\Big(c_{1,n}\Phi_\alpha(u(s))^n +n\|\Gamma(s)\|^{2n}_{H}+n\|\triangle\Gamma(s)\|^ {2n}_{H}\Big)\d s\notag  \\
& \qquad\qquad- K_{n} Q_{1,\alpha,n}t   \Big]  \ge \Phi_\alpha(u_0)^n+\rho\sqrt{T} \Big) \notag \\
&\le  \frac{K_{1,n,q}}{\rho^q} \big(   \E\Phi_\alpha(u_0)^{nq}+ Q_{2,\alpha,n} ^q \big),\quad T>0;
\end{align}
and
\begin{align}\label{ineq:P(sup_[t>T][Phi+int.Phi])}
 \P\Big( \sup_{t\ge T}\Big[ & \Phi_\alpha(u(t))^n  + \lambda  \int_0^t\Big(c_{1,n}\Phi_\alpha(u(s))^n +n\|\Gamma(s)\|^{2n}_{H}+n\|\triangle\Gamma(s)\|^{2n}_{H}\Big)\d s\notag  \\
& \qquad\qquad- K_{n}  \big(Q_{1,\alpha,n} +1 \big)t  \Big]  \ge \Phi_\alpha(u_0)^n+\rho \Big) \notag \\
&\le  \frac{K_{2,n,q}}{(T+\rho)^{q/2-1}} \big(   \E\Phi_\alpha(u_0)^{nq}+ Q_{2,\alpha,n} ^q +1 \big),\quad T>1.
\end{align}
In the above, $\Phi_\alpha$ is the functional defined in \eqref{form:Phi}, $\Gamma$ is the stochastic convolution \eqref{eqn:Schrodinger:Gamma}, and $Q_{1,\alpha,n}$ and $Q_{2,\alpha,n}$ are given by
\begin{align} \label{form:Q_(1,alpha,n)}
Q_{1,\alpha,n}=\begin{cases} \dfrac{\|Q\|^{2n}_{L_{\HS}(U;H^2)}}{\lambda^{n-1}}+ \dfrac{\|Q\|^{2n(1+\sigma)}_{L_{\HS}(U;H^1)}}{\lambda^{n(1+\sigma)-1}},&\alpha=-1,\\
\noalign{\vskip 9pt}
\dfrac{\|Q\|^{2n}_{L_{\HS}(U;H^2)}}{\lambda^{n-1}}+ \dfrac{\|Q\|^{2n(1+\sigma)}_{L_{\HS}(U;H^1)}}{\lambda^{n(1+\sigma)-1}}+ \dfrac{\|Q\|^{2n\sigma_d}_{L_{\HS}(U;H)}}{\lambda^{n\sigma_d-1}},&\alpha=1.
\end{cases}
\end{align}
and
\begin{align} \label{form:Q_(2,alpha,n)}
Q_{2,\alpha,n}=\begin{cases} \dfrac{\|Q\|^{2n}_{L_{\HS}(U;H^2)}}{\lambda^{n-1/2}}+ \dfrac{\|Q\|^{2n}_{L_{\HS}(U;H^1)}}{\lambda^{n}} + \dfrac{\|Q\|^{2n(1+\sigma)}_{L_{\HS}(U;H^1)}}{\lambda^{ n(1+\sigma)}}\\
\noalign{\vskip 7pt} \qquad\qquad+ \|Q\|^{2n}_{L_{\HS}(U;H^1)}+\|Q\|^{2n(1+\sigma)}_{L_{\HS}(U;H^1)},&\alpha=-1,\\
\noalign{\vskip 11pt}
\dfrac{\|Q\|^{2n}_{L_{\HS}(U;H^2)}}{\lambda^{n-1/2}}+\dfrac{\|Q\|^{2n}_{L_{\HS}(U;H^1)}}{\lambda^{n}}+ \dfrac{\|Q\|^{2n(1+\sigma)}_{L_{\HS}(U;H^1)}}{\lambda^{n(1+\sigma)}}+ \dfrac{\|Q\|^{2n\sigma_d}_{L_{\HS}(U;H)}}{\lambda^{n\sigma_d}}\\
\noalign{\vskip 7pt}\qquad+ \|Q\|^{2n}_{L_{\HS}(U;H^1)}+\|Q\|^{2n(1+\sigma)}_{L_{\HS}(U;H^1)}+ \|Q\|^{2n\sigma_d}_{L_{\HS}(U;H)},&\alpha=1.
\end{cases}
\end{align}

\end{lemma}

\begin{proof}
With regard to \eqref{ineq:P(sup_[0,T][Phi+int.Phi])}, we first collect the estimates from \eqref{ineq:d|Gamma|^2n_H} and \eqref{ineq:d|triangle.Gamma|^2n_H} to see that
\begin{align*}
\lambda n\int_0^t \Big(\|\Gamma(s)\|^{2n}_H+n\|\triangle \Gamma(s)\|^{2n}_H\Big)\d s & \le c\frac{\|Q\|^{2n}_{L_{\HS}(U;H^2)}}{\lambda^{n-1}} t+\int_0^t 2n\|\Gamma(s)\|^{2(n-1)}_H\Re\big(\la \Gamma(s),Q\d W(s)\ra_H\big)\\
&\qquad+2n\int_0^t \|\triangle \Gamma(s)\|^{2(n-1)}_H\Re\big(\la \triangle \Gamma(s),\triangle Q\d W(s)\ra_H\big).
\end{align*}
This together with \eqref{ineq:d.Phi(u)^n:defocusing} and \eqref{ineq:d.Phi(u)^n:focusing} produces 
\begin{align} \label{ineq:Phi(u)^n+int.Phi(u)^n+|Gamma|^n<Phi(u_0)+M_alpha}
 \Phi_\alpha(u(t))^{n} +&  \lambda \int_0^t \Big(c_{1,n}\Phi_\alpha(u(s))^{n}+n\|\Gamma(s)\|^{2n}_H+n\|\triangle \Gamma(s)\|^{2n}_H\Big)\d s \notag \\
 &\le  \Phi_\alpha(u_0)^{n} + K_n\,Q_{1,\alpha, n}t+  M_{\alpha,n}(t).
\end{align}
In the above, $c_{1,n}$ is as in \eqref{ineq:d.Phi(u)^n:defocusing} and \eqref{ineq:d.Phi(u)^n:focusing}, $K_n$ is a positive constant independent of $t,u_0,\lambda, Q_{1,\alpha,n}$, and the semi Martingale process $M_{\alpha,n}$ is given by
\begin{align} \label{form:d.M_(alpha,n)}
\d M_{\alpha,n}&= n\Phi_\alpha(u)^{n-1}\d M_{\alpha}+2n\|\Gamma\|^{2(n-1)}_H\Re\big(\la \Gamma,Q\d W\ra_H\big) \notag \\
&\qquad+2n\|\triangle \Gamma\|^{2(n-1)}_H\Re\big(\la \triangle \Gamma,\triangle Q\d W\ra_H\big).
\end{align}
where we recall that $M_\alpha$ is defined in either \eqref{form:dM_(-1)} or \eqref{form:M_1}, corresponding to either $\alpha=-1$ or $\alpha=1$, respectively. In particular, given $\rho,q,T>0$, we have
\begin{align*}
\Big\{\sup_{t\in[0,T]}\Big[ \Phi_\alpha(u(t))^{n} +&  \lambda \int_0^t \Big(c_{1,n}\Phi_\alpha(u(s))^{n}+n\|\Gamma(s)\|^{2n}_H+n\|\triangle \Gamma(s)\|^{2n}_H\Big)\d s - K_0 Q_{\alpha,n}t\Big]\\
&\qquad\ge   \Phi_\alpha(u_0)^{n} +\rho\sqrt{T}\Big\} \subset\Big\{\sup_{t\in[0,T]} M_{\alpha,n}(t)\ge \rho\sqrt{T}\Big\}.
\end{align*}
As a consequence of Markov's and Burkholder's inequalities, we deduce the following bound in probability
\begin{align} \label{ineq:P(sup_[0,T][Phi+int.Phi]):a}
\P\Big( \sup_{t\in[0,T]}\Big[ & \Phi_\alpha(u(t))^{n} + \lambda \int_0^t \Big(c_{1,n}\Phi_\alpha(u(s))^{n}+n\|\Gamma(s)\|^{2n}_H+n\|\triangle \Gamma(s)\|^{2n}_H\Big)\d s - K_0 Q_{\alpha,n}t\Big] \notag \\
&\qquad\qquad\ge   \Phi(u_0)^{n} +\rho\sqrt{T} \Big)\notag \\
&\le \frac{1}{\rho^qT^{q/2}}\E\Big[\sup_{t\in[0,T]} |M_{\alpha,n}(t)|^q  \Big] \le \frac{1}{\rho^qT^{q/2}} \E |\la M_{\alpha,n}\ra(T)|^{q/2}.
\end{align}
It therefore remains to estimate the quadratic variation process on the above right-hand side. To this end, from \eqref{form:d.M_(alpha,n)}, we observe that
\begin{align*}
\d \la M_{\alpha,n}\ra \le c \Phi_\alpha(u)^{2(n-1)}\d \la M_\alpha\ra + c\|\Gamma\|^{4n-2}_H\|Q\|^2_{L_{\HS}(U;H)}\d t+c \|\triangle\Gamma\|^{4n-2}_H\|Q\|^2_{L_{\HS}(U;H^2)}\d t,
\end{align*}
whence
\begin{align} \label{ineq:E<M_(alpha,n)>^(q/2)}
\E |\la M_{\alpha,n}\ra(T)|^{q/2} &\le c\,\E \Big|\int_0^T \Phi_\alpha(u(t))^{2(n-1)}\d \la M_\alpha(t)\ra\Big|^{q/2}+c \, \E\Big|\int_0^T\|\Gamma(t)\|^{4n-2}_H\d t\Big|^{q/2}\|Q\|^q_{L_{\HS}(U;H)} \notag \\
&\qquad +c \, \E\Big|\int_0^T\|\triangle\Gamma(t)\|^{4n-2}_H\d t\Big|^{q/2}\|Q\|^q_{L_{\HS}(U;H^2)}.
\end{align}
Regarding the terms involving $\Gamma$ on the above right-hand side, in view of \eqref{ineq:E.int.|Gamma|^2n_H} and \eqref{ineq:E.int.|triangle.Gamma|^2n_H}, we employ Holder's inequality to infer
\begin{align*}
&\E\Big|\int_0^T\|\Gamma(t)\|^{4n-2}_H\d t\Big|^{q/2}\|Q\|^q_{L_{\HS}(U;H)}+ \E\Big|\int_0^T\|\triangle\Gamma(t)\|^{4n-2}_H\d t\Big|^{q/2}\|Q\|^q_{L_{\HS}(U;H^2)}\\
&\le T^{q/2-1}\Big( \E\int_0^T\|\Gamma(t)\|^{(2n-1)q}_H\d t\|Q\|^q_{L_{\HS}(U;H)}+ \E\int_0^T\|\triangle\Gamma(t)\|^{(2n-1)q}_H\d t\|Q\|^q_{L_{\HS}(U;H^2)}\Big)\\
&\le c\,T^{q/2} \cdot \frac{\|Q\|^{2nq}_{L_{\HS}(U;H^2)}}{\lambda^{(2n-1)q/2}}.
\end{align*}
Turning to the term involving $\la M_\alpha\ra$, there are two cases to be considered depending on $\alpha$. On the one hand, when $\alpha=-1$,  we recall from \eqref{ineq:d<M_(-1)>} that
\begin{align*}
\d \la M_{-1}\ra &\le c \big(\Phi_{-1}(u)+\Phi_{-1}(u)^{2-\frac{1}{1+\sigma}} \big)\|Q\|^2_{L_{\HS}(U;H^1)}\d t,
\end{align*}
whence
\begin{align*}
&\E \Big|\int_0^T \Phi_{-1}(u(t))^{2(n-1)}\d \la M_{-1}(t)\ra\Big|^{q/2} \\
& \le c\,T^{q/2-1}\E \int_0^T\big[ \Phi_{-1}(u(t))^{2n-1}+ \Phi_{-1}(u(t))^{2n-\frac{1}{1+\sigma}}\big]^{q/2}\d t \cdot \|Q\|^q_{L_{\HS}(U;H^1)}\\
&\le c\,T^{q/2-1}\E \int_0^T \Big[\Phi_{-1}(u(t))^{nq}+ \|Q\|^{2qn}_{L_{\HS}(U;H^1)}+\|Q\|^{2qn(1+\sigma)}_{L_{\HS}(U;H^1)} \Big]\d t.
\end{align*}
This together with \eqref{ineq:moment-bound:E.Phi(u(t))^n<e^(-lambda.t)E.Phi(u_0)+C:defocusing} produces
\begin{align*}
&\E \Big|\int_0^T \Phi_{-1}(u(t))^{2(n-1)}\d \la M_{-1}(t)\ra\Big|^{q/2} \\
& \le c\, T^{q/2}\Big[ \E \Phi_{-1}(u_0)^{nq} + \frac{\|Q\|^{2nq}_{L_{\HS}(U;H^1)}}{\lambda^{nq}} + \frac{\|Q\|^{2nq(1+\sigma)}_{L_{\HS}(U;H^1)}}{\lambda^{ nq(1+\sigma)}}
+ \|Q\|^{2qn}_{L_{\HS}(U;H^1)}+\|Q\|^{2qn(1+\sigma)}_{L_{\HS}(U;H^1)}\Big].
\end{align*}
On the other hand, when $\alpha=1$, we recall from \eqref{ineq:d.<M_1>} that
\begin{align*}
\d \la M_1\ra \le c \big(\Phi_1(u)+\Phi_1(u)^{2-\frac{1}{1+\sigma}} \big)\|Q\|^2_{L_{\HS}(U;H^1)}\d t+ c \Phi_1(u)^{2-\frac{1}{\sigma_d}}\|Q\|^2_{L_{\HS}(U;H)}\d t.
\end{align*}
It follows that
\begin{align*}
&\E \Big|\int_0^T \Phi_{1}(u(t))^{2(n-1)}\d \la M_{1}(t)\ra\Big|^{q/2} \\
& \le c\,T^{q/2-1}\E \int_0^T\big[ \Phi_{1}(u(t))^{2n-1}+ \Phi_{-1}(u(t))^{2n-\frac{1}{1+\sigma}}\big]^{q/2}\d t \cdot \|Q\|^q_{L_{\HS}(U;H^1)}\\
&\qquad + c\,T^{q/2-1}\E \int_0^T[ \Phi_{1}(u(t))^{(2n-\frac{1}{\sigma_d})\frac{q}{2}}\d t \cdot \|Q\|^q_{L_{\HS}(U;H)},
\end{align*}
whence
\begin{align*}
&\E \Big|\int_0^T \Phi_{1}(u(t))^{2(n-1)}\d \la M_{1}(t)\ra\Big|^{q/2} \\
&\le c\,T^{q/2-1}\E \int_0^T \Big[\Phi_{1}(u(t))^{nq}+ \|Q\|^{2qn}_{L_{\HS}(U;H^1)}+\|Q\|^{2qn(1+\sigma)}_{L_{\HS}(U;H^1)}+ \|Q\|^{2qn\sigma_d}_{L_{\HS}(U;H)} \Big]\d t.
\end{align*}
In view of \eqref{ineq:moment-bound:E.Phi(u(t))^n<e^(-lambda.t)E.Phi(u_0)+C:focusing}, we immediately obtain
\begin{align*}
&\E \Big|\int_0^T \Phi_{1}(u(t))^{2(n-1)}\d \la M_{1}(t)\ra\Big|^{q/2} \\
&\le c\,T^{q/2}\Big[ \E \Phi_1(u_0)^{nq} + \frac{\|Q\|^{2nq}_{L_{\HS}(U;H^1)} }{\lambda^{nq}}+ \frac{\|Q\|^{2nq(1+\sigma)}_{L_{\HS}(U;H^1)}}{\lambda^{ nq(1+\sigma)}}+\frac{\|Q\|^{2nq\sigma_d}_{L_{\HS}(U;H)}}{\lambda^{nq\sigma_d}}\\
&\qquad \qquad\qquad + \|Q\|^{2qn}_{L_{\HS}(U;H^1)}+\|Q\|^{2qn(1+\sigma)}_{L_{\HS}(U;H^1)}+ \|Q\|^{2qn\sigma_d}_{L_{\HS}(U;H)}\Big].
\end{align*}
Now, we collect the above estimates with \eqref{ineq:E<M_(alpha,n)>^(q/2)} to deduce the bound
\begin{align} \label{ineq:E|<M_(alpha,n)>|^(q/2)}
\E |\la M_{\alpha,n}\ra(T)|^{q/2} \le c\, T^{q/2} \big(\E \Phi_\alpha(u_0)^{nq}+ Q_{2,\alpha,n}^q \big),
\end{align}
where $Q_{2,\alpha,n}$ is the constant defined in \eqref{form:Q_(2,alpha,n)}. In turn, this together with \eqref{ineq:P(sup_[0,T][Phi+int.Phi]):a} implies that
\begin{align*}
\P\Big( \sup_{t\in[0,T]}\Big[ & \Phi_\alpha(u(t))^{n} + \lambda \int_0^t \Big(c_{1,n}\Phi_\alpha(u(s))^{n}+n\|\Gamma(s)\|^{2n}_H+n\|\triangle \Gamma(s)\|^{2n}_H\Big)\d s - K_0 Q_{\alpha,n}t\Big] \notag \\
&\qquad\qquad\ge   \Phi(u_0)^{n} +\rho\sqrt{T} \Big)\notag \\
& \le \frac{1}{\rho^qT^{q/2}} \E |\la M_{\alpha,n}\ra(T)|^{q/2}\le \frac{c}{\rho^{q}}\big(\E \Phi_\alpha(u_0)^{nq}+ Q_{2,\alpha,n}^q \big).
\end{align*}
We emphasize that the constant $c$ on the above right-hand side only depends on $\alpha,n$ and $q$. This establishes \eqref{ineq:P(sup_[0,T][Phi+int.Phi])}, as claimed.

Turning to \eqref{ineq:P(sup_[t>T][Phi+int.Phi])}, from \eqref{ineq:Phi(u)^n+int.Phi(u)^n+|Gamma|^n<Phi(u_0)+M_alpha}, we note that
\begin{align*}
 \Phi_\alpha(u(t))^{n} +&  \lambda \int_0^t \Big(c_{1,n}\Phi_\alpha(u(s))^{n}+n\|\Gamma(s)\|^{2n}_H+n\|\triangle \Gamma(s)\|^{2n}_H\Big)\d s -  K_n\big(Q_{1,\alpha, n}+1\big)t \\
 &\le  \Phi_\alpha(u_0)^{n} +  M_{\alpha,n}(t)-K_n t.
\end{align*}
So, for each $m\ge 0$, we have
\begin{align*}
\P\Big( \sup_{t\in[T+m,T+m+1]} \Big[ & \Phi_\alpha(u(t))^{n} +  \lambda \int_0^t \Big(c_{1,n}\Phi_\alpha(u(s))^{n}+n\|\Gamma(s)\|^{2n}_H+n\|\triangle \Gamma(s)\|^{2n}_H\Big)\d s \\
&\qquad\qquad-  K_n\big(Q_{1,\alpha, n}+1\big)t  \Big]\ge \Phi(u_0)^n+\rho\Big)\\
&\le \P\Big(\sup_{t\in[T+m,T+m+1]} \big[M_{\alpha,n}(t)-K_nt\big]\ge \rho \Big)\\
&\le \P\Big(\sup_{t\in[T+m,T+m+1]} M_{\alpha,n}(t)\ge \rho +K_n(T+m)\Big).
\end{align*}
We invoke Markov's and Burkholder's inequality again to infer
\begin{align*}
& \P\Big(\sup_{t\in[T+m,T+m+1]} M_{\alpha,n}(t)\ge \rho +K_n(T+m)\Big) \\
&\quad\le  \frac{1}{(\rho+K_n(T+m))^{q}}\E\Big[ \sup_{t\in[T+m,T+m+1]}| M_{\alpha,n}(t)|^q\Big] \\
&\quad \le \frac{1}{(\rho+K_n(T+m))^{q}}\E\big[|\la M_{\alpha,n}\ra(T+m+1)|^{q/2}\big].
\end{align*}
In light of \eqref{ineq:E|<M_(alpha,n)>|^(q/2)}, we obtain
\begin{align*}
& \P\Big(\sup_{t\in[T+m,T+m+1]} M_{\alpha,n}(t)\ge \rho +K_n(T+m)\Big)\\
&\quad \le c\frac{(T+m+1)^{q/2}}{(\rho+K_n(T+m))^{q}}\big(\E\Phi_\alpha(u_0)^{nq}+Q_{2,\alpha,n}^q\big),
\end{align*}
whence
\begin{align*}
\P\Big( \sup_{t\in[T+m,T+m+1]} \Big[ & \Phi_\alpha(u(t))^{n} +  \lambda \int_0^t \Big(c_{1,n}\Phi_\alpha(u(s))^{n}+n\|\Gamma(s)\|^{2n}_H+n\|\triangle \Gamma(s)\|^{2n}_H\Big)\d s \\
&\qquad\qquad-  K_n\big(Q_{1,\alpha, n}+1\big)t  \Big]\ge \Phi(u_0)^n+\rho\Big)\\
&\le c\frac{(T+m+1)^{q/2}}{(\rho+K_n(T+m))^{q}}\big(\E\Phi_\alpha(u_0)^{nq}+Q_{2,\alpha,n}^q\big).
\end{align*}
In the above, we emphasize that the positive constant $c$ is independent of $\rho,T+m,u_0$ and $Q_{2,\alpha,n}$. As a consequence, we deduce the estimate on $[T,\infty)$
\begin{align*}
\P\Big( \sup_{t\ge T} \Big[ & \Phi_\alpha(u(t))^{n} +  \lambda \int_0^t \Big(c_{1,n}\Phi_\alpha(u(s))^{n}+n\|\Gamma(s)\|^{2n}_H+n\|\triangle \Gamma(s)\|^{2n}_H\Big)\d s \\
&\qquad\qquad-  K_n\big(Q_{1,\alpha, n}+1\big)t  \Big]\ge \Phi(u_0)^n+\rho\Big)\\
&\le c\,\big(\E\Phi_\alpha(u_0)^{nq}+Q_{2,\alpha,n}^q\big)\cdot\sum_{m\ge 0}\frac{(T+m+1)^{q/2}}{(\rho+K_n(T+m))^{q}}.
\end{align*}
To control the infinite sum on the above right-hand side, it is important to recall that $T\ge 1$, $q>2$ and that $K_n$ does not depend on $\rho$ and $T$. Thus, we have
\begin{align*}
\sum_{m\ge 0}\frac{(T+m+1)^{q/2}}{(\rho+K_n(T+m))^{q}} \le c\sum_{m\ge 0}\frac{1}{(\rho+T+m)^{q/2}} \le \frac{c}{(\rho+T)^{q/2-1}},
\end{align*}
implying
\begin{align*}
\P\Big( \sup_{t\ge T} \Big[ & \Phi_\alpha(u(t))^{n} +  \lambda \int_0^t \Big(c_{1,n}\Phi_\alpha(u(s))^{n}+n\|\Gamma(s)\|^{2n}_H+n\|\triangle \Gamma(s)\|^{2n}_H\Big)\d s \\
&\qquad\qquad-  K_n\big(Q_{1,\alpha, n}+1\big)t  \Big]\ge \Phi(u_0)^n+\rho\Big)\\
&\le c\,\big(\E\Phi_\alpha(u_0)^{nq}+Q_{2,\alpha,n}^q\big)\cdot \frac{1}{(\rho+T)^{q/2-1}}.
\end{align*}
This produces \eqref{ineq:P(sup_[t>T][Phi+int.Phi])}, thereby finishing the proof.

\end{proof}

In Lemma \ref{lem:Strichart:int_0^t|u|_infty^2sigma}, stated and proven next, we assert a pathwise estimates on the solution of \eqref{eqn:Schrodinger:original} in $L^\infty$ norm via the Lyapunov function $\Phi_\alpha$ and the noise by exploiting Strichartz's inequalities derived in Lemma \ref{lem:Strichart:|u|_infty}. The result of Lemma \ref{lem:Strichart:int_0^t|u|_infty^2sigma} will be used directly to establish both Theorem \ref{thm:unique-ergodicity} and Theorem \ref{thm:poly-mixing} later in Section \ref{sec:poly-mixing}.

\begin{lemma} \label{lem:Strichart:int_0^t|u|_infty^2sigma}
Given $u_0\in L^2(\Omega;H^1)$, let $u(t)$ be the solution of \eqref{eqn:Schrodinger:original} with initial condition $u_0$. Suppose that 
 
1. in dimension $d=1,2$, Assumptions \ref{cond:sigma} and \ref{cond:Q:poly-mixing} hold;

2. in dimension $d=3$, Assumptions \ref{cond:sigma} and \ref{cond:Q:poly-mixing} and condition \eqref{cond:sigma:poly-mixing} hold.
 
Then, there exist positive constants $C_\sigma$ and $n_\sigma$ independent of $\lambda$, $t$ and $u_0$ such that the following holds
\begin{align} \label{ineq:Strichart:int_0^t|u|_infty^2sigma}
\int_0^t \|u(s)\|^{2\sigma}_{L^\infty}\d s &\le  C_\sigma \int_0^t \Big[c_{1,n_\sigma}\Phi_\alpha(u(s))^{n_\sigma}+n_\sigma\|\Gamma(s)\|^{n_\sigma}_H+n_\sigma \|\triangle\Gamma(s)\|^{n_\sigma}_H\Big]\d s \notag \\
&\qquad + C_\sigma (\|u_0\|^{2\sigma}_{H^1}+1+t),\quad t\ge 0.
\end{align}
In the above, $c_{1,n}$ is the constant from Lemma \ref{lem:moment-bound:H1}, and $\Gamma(s)$ is the stochastic convolution solving \eqref{eqn:Schrodinger:Gamma}.
\end{lemma}
\begin{proof}
We note that in dimension $d=1$, estimate \eqref{ineq:Strichart:int_0^t|u|_infty^2sigma} clearly holds for suitable constants $C_\sigma$ and $n_\sigma$, thanks to the fact that $H^1(\rbb)\subset L^\infty(\rbb)$ and that $\Phi_\alpha$ given by \eqref{form:Phi} dominates $H^1$ norm. Considering the case $d=2,3$, from \eqref{eqn:Schrodinger:original}, we can recast $u(t)=u(t;u_0)$ in the mild formulation as follow.
\begin{align} \label{form:u(t):mild_solution}
u(t) = e^{-\lambda t}S(t) u_0 + \int_0^t e^{-\lambda (t-s) }S(t-s)F_\alpha(u(s))\d s+ \Gamma(t) .
\end{align}
So,
\begin{align*}
\int_0^t \|u(s)\|^{2\sigma}_{L^\infty}\d s
&\le C\int_0^t \big\|e^{-\lambda s}S(s)u_0\big\|^{2\sigma}_{L^\infty} +  \Big\|\int_0^s  S(s-\ell)F_\alpha(u(\ell))\d \ell\Big\|_{L^\infty}^{2\sigma} +  \|\Gamma(s)\|_{L^\infty}^{2\sigma} \d s.
\end{align*}
In light of Lemma \ref{lem:Strichart:|u|_infty} together with the embedding $L^\infty\subset H^2$ ($d\le 3$), we readily obtain
\begin{align*}
\int_0^t \|u(s)\|^{2\sigma}_{L^\infty}\d s & \le C \|u_0\|^{2\sigma}_{H^1}+ CT+C\int_0^t \|u(s)\|^{q_\sigma}_{H^1}\d s+ C\int_0^t \|\Gamma(s)\|^{2\sigma}_H+ \|\triangle\Gamma(s)\|^{2\sigma}_H\d s,
\end{align*}
where $q_\sigma>2$ is the constant from Lemma \ref{lem:Strichart:|u|_infty}. Furthermore, since $\Phi_\alpha(u)$ defined in \eqref{form:Phi} dominates $H^1$ norm, we may infer a positive constant $n_\sigma$ sufficiently large such that
\begin{align*}
\int_0^t \|u(s)\|^{2\sigma}_{L^\infty}\d s \le C (\|u_0\|^{2\sigma}_{H^1}+1+t)+ C \int_0^t c_{1,n_\sigma}\Phi_\alpha(u(s))^{n_\sigma}+n_\sigma\|\Gamma(s)\|^{n_\sigma}_H+n_\sigma \|\triangle\Gamma(s)\|^{n_\sigma}_H\d s,
\end{align*}
for some positive constant $C$ independent of $\lambda$, $t$ and $u_0$. This produces \eqref{ineq:Strichart:int_0^t|u|_infty^2sigma}, as claimed.
\end{proof}

Lastly, we state the following irreducibility condition through Lemma \ref{lem:irreducibility}, asserting that the probability that the solutions eventually return to the origin is uniform with respect to any initial data in a bounded ball. The proof of Lemma \ref{lem:irreducibility} is similar to those of \cite[Estimate (4.13)]{debussche2005ergodicity}
 and \cite[Lemma 4.5]{nguyen2024inviscid} tailored to the setting of unbounded domains. We note that the result of Lemma \ref{lem:irreducibility} appears later in the proof of Lemma \ref{lem:ergodicity:P(ell(k+1)=k+1|l(k)=infty)>epsilon}, which is one of the ingredients for the mixing rate \eqref{ineq:poly-mixing}.

\begin{lemma} \label{lem:irreducibility} 
Suppose that 
 
1. in dimension $d=1,2$, Assumptions \ref{cond:sigma} and \ref{cond:Q:poly-mixing} hold;

2. in dimension $d=3$, Assumptions \ref{cond:sigma} and \ref{cond:Q:poly-mixing} and condition \eqref{cond:sigma:poly-mixing} hold.
 
Then, for all $R,\,r>0$, there exists $T_*=T_*(R,r)>1$ such that for all $t\ge T_*$ and $u_0^1,u_0^2$ satisfying $\Phi_\alpha(u_0^1)+\Phi_\alpha(u_0^2)\le R$, the following holds
\begin{align} \label{ineq:irreducibility}
\P\Big( \Phi_\alpha\big(u(t;u_0^1)\big) + \Phi_\alpha\big(u(t;u_0^2)\big)\le r \Big)\ge \varepsilon_*,
\end{align}
for some positive constant $\varepsilon_*=\varepsilon_*(t,R,r)$ independent of $u_0^2$ and $u_0^2$. 
\end{lemma}
\begin{proof}
Letting $\Gamma$ be the stochastic convolution solving \eqref{eqn:Schrodinger:Gamma}, we set $v=u-\Gamma$. From \eqref{eqn:Schrodinger:original} and \eqref{eqn:Schrodinger:Gamma}, observe that $v$ satisfies the equation
\begin{align} \label{eqn:v}
\frac{\d}{\d t}v+ \i\triangle v+ \lambda v+ \i \alpha |v+\Gamma|^{2\sigma} (v+\Gamma)=0,
\end{align}
with the initial data $v(0)=u(0)$. Recalling $\Phi_\alpha$ from \eqref{form:Phi}, we aim to produce an estimate on $\Phi_\alpha(v)$ in terms of $\Gamma$. We start with computing $\frac{\d}{\d t}\|\grad v\|_H^2$ while making use of integration by parts.
\begin{align*}
&\frac{\d}{\d t}\|\grad v\|_H^2\\
 & = -2\lambda \|\grad v\|_H^2-\alpha \i \big\la \grad\big[|v+\Gamma|^{2\sigma} \big](v+\Gamma),\grad v  \big\ra_H-\alpha \i \big\la |v+\Gamma|^{2\sigma} \grad(v+\Gamma),\grad v  \big\ra_H\\
&\qquad +\alpha \i \big\la \grad\big[|v+\Gamma|^{2\sigma} \big](\bar{v}+\bar{\Gamma}),\grad \bar{v}  \big\ra_H +\alpha \i \big\la |v+\Gamma|^{2\sigma} \grad(\bar{v}+\bar{\Gamma}),\grad \bar{v}  \big\ra_H \\
&= -2\lambda \|\grad v\|_H^2-2\alpha(\sigma+1)\Re\Big( \i \big\la |v+\Gamma|^{2\sigma}\grad \Gamma,\grad v \big\ra_H\Big) -2\alpha \sigma\Re\Big( \i \big\la |v+\Gamma|^{2\sigma-2}(v+\Gamma)^2\grad\bar{\Gamma},\grad v  \big\ra_H\Big)\\
&\qquad - 2\alpha\sigma\Re \Big(\i \big\la |v+\Gamma|^{2\sigma-2}(v+\Gamma)^2,(\grad v)^2   \big\ra_H\Big).
\end{align*}
We employ Holder's inequality to infer
\begin{align*}
-\alpha(\sigma+1)&2\Re\Big( \i \big\la |v+\Gamma|^{2\sigma}\grad \Gamma,\grad v \big\ra_H\Big) -\alpha \sigma2\Re\Big( \i \big\la |v+\Gamma|^{2\sigma-2}(v+\Gamma)^2\grad\bar{\Gamma},\grad v  \big\ra_H\Big)\\
&\le (4\sigma+2)\|\grad \Gamma\|_{L^\infty}\|\grad v\|_H\|v+\Gamma\|^{2\sigma}_{L^{4\sigma}}.
\end{align*}
To further bound the above right-hand side, we employ Sobolev embedding $H^2\subset L^\infty$ and $H^1\subset L^p$ ($p\le \infty$, $d=1$, or $p<\infty$, $d=2$ or $p\le 6$, $d=3$) to deduce
\begin{align*}
\|\grad \Gamma\|_{L^\infty}\|\grad v\|_H\|v+\Gamma\|^{2\sigma}_{L^{4\sigma}}\le c\|\Gamma\|_{H^3}\|v\|^{1+2\sigma}_{H^1}+c\|\Gamma\|^{1+2\sigma}_{H^3}\|v\|_{H^1}.
\end{align*}
It follows that
\begin{align} \label{ineq:d/dt.|grad.v|^2}
\frac{\d}{\d t}\|\grad v\|^2_H&\le -2\lambda \|\grad v\|_H^2+c\|\Gamma\|_{H^3}\|v\|^{1+2\sigma}_{H^1}+c\|\Gamma\|^{1+2\sigma}_{H^3}\|v\|_{H^1}  \notag \\
&\qquad- 2\alpha\sigma\Re \Big(\i \big\la |v+\Gamma|^{2\sigma-2}(v+\Gamma)^2,(\grad v)^2   \big\ra_H\Big).
\end{align}

Next, considering $\|v\|^{2+2\sigma}_{L^{2+2\sigma}}$, we have
\begin{align*}
&\frac{\d}{\d t}\Big(\frac{-\alpha}{1+\sigma}\|v\|^{2+2\sigma}_{L^{2+2\sigma}}\Big)\\
&= 2\alpha\lambda\|v\|^{2+2\sigma}_{L^{2+2\sigma}}-2\alpha \Re\Big(  \i \big\la  |v|^{2\sigma}v,\triangle v\big\ra_H \Big)+2\Re\Big(  \i\big\la |v|^{2\sigma}|v+\Gamma|^{2\sigma}(v+\Gamma),v\big\ra_H \Big).
\end{align*}
On the one hand, we employ integration by parts on the term involving $\triangle v$ and obtain
\begin{align*}
-2\alpha \Re\Big(  \i \big\la  |v|^{2\sigma}v,\triangle v\big\ra_H \Big) = 2\alpha\sigma\Re \Big(\i \big\la |v|^{2\sigma-2}(v)^2,(\grad v)^2   \big\ra_H\Big).
\end{align*}
On the other hand, we note that
\begin{align*}
2\Re\Big(  \i\big\la |v|^{2\sigma}|v+\Gamma|^{2\sigma}(v+\Gamma),v\big\ra_H \Big) 
&= 2\Re\Big(  \i\big\la |v|^{2\sigma}|v+\Gamma|^{2\sigma}\Gamma,v\big\ra_H \Big)\\
& \le c \big\la |v|^{2\sigma}(|v|^{2\sigma}+|\Gamma|^{2\sigma})|\Gamma|,|v|\big\ra_H\\
&\le c \big( \|v\|^{1+4\sigma}_{L^{1+4\sigma}}\|\Gamma\|_{L^\infty}+\|v\|^{1+2\sigma}_{L^{1+2\sigma}}\|\Gamma\|^{1+2\sigma}_{L^\infty}  \big).
\end{align*}
We invoke Sobolev embedding once again to deduce
\begin{align*}
2\Re\Big(  \i\big\la |v|^{2\sigma}|v+\Gamma|^{2\sigma}(v+\Gamma),v\big\ra_H \Big) \le c \big( \|v\|^{1+4\sigma}_{H^1}\|\Gamma\|_{H^3}+\|v\|^{1+2\sigma}_{H^1}\|\Gamma\|^{1+2\sigma}_{H^3}  \big).
\end{align*}
It follows that 
\begin{align}\label{ineq:d/dt.|v|^(2+2sigma)_(L^(2+2sigma))}
&\frac{\d}{\d t}\Big(\frac{-\alpha}{1+\sigma}\|v\|^{2+2\sigma}_{L^{2+2\sigma}}\Big) \notag \\
&\le 2\alpha\lambda\|v\|^{2+2\sigma}_{L^{2+2\sigma}}+ 2\alpha\sigma\Re \Big(\i \big\la |v|^{2\sigma-2}(v)^2,(\grad v)^2   \big\ra_H\Big)\notag \\
&\qquad+ c \big( \|v\|^{1+4\sigma}_{H^1}\|\Gamma\|_{H^3}+\|v\|^{1+2\sigma}_{H^1}\|\Gamma\|^{1+2\sigma}_{H^3}+\|\Gamma\|_{H^3}\|v\|^{1+2\sigma}_{H^1}+\|\Gamma\|^{1+2\sigma}_{H^3}\|v\|_{H^1}  \big).
\end{align}

From \eqref{ineq:d/dt.|grad.v|^2} and \eqref{ineq:d/dt.|v|^(2+2sigma)_(L^(2+2sigma))}, we obtain
\begin{align*}
& \frac{\d}{\d t}\Big(\|\grad v\|^2_H+\frac{-\alpha}{1+\sigma}\|v\|^{2+2\sigma}_{L^{2+2\sigma}}\Big)\\
&\le -2\lambda\big( \|\grad v\|^2_{H^1}-\alpha \|v\|^{2+2\sigma}_{L^{2+2\sigma}} \big)\\
&\qquad+c\big(\|v\|^{1+2\sigma}_{H^1}\|\Gamma\|^{1+2\sigma}_{H^3}+ \|v\|^{1+4\sigma}_{H^1}\|\Gamma\|_{H^3}+\|\Gamma\|_{H^3}\|v\|^{1+2\sigma}_{H^1}+\|\Gamma\|^{1+2\sigma}_{H^3}\|v\|_{H^1}\big)\\
&\qquad +2\alpha\sigma\Re \Big(\i \big\la |v|^{2\sigma-2}(v)^2- |v+\Gamma|^{2\sigma-2}(v+\Gamma)^2,(\grad v)^2   \big\ra_H\Big).
\end{align*}
Regarding the last term on the above right-hand side, in view of Lemma \ref{lem:|z|^(2-2sigma)z^2-|x|^(2-2sigma)<|x-z|^2sigma}, on the one hand, when $\sigma\in(0,1/2]$, inequality \eqref{ineq:|z|^(2-2sigma)z^2-|x|^(2-2sigma)<|x-z|^2sigma:sigma<1/2} implies
\begin{align*}
&2\alpha\sigma\Re \Big(\i \big\la |v|^{2\sigma-2}(v)^2- |v+\Gamma|^{2\sigma-2}(v+\Gamma)^2,(\grad v)^2   \big\ra_H\Big)\\
&\le c\, \la |\Gamma|^{2\sigma},|\grad v|^2\ra_H\le c\|\Gamma\|^{2\sigma}_{L^\infty}\|\grad v\|^2_H\le c\|\Gamma\|^{2\sigma}_{H^3}\|\grad v\|^2_H.
\end{align*}
On the other hand, when $\sigma>1/2$, we invoke \eqref{ineq:|z|^(2-2sigma)z^2-|x|^(2-2sigma)<|x-z|^2sigma:sigma>1/2} to infer
\begin{align*}
&2\alpha\sigma\Re \Big(\i \big\la |v|^{2\sigma-2}(v)^2- |v+\Gamma|^{2\sigma-2}(v+\Gamma)^2,(\grad v)^2   \big\ra_H\Big)\\
&\le c\big\la |\Gamma|(|v+\Gamma|^{2\sigma-1}+|v|^{2\sigma-1} ),|\grad v|^2\big\ra_H \\
&\le c\big(\|\Gamma\|^{2\sigma}_{H^3}+\|\Gamma\|_{H^3}\|v\|^{2\sigma-1}_{L^\infty}  \big)\|\grad v\|^2_H\\
&\le c \big(\|\Gamma\|^{2\sigma}_{H^3}\|\grad v\|^2_H +\|\Gamma\|_{H^3}\|v\|^{2\sigma}_{L^\infty}+\|\Gamma\|_{H^3}\|\grad v\|^{4\sigma}_H\big).
\end{align*}
In the last implication above, we invoked Holder's inequality. It follows from both cases that we may infer a positive constant $c$ such that 
\begin{align*}
&2\alpha\sigma\Re \Big(\i \big\la |v|^{2\sigma-2}(v)^2- |v+\Gamma|^{2\sigma-2}(v+\Gamma)^2,(\grad v)^2   \big\ra_H\Big)\\
&\le c \big( \|\Gamma\|^{2\sigma}_{H^3}\|\grad v\|^2_H +\|\Gamma\|_{H^3}\|v\|^{2\sigma}_{L^\infty}+\|\Gamma\|_{H^3}\|\grad v\|^{4\sigma}_H\big),
\end{align*}
whence
\begin{align}\label{ineq:d/dt(|grad.v|^2+|v|^(2+2sigma)_(L^(2+2sigma)))}
& \frac{\d}{\d t}\Big(\|\grad v\|^2_H+\frac{-\alpha}{1+\sigma}\|v\|^{2+2\sigma}_{L^{2+2\sigma}}\Big)\notag \\
&\le -2\lambda\big( \|\grad v\|^2_{H^1}-\alpha \|v\|^{2+2\sigma}_{L^{2+2\sigma}} \big)+c\|\Gamma\|_{H^3}\|v\|^{2\sigma}_{L^\infty}+c\big( \|\Gamma\|^{2\sigma}_{H^3}\| v\|^2_{H^1} +\|\Gamma\|_{H^3}\| v\|^{4\sigma}_{H^1}\big) \notag \\
&\qquad+c\big(\|\Gamma\|^{1+2\sigma}_{H^3}\|v\|^{1+2\sigma}_{H^1}+\|\Gamma\|_{H^3} \|v\|^{1+4\sigma}_{H^1}+\|\Gamma\|_{H^3}\|v\|^{1+2\sigma}_{H^1}+\|\Gamma\|^{1+2\sigma}_{H^3}\|v\|_{H^1}\big) .
\end{align}
 
Next, with regard to $\|v\|_H$, from expression \eqref{form:Phi}, we first consider the case defocusing case when $\alpha=-1$.
\begin{align*}
\frac{\d}{\d t}\|v\|_H^2 & = -2\lambda\|v\|_H^2-2\alpha\Re\Big( \i \big\la |v+\Gamma|^{2\sigma}\Gamma,v \big\ra_H\Big)\\
&\le -2\lambda\|v\|^2_H+c\big(\|v\|^{1+2\sigma}_{L^{1+2\sigma}}\|\Gamma\|_{L^\infty}+\|\Gamma\|^{1+2\sigma}_{L^\infty}\|v\|_{L^1}  \big)\\
&\le -2\lambda\|v\|^2_H+c\big(\|v\|^{1+2\sigma}_{H^1}\|\Gamma\|_{H^3}+\|\Gamma\|^{1+2\sigma}_{H^3}\|v\|_{H^1}  \big).
\end{align*}
Similarly, when $\alpha=1$, recalling $\sigma_d=1+ 2\sigma /(2-\sigma d)$ defined in \eqref{form:Phi}, we employ the above estimate to see that
\begin{align*}
\frac{\d}{\d t}\kappa\|v\|^{2\sigma_d}_{H}&=\sigma_d\kappa\|v\|^{2(\sigma_d-1)}_H\frac{\d}{\d t}\|v\|^2_H\\
&\le -2\sigma_d\lambda\kappa\|v\|^{2\sigma_d}_H+c\|v\|^{2(\sigma_d-1)}_H\big(\|v\|^{1+2\sigma}_{H^1}\|\Gamma\|_{H^3}+\|\Gamma\|^{1+2\sigma}_{H^3}\|v\|_{H^1}  \big)
\end{align*}
Together with \eqref{ineq:d/dt(|grad.v|^2+|v|^(2+2sigma)_(L^(2+2sigma)))}, the following holds
\begin{align*} 
&\frac{\d}{\d t} \Phi_\alpha(v)  \le -c \lambda\Phi_\alpha(v)+ C\|\Gamma\|_{H^3}\|v\|^{2\sigma}_{L^\infty} \notag \\
&\qquad+ C\big( \|\Gamma\|_{H^3}^{1+2\sigma}+ \|\Gamma\|_{H^3}^{1+4\sigma}+\|\Gamma\|_{H^3}^{2\sigma}+\|\Gamma\|_{H^3} \big)\big(\Phi_\alpha(v)^{n}+1\big),
\end{align*}
for some positive constants $c,C$ and $n$ independent of $\lambda$, $v$ and $\Gamma$. In turn, Gronwall's inequality implies that
\begin{align} \label{ineq:Phi(v)<|phi(u_0)+int_0^t|v|_infty}
 \Phi_\alpha(v(t))  &\le e^{-c\lambda t}\Phi(u_0) + C\sup_{s\in[0,t]}\|\Gamma(s)\|_{H^3}\int_0^t  \|v(s)\|^{2\sigma}_{L^\infty}\d s  \notag \\
&\qquad+ C\,t\sup_{s\in[0,t]}\big[ \|\Gamma(s)\|_{H^3}^{1+2\sigma}+ \|\Gamma(s)\|_{H^3}^{1+4\sigma}+\|\Gamma(s)\|_{H^3}^{2\sigma}+\|\Gamma(s)\|_{H^3} \big] \notag \\
&\qquad\qquad\times\sup_{s\in[0,t]}\big[\Phi_\alpha(v(s))^{n}+1\big].
\end{align}
At this point, there are two cases to be considered depending on the dimension $d$.

Case 1: $d=1$. In this case, thanks to the embedding $H^1(\rbb)\subset L^\infty(\rbb)$, we immediately obtain
\begin{align*}
\sup_{s\in[0,t]}\|\Gamma(s)\|_{H^3}\int_0^t  \|v(s)\|^{2\sigma}_{L^\infty}\d s \le C \,t \sup_{s\in[0,t]}\|\Gamma(s)\|_{H^3}\sup_{s\in[0,t]}\big[\Phi_\alpha(v(s))^{n}+1\big],
\end{align*}
whence
\begin{align*}
\Phi_\alpha(v(t))  &\le e^{-c\lambda t}\Phi(u_0) + C\,t\sup_{s\in[0,t]}\big[ \|\Gamma(s)\|_{H^3}^{1+2\sigma}+ \|\Gamma(s)\|_{H^3}^{1+4\sigma}+\|\Gamma(s)\|_{H^3}^{2\sigma}+\|\Gamma(s)\|_{H^3} \big] \notag \\
&\qquad\qquad\times\sup_{s\in[0,t]}\big[\Phi_\alpha(v(s))^{n}+1\big].
\end{align*}

Case 2: $d=2,3$. In this case, we aim to control the term involving $\|v\|_{L^\infty}$ on the right-hand side of \eqref{ineq:Phi(v)<|phi(u_0)+int_0^t|v|_infty} by exploiting the Strichartz-typed estimates established in Lemma \ref{lem:Strichart:|u|_infty}. To this end, from \eqref{eqn:v}, we note that $v$ can be recast as follow.
\begin{align*}
v(s) =e^{-\lambda s}S(s)u_0+\int_0^s e^{-\lambda(s-\ell)}S(s-\ell)F_\alpha(v(\ell)+\Gamma(\ell))\d \ell,
\end{align*}
whence
\begin{align*}
&\int_0^t \|v(s)\|^{2\sigma}_{L^\infty}\d s \\
& \le c \int_0^t \|e^{-\lambda s}S(s)u_0\|^{2\sigma}_{L^\infty}\d s + c\int_0^t \Big\|  \int_0^s S(s-\ell)F_\alpha(v(\ell)+\Gamma(\ell))\d \ell \Big\|^{2\sigma}_{L^\infty}\d s.
\end{align*}
In view of \eqref{ineq:Strichart:|u_0|_infty}, we readily have
\begin{align*}
\int_0^t \|e^{-\lambda s}S(s)u_0\|^{2\sigma}_{L^\infty}\d s\le c\|u_0\|^{2\sigma}_{H^1},
\end{align*}
whereas \eqref{ineq:Strichart:|u|_infty} implies
\begin{align*}
&\int_0^t \Big\|  \int_0^s S(s-\ell)F_\alpha(v(\ell)+\Gamma(\ell))\d \ell \Big\|^{2\sigma}_{L^\infty}\d s\\
&\le c\Big(t+\int_0^t \|v(s)\|^{q_\sigma}_{H^1}+ \|\Gamma(s)\|^{q_\sigma}_{H^1}\d s \Big).
\end{align*}
In the above, $q_\sigma$ is the constant from Lemma \ref{lem:Strichart:|u|_infty}. Since $\Phi_\alpha$ dominates $H^1$-norm, we may pick $n$ larger (if necessary) to infer
\begin{align*}
\int_0^t  \|v(s)\|^{2\sigma}_{L^\infty}\d s \le c\, t \sup_{s\in[0,t]}\big[\Phi_\alpha(v)^n+\|\Gamma(s)\|^n+1 \big].
\end{align*}
It follows from \eqref{ineq:Phi(v)<|phi(u_0)+int_0^t|v|_infty} that
\begin{align*}
 \Phi_\alpha(v(t)) &
 \le e^{-c\lambda t}\Phi(u_0)+C \|u_0\|_{H^1}^{2\sigma}\sup_{s\in[0,t]}\|\Gamma(s)\|_{H^3} \\
&\qquad+ C\,t\sup_{s\in[0,t]}\big[ \|\Gamma(s)\|_{H^3}^{1+2\sigma}+ \|\Gamma(s)\|_{H^3}^{1+4\sigma}+\|\Gamma(s)\|_{H^3}^{2\sigma}+\|\Gamma(s)\|_{H^3} \big]\\
&\qquad\qquad\times  \sup_{s\in[0,t]}\big[\Phi_\alpha(v)^n+\|\Gamma(s)\|^{n}_{H^3}+1 \big] .
\end{align*}

Altogether, from both cases, we obtain the following bound
\begin{align} \label{ineq:Phi(u_1)+Phi(u_2))}
&\Phi_\alpha\big(u(t;u_0^1)\big)+\Phi_\alpha\big(u(t;u_0^2)\big) \notag \\
&\le Ce^{-c\lambda t} \big[\Phi_\alpha(u_0^1)+\Phi_\alpha(u_0^2)\big]+C\big[\big(\Phi_\alpha(u_0^1)^n+\Phi_\alpha(u_0^2)\big)^n+1\big]\sup_{s\in[0,t]}\|\Gamma(s)\|_{H^3} \notag\\
& \qquad+ C\,t\sup_{s\in[0,t]}\big[ \|\Gamma(s)\|_{H^3}^{1+2\sigma}+ \|\Gamma(s)\|_{H^3}^{1+4\sigma}+\|\Gamma(s)\|_{H^3}^{2\sigma}+\|\Gamma(s)\|_{H^3} \big]  \notag\\
&\qquad\qquad\times \sup_{s\in[0,t]}\big[\big(\Phi_\alpha\big(u(t;u_0^1)\big)+\Phi_\alpha\big(u(t;u_0^2)\big)\big)^n+\|\Gamma(s)\|^{n}_{H^3}+1 \big] .
\end{align}
In the above, we emphasize that the positive constants $C$ and $c$ are independent of $u_0^1$, $u_0^2$, $t$ and $\lambda$. 

Turning back to \eqref{ineq:irreducibility}, we shall follow closely the arguments of \cite[(4.13)]{debussche2005ergodicity} and \cite[Lemma 4.5]{nguyen2024inviscid} adapting to our settings so as to produce the desired irreducible property. To see this, let $\tau$ be the stopping time defined as
\begin{align*}
\tau=\inf\{t\ge 0:\big(\Phi_\alpha\big(u(t;u_0^1)\big)+\Phi_\alpha\big(u(t;u_0^2)\big)\big)^n\ge 3CR\},
\end{align*}
where $C$ is the same constant as in \eqref{ineq:Phi(u_1)+Phi(u_2))}. Also, for each $t\ge 1$, consider the event
\begin{align*}
B=\Big\{ \sup_{s\in[0,t]} \big[ \|\Gamma(s)\|_{H^3}^{1+2\sigma}+ \|\Gamma(s)\|_{H^3}^{1+4\sigma}&+\|\Gamma(s)\|_{H^3}^{2\sigma}+\|\Gamma(s)\|_{H^3}+\|\Gamma\|^n_{H^3} \big]\\
& \le \frac{R}{t^2(R^n+3CR+R+1)}\Big\}.
\end{align*}
We claim that conditioned on $B$, $\tau\ge t$ a.s. Indeed, suppose by contradiction, $\tau\le t$. From \eqref{ineq:Phi(u_1)+Phi(u_2))}, since $\Phi_\alpha(u_0^1)+\Phi_\alpha(u_0^2)\le R$, observe that
\begin{align*}
&\Phi_\alpha\big(u(\tau;u_0^1)\big)+\Phi_\alpha\big(u(\tau;u_0^2)\big)\\
& \le Ce^{-c\tau}R +C(R^n+1)\cdot \frac{R}{t^2(R^n+3CR+R+1)}\\
&\qquad+ C\, \tau \cdot \frac{R}{t^2(R^n+3CR+R+1)} \Big[ 3CR+ \frac{R}{t^2(R^n+3CR+R+1)}+1  \Big]\\
&<3CR,
\end{align*}
which contradicts the definition of $\tau$. It follows that $\tau\ge t$. Now, let $T^*$ be given by
\begin{align} \label{form:T^*}
T^* = \max\Big\{\frac{3CR}{cr},\frac{3CR}{r},1  \Big\}.
\end{align}
For all $t\ge T^*$, conditioning on $B$ again yields
\begin{align*}
&\Phi_\alpha\big(u(\tau;u_0^1)\big)+\Phi_\alpha\big(u(\tau;u_0^2)\big)\\
& \le Ce^{-c t}R +C(R^n+1)\cdot \frac{R}{t^2(R^n+3CR+R+1)}\\
&\qquad +C \frac{R}{t(R^n+3CR+R+1)} \Big[ 3CR+ \frac{R}{t^2(R^n+3CR+R+1)}+1  \Big]\\
&\le CR\Big( e^{-ct} +\frac{2}{t}\Big)
\end{align*}
We invoke the elementary inequality $e^a\ge a$ to further deduce
\begin{align*}
\Phi_\alpha\big(u(\tau;u_0^1)\big)+\Phi_\alpha\big(u(\tau;u_0^2)\big) \le r.
\end{align*}
In view of Assumption \ref{cond:Q:poly-mixing}, the law of $\Gamma$ is full in $C([0,t];H^3)$, implying
\begin{align*}
\P\Big(\Phi_\alpha\big(u(\tau;u_0^1)\big)+\Phi_\alpha\big(u(\tau;u_0^2)\big) \le r\Big)\ge \P(B)>0.
\end{align*} 
This produces \eqref{ineq:irreducibility}, thereby finishing the proof.
\end{proof}

\section{Unique ergodicity and Polynomial Mixing} \label{sec:poly-mixing}

In this section, we establish Theorem \ref{thm:unique-ergodicity} concerning the uniqueness of the invariant probability measure and Theorem \ref{thm:poly-mixing} giving the algebraic convergence rate toward equilibrium for the solutions of \eqref{eqn:Schrodinger:original}. Particularly, in Section \ref{sec:poly-mixing:proof-unique-ergodicity}, we provide the proof of Theorem \ref{thm:unique-ergodicity} whereas in Section \ref{sec:poly-mixing:proof-main-result}, we review the coupling argument developed in \cite{debussche2005ergodicity,odasso2006ergodicity} and supply the proof of Theorem \ref{thm:poly-mixing}. In Section \ref{sec:poly-mixing:proof-aux-result}, we prove the auxiliary results employed to conclude Theorem \ref{thm:poly-mixing} while making use of the moment estimates collected in Section \ref{sec:moment-bound}.

\subsection{Proof of Theorem \ref{thm:unique-ergodicity}} \label{sec:poly-mixing:proof-unique-ergodicity}

Following the approach of \cite{brzezniak2023ergodic, glatt2021long}, the unique ergodicity argument consists of two main steps. First of all, we derive moment bounds on the regularity of $\nu$ with respect to $L^\infty$ norm. This is summarized in Lemma \ref{lem:nu:regularity} below. Then, we employ Birkhoff's Ergodic Theorem to assert that the distance between two solutions starting from two distinct initial conditions can be approximated by regularity of $\nu$ independently of $\lambda$. Altogether, we may take $\lambda$ sufficiently large to conclude that the two solutions must converge to one another, yielding the uniqueness of $\nu$.

We start the procedure by stating and proving Lemma \ref{lem:nu:regularity}, giving an estimate on the support of an invariant probability measure $\nu$.

\begin{lemma}\label{lem:nu:regularity}
Under the same hypotheis of Theorem \ref{thm:unique-ergodicity}, let $\nu$ be an invariant probability measure of $P_t$. Then, the following holds
    \begin{align} \label{ineq:nu:L^infty}
        \int_{H}\|u\|_{L^\infty}^{2\sigma}\nu(\d u)<C,
    \end{align}
    for some positive constant $C$ independent of $\nu$ and $\lambda$.
\end{lemma}
\begin{proof}
    We note that estimate \eqref{ineq:nu:L^infty} was previously proven in \cite[Proposition 4.1]{brzezniak2023ergodic} except for the case 
    \begin{align*}
        \frac{1+\sqrt{17}}{4}\le \sigma<\frac{3}{2},\quad \alpha=-1,\, d=3.
    \end{align*}
   To this end, we employ an argument similar to the proof of Lemma \ref{lem:Strichart:int_0^t|u|_infty^2sigma} as follows.
   \begin{align*}
\int_0^T \|u(t)\|^{2\sigma}_{L^\infty}\d t &\le C\int_0^T \|e^{-\lambda t}S(t)u_0\|^{2\sigma}_{L^\infty}+ \Big\|\int_0^t e^{-\lambda(t-s)}S(t-s)F_\alpha(u(s))\d s\Big\|^{2\sigma}_{L^\infty}\d t\\
&\qquad+C\int_0^T \|\Gamma(t)\|^{2\sigma}_{L^\infty}\d t.
\end{align*}
   In the above, $\Gamma$ is the stochastic convolution solving \eqref{eqn:Schrodinger:Gamma} and $C=C(\sigma)>0$ is a positive constant independent of $\lambda$, $T$ and $u_0$. On the one hand, in light of Lemma \ref{lem:moment-bound:H1}, we readily have
   \begin{align*}
&\E\int_0^T \|e^{-\lambda t}S(t)u_0\|^{2\sigma}_{L^\infty}+ \Big\|\int_0^t e^{-\lambda(t-s)}S(t-s)F_\alpha(u(s))\d s\Big\|^{2\sigma}_{L^\infty}\d t\\
&\le   C\Big(\|u_0\|^{2\sigma}_{H^1}+ T+\int_0^T\E \|u(t)\|^{q_\sigma}_{H^1}\d s\Big)\\
&\le C( \|u_0\|^{q}_{H^1}+1+T),
\end{align*}
where in the last implication, we have employed Lemma \ref{lem:moment-bound:H1} for some positive constant $q$ large enough. On the other hand, from \eqref{ineq:E.int.|Gamma|^2sigma_L^infty}, we get
\begin{align*}
    \E\int_0^T \|\Gamma(t)\|^{2\sigma}_{L^\infty}\d t \le CT.
\end{align*}
Altogether, we deduce the bound
\begin{align} \label{ineq:|u|^2sigma_infty:unique-ergodicity}
   \E \int_0^T \|u(t)\|^{2\sigma}_{L^\infty}\d t  \le C( \|u_0\|^{q}_{H^1}+1+T),
\end{align}
for some positive constant $C$ independent of $\lambda$. In turn, we can employ an argument similar to the proof of \cite[Proposition 4.1]{brzezniak2023ergodic} to establish \eqref{ineq:nu:L^infty}. More specifically, for $R>0$, consider the function
\begin{align*}
    g_R(u) = \|u\|^{2\sigma}_{L^\infty}\wedge R,\quad u\in H.
\end{align*}
By the invariance of $\nu$, we have the following chain of implications
\begin{align*}
    \int_H g_R(u_0)\nu(\d u_0)  = \int_0^1  \int_H g_R(u_0)\nu(\d u_0) \d s & =  \int_H \int_0^1 P_sg_R(u_0)\d s\, \nu(\d u_0)\\
    &\le \int_H \int_0^1 \E\|u(s;u_0)\|^{2\sigma}_{L^\infty}\d s\, \nu(\d u_0)\\
    &\le C\int_H \|u_0\|^{q}_{H^1} \nu(\d u_0)+C.
\end{align*}
We note that under Assumptions \ref{cond:Q:well-posed} and \ref{cond:sigma}, it is already established in \cite[Inequality (4.1)]{brzezniak2023ergodic} that
\begin{align*}
    \int_H \|u_0\|^n_{H^1}\nu(\d u_0)<\infty,\quad n>0.
\end{align*}
It follows that picking $n$ sufficiently large such that $n>q$ yields the bound
\begin{align*}
     \int_H \|u_0\|^{2\sigma}_{L^\infty}\nu(\d u_0)=\int_H g_R(u_0)\nu(\d u_0)  \le  C\int_H \|u_0\|^{n}_{H^1} \nu(\d u_0)+C<C,
\end{align*}
for some positive constant $C$ independent of $R$. In turn, we may take $R$ to infinity to obtain \eqref{ineq:nu:L^infty}, by virtue of the Dominated Convergence Theorem. The proof is thus finished.
\end{proof}

\begin{remark}
We note that in the proof of Lemma \ref{lem:nu:regularity}, we cannot employ Lemma \ref{lem:Strichart:int_0^t|u|_infty^2sigma} to produce estimate \eqref{ineq:|u|^2sigma_infty:unique-ergodicity}. In fact, Lemma \ref{lem:Strichart:int_0^t|u|_infty^2sigma} requires noise satisfy Assumption \ref{cond:Q:poly-mixing} ($H^3$-regularity) whereas in Lemma \ref{lem:nu:regularity}, we only impose Assumption \ref{cond:Q:well-posed} ($H^1$-regularity).
\end{remark}

Having obtained a moment bound that is independent of $\lambda$, we now conclude Theorem \ref{thm:unique-ergodicity} while making use of the regularity of $\nu$ from Lemma \ref{lem:nu:regularity}. Since the argument is relatively short, we include it here for the sake of completeness. See also the proof of \cite[Theorem 5.1]{brzezniak2023ergodic}.

\begin{proof}[Proof of Theorem \ref{thm:unique-ergodicity}]
Let $\nu_1$ and $\nu_2$ be two invariant probability measures. By the ergodic decomposition, we may assume that they are both ergodic. Given an arbitrary bounded and Lipschitz function $f:H\to\rbb$, we aim to prove that
\begin{align*}
    \int_H f(u)\nu_1(\d u) =\int_H f(u)\nu_2(\d u). 
\end{align*}
Indeed, let $u_0^i$, $i=1,2$ be arbitrary elements in the support of $\nu_i$. By Birkhoff's ergodic Theorem, we have a.s.
\begin{align*}
    \frac{1}{t}\int_0^tf(u(s;u_0^i))\d s\to \int_H f(u)\nu_i(\d u) ,\quad t\to\infty.
\end{align*}
Since $f$ is Lipschitz, we infer
\begin{align*}
    \Big|\int_H f(u)\nu_1(\d u) -\int_H f(u)\nu_2(\d u) \Big|\le \liminf_{t\to\infty}\frac{C}{t}\int_0^t\|u(s;u_0^1) -u(s;u_0^2)\|_H\d t,
\end{align*}
for some positive constant $C=C(f)$ independent of $t,u_0^1$ and $u_0^2$. It therefore suffices to establish that for all $\lambda$ sufficiently large, a.s.
\begin{align*}
    \|u(t;u_0^1) -u(t;u_0^2)\|_H\to 0,\quad t\to\infty.
\end{align*}
To this end, for notational simplicity, we set $u_i = u(t;u_0^i)$, $i=1,2$ and consider the difference $w= u_1-u_2$. From \eqref{eqn:Schrodinger:original}, observe that $w$ satisfies the equation
\begin{align*}
\frac{\d}{\d t} w +\i \triangle w+\i \alpha\big(|u_1|^{2\sigma}u_1 -|u_2|^{2\sigma}u_2\big)+\lambda w =0,
\end{align*}
with the initial condition $w(0)=u_0^1-u_0^2$. A routine calculation in $H$ produces the estimate
\begin{align*}
\frac{\d}{\d t}\|w\|^2_H & \le  -2\lambda \|w\|^2_H +\big|\la |u_1|^{2\sigma}u_1-|u_2|^{2\sigma}u_2, w \ra_H  \big|.
\end{align*}
Using the elementary inequality
\begin{align*}
\big||z_1|^{2\sigma}z_1-|z_2|^{2\sigma}z_2\big|\le C(\sigma) |z_1-z_2|(|z_1|^{2\sigma}+|z_2|^{2\sigma}),
\end{align*}
we obtain
\begin{align*}
\frac{\d}{\d t}\|w\|^2_H & \le  -2\lambda \|w\|^2_H + C\big(\|u_1\|^{2\sigma}_{L^\infty}+\|u_2\|^{2\sigma}_{L^\infty}\big)\| w \|_H^2  ,
\end{align*}
whence
\begin{align} \label{ineq:|u_1-u_2|} 
\|w(t)\|^2_H \le \|u_0^1-u_0^2\|^2_H\exp\Big\{-2\lambda t +C\int_0^t \big(\|u_1(s)\|^{2\sigma}_{L^\infty}+\|u_2(s)\|^{2\sigma}_{L^\infty}\big)\d s  \Big\}.
\end{align}
In the above, we emphasize that the constant $C=C(\sigma)$ is independent of $\lambda$, $t$, and $u_0^i$, $i=1,2$. In light of Lemma \ref{lem:nu:regularity}, we have a.s.
\begin{align} \label{lim:int_0^t.|u|_infty->nu}
   t\to\infty,\quad \int_0^t\|u_i(s)\|^{2\sigma}_{L^\infty}\d s\to \int_H\|u\|^{2\sigma}_{L^\infty}\nu_i(\d u)<C_1,
\end{align}
where $C_1$ does not depend on $\lambda$ and $\nu_i$, $i=1,2$. Therefore, we may take $\lambda$ sufficiently large to obtain the a.s. bound
\begin{align*} 
\|w(t)\|^2_H \le \|u_0^1-u_0^2\|^2_H\exp\Big\{t\Big(-2\lambda  +2C\cdot C_1\Big)  \Big\}\to 0,\quad t\to\infty.
\end{align*}
The proof is thus finished.
\end{proof}

\subsection{Proof of Theorem \ref{thm:poly-mixing}} \label{sec:poly-mixing:proof-main-result}

 Turning to Theorem \ref{thm:poly-mixing} on the mixing rate of $P_t$ toward $\nu$, the argument is drawn upon the framework of \cite{debussche2005ergodicity,nguyen2024inviscid, odasso2006ergodicity} tailored to the settings of $\rbb^d$. First of all, we introduce the notion of coupling two solutions through Definition \ref{def:ell_beta} below.

\begin{definition} \label{def:ell_beta} 
1. For every $u_0^1,u_0^2\in L^2(\Omega;H^1)$, $\theta>0$, $\beta>0$, $T>0$, $n>0$ and $k\in\nbb$, define 
\begin{align*}
\ell_{\theta,\beta}(k)=\min\{l\in\{0,\dots,k\}: P_{l,k} \textup{ holds} \},
\end{align*}
where $\min \emptyset =\infty$ and 
\begin{align*}
P_{l,k} =\begin{cases}
\Phi_\alpha(u(lT;u_0^1))+ \Phi_\alpha(u(lT;u_0^2))\le \beta,\\
E_n(t,u_0^i) \le \theta+ \beta^n + K_n(Q_{1,\alpha,n}+1)(t-lT),& \forall t\in [lT,kT],\, i=1,2.
\end{cases}
\end{align*}
In the above, $\Phi_\alpha$ is defined in \eqref{form:Phi}, $E(t;u_0)$ is given by 
\begin{align} \label{form:E_n}
E_n(t;u_0) = \Phi_\alpha(u(t;u_0))^n+\lambda  \int_0^t\Big(c_{1,n}\Phi_\alpha(u(s;u_0))^n +n\|\Gamma(s)\|^{2n}_{H}+n\|\triangle\Gamma(s)\|^ {2n}_{H}\Big)\d s,
\end{align}
and $K_n$ and $Q_{1,\alpha,n}$ are the constants from Lemma \ref{lem:P(sup[Phi+int.Phi])}.

2. The pair $\big(u(t;u_0^1),\ug(t;u_0^2)\big)$ is said to be \textup{coupled} in $[lT,kT]$ if $\ell_{\theta,\beta}(k)=l$.
\end{definition}

In order to help explain the strategy of our coupling argument as well as the motivation behind the random variable $\ell_{\theta,\beta}$, let us briefly discuss the main shortcoming in the proof of Theorem \ref{thm:unique-ergodicity}. Indeed, it is important to point out that the argument presented in Section \ref{sec:poly-mixing:proof-unique-ergodicity} relies on the ergodic behavior \eqref{lim:int_0^t.|u|_infty->nu}. While this limit is sufficient to deduce the unique ergodicity, it does not provide a quantitative estimate on the convergence speed, so as to deduce a mixing rate. To overcome the issue, we observe that by using Strichartz estimates, cf. Lemma \ref{lem:Strichart:int_0^t|u|_infty^2sigma}, the $L^\infty$ norm in \eqref{lim:int_0^t.|u|_infty->nu} can be subsumed by the Lyapunov functional $\Phi_\alpha$ and the stochastic convolution $\Gamma(\cdot)$. The expression $E_n(\cdot)$ appearing in Definition \ref{def:ell_beta} therefore plays the role of a control on the growth rate over time of the solutions' trajectories. Heuristically, the random variable $\ell_{\theta,\beta}$ is used to keep track of the coupling behaviors of the two solutions starting from an initial time $lT$ until they become decoupled.    

Now, we proceed to establish the algebraic rate in \eqref{ineq:poly-mixing}, by exploiting the structure of $\ell_{\theta,\beta}$. Following the framework of \cite{debussche2005ergodicity}, it is not difficult to see that for every $\theta>0$, $\beta>0$, the random variable $\ell_{\theta,\beta}$ as in Definition \ref{def:ell_beta} satisfies
\begin{align*}
\begin{cases}
\ell_{\theta,\beta}(k+1)=l \text{ implies } \ell_{\theta,\beta}(k)=l, \quad l\le k,\\
\ell_{\theta,\beta}(k) \in \{0,1,\dots,k\}\cup\{\infty\},\\
\ell_{\theta,\beta}(k) \text{ depends only on } u(\,\cdot\,;u_0^1) \text{ and } u(\,\cdot\,;u_0^2),\\
\ell_{\theta,\beta}(k)=k \text{ implies }\Psi(u(kT;u_0^1))+ \Psi(u(kT;u_0^2))\le \beta.
\end{cases}
\end{align*}
In particular, this verifies \cite[(2.11)]{debussche2005ergodicity}. 

Next, the three main ingredients for Theorem \ref{thm:poly-mixing} are stated below through Lemmas \ref{lem:ergodicity:P(d_1(u_1,u_2)>t^-q)}, \ref{lem:ergodicity:P(ell(k+1)=k+1|l(k)=infty)>epsilon} and \ref{lem:ergodicity:ell(k+1).neq.l|ell(k)=l)<1/2(1+(k-l)T)^-q}, whose proofs are deferred to Section \ref{sec:poly-mixing:proof-aux-result}.

We start with Lemma \ref{lem:ergodicity:P(d_1(u_1,u_2)>t^-q)} proving that once the solutions are coupled, i.e., they enter a ball and subsequently have moderate growth rates, they have to stay close to one another with respect to the distance $d_1$ defined in \eqref{form:d_1}. Notably, the proof of Lemma \ref{lem:ergodicity:P(d_1(u_1,u_2)>t^-q)} will employ the Strichartz estimate derived in Lemma \ref{lem:Strichart:int_0^t|u|_infty^2sigma}.

\begin{lemma} \label{lem:ergodicity:P(d_1(u_1,u_2)>t^-q)}
Under the same hypothesis of Theorem \ref{thm:poly-mixing}, let $\ell_{\theta,\beta}$ be the random variable as in Definition \ref{def:ell_beta}. Then, there exists $\lambda=\lambda(Q)$ sufficiently large such that for all $\theta$, $\beta$, $q$, $T> 0$, $0\le l\le k$, and $t\in [lT,kT]$, the following holds
\begin{align} \label{ineq:ergodicity:P(d_1(u_1,u_2)>t^-q)}
\P\big(\big\{ \big\| u(t;u_0^1)-u(t;u_0^2) \big\|_{H}\mi 1\ge c_0 (t-lT)^{-q}\big\} \cap \big\{\ell_{\theta,\beta}(k)\le l \big\} \big) \le c_0(t-lT)^{-q},
\end{align}
for some positive constant $c_0=c_0(\theta,\beta,\lambda,q)$ independent $t$, $T$, $l$, $k$, $u_0^1$, and $u_0^2$. 
\end{lemma}

The second auxiliary result for the proof of Theorem \ref{thm:poly-mixing} is given below through Lemma \ref{lem:ergodicity:P(ell(k+1)=k+1|l(k)=infty)>epsilon}, establishing a positive probability of coupling while the two solutions are being decoupled. We remark that although Lemma \ref{lem:ergodicity:P(ell(k+1)=k+1|l(k)=infty)>epsilon} does not directly employ Lemma \ref{lem:Strichart:int_0^t|u|_infty^2sigma}, its argument uses the irreducibility property from Lemma \ref{lem:irreducibility}, which is established by invoking Lemma \ref{lem:Strichart:int_0^t|u|_infty^2sigma}.

\begin{lemma} \label{lem:ergodicity:P(ell(k+1)=k+1|l(k)=infty)>epsilon}
Under the same hypothesis of Theorem \ref{thm:poly-mixing}, let $R$, $\theta$, $\beta$ and $\lambda>0$ be arbitrarily given. Then, there exists a positive constant $T_1=T_1(R,\beta)>1$ such that the following holds 
\begin{align} \label{ineq:ergodicity:P(ell(k+1)=k+1|l(k)=infty)>epsilon}
\P\big( \ell_{\theta,\beta}(k+1)=k+1 \big| \ell_{\theta,\beta}(k)=\infty, \Phi_\alpha (u(kT_1;u_0^1))+\Phi_\alpha (u(kT_1;u_0^2))\le R  \big)\ge \varepsilon_1,
\end{align}
for some positive constant $\varepsilon_1=\varepsilon_1(T_1,\beta)$ and for all $k=0,1,2,\dots$
\end{lemma}

Lastly, we formulate a small probability of two solutions decoupling over time, through Lemma \ref{lem:ergodicity:ell(k+1).neq.l|ell(k)=l)<1/2(1+(k-l)T)^-q} below, whose argument employs the probabilistic estimates in Lemma \ref{lem:P(sup[Phi+int.Phi])}. 

\begin{lemma} \label{lem:ergodicity:ell(k+1).neq.l|ell(k)=l)<1/2(1+(k-l)T)^-q}
Under the same hypothesis of Theorem \ref{thm:poly-mixing}, let $T_1$ be the time constant as in Lemma \ref{lem:ergodicity:P(ell(k+1)=k+1|l(k)=infty)>epsilon}. Then, for all $\theta=\theta(T_1,\beta)$ sufficiently large and $q\ge4$, the following holds
\begin{align} \label{ineq:ergodicity:ell(k+1).neq.l|ell(k)=l)<1/2(1+(k-l)T)^-q}
\P\big(  \ell_{\theta,\beta}(k+1)\neq l|\ell_{\theta,\beta}(k)=l \big) \le \frac{1}{2}[1+(k-l)T_1]^{-q},\quad 0\le l\le k.
\end{align}

\end{lemma}

\begin{remark} \label{remark:poly-mixing} 1. Following the proof of \cite[Theorem 2.9]{debussche2005ergodicity}, it is important to point out that the algebraic inequality of Lemma \ref{lem:ergodicity:ell(k+1).neq.l|ell(k)=l)<1/2(1+(k-l)T)^-q} ultimately yields the polynomial mixing rate appearing in \eqref{ineq:poly-mixing}. 

2. As presented in Section \ref{sec:poly-mixing:proof-aux-result}, on the one hand, the condition that the parameter $\lambda$ is sufficiently greater than noise intensity is only needed in Lemma \ref{lem:ergodicity:P(d_1(u_1,u_2)>t^-q)}. On the other hand, the proofs of Lemmas \ref{lem:ergodicity:P(ell(k+1)=k+1|l(k)=infty)>epsilon} and \ref{lem:ergodicity:ell(k+1).neq.l|ell(k)=l)<1/2(1+(k-l)T)^-q} rely on the moment estimates in Lemma \ref{lem:P(sup[Phi+int.Phi])} and the irreducibility condition in Lemma \ref{lem:irreducibility}, both of which do not require large damping.
\end{remark}

Assuming the above results, let us now conclude the proof of Theorem \ref{thm:poly-mixing} by verifying the hypothesis of \cite[Theorem 2.9]{debussche2005ergodicity}.  Indeed, the results in Lemma \ref{lem:ergodicity:P(d_1(u_1,u_2)>t^-q)}, Lemma \ref{lem:ergodicity:ell(k+1).neq.l|ell(k)=l)<1/2(1+(k-l)T)^-q} and Lemma \ref{lem:ergodicity:P(ell(k+1)=k+1|l(k)=infty)>epsilon} respectively establish the conditions \cite[(2.12)-(2.13)-(2.14)]{debussche2005ergodicity}. In addition, the moment bounds of Lemma \ref{lem:moment-bound:H1} supply the required Lyapunov functional in \cite[(2.15)]{debussche2005ergodicity}. All of these in turn allow for producing the polynomial mixing rate \eqref{ineq:poly-mixing}, by virtue of \cite[Theorem 2.9]{debussche2005ergodicity}. See also \cite[Theorem 2.3]{nguyen2024inviscid} and \cite[Theorem 1.8]{odasso2006ergodicity}.

\subsection{Proofs of auxiliary results} \label{sec:poly-mixing:proof-aux-result}

In this subsection, we establish the auxiliary coupling results, which were employed to conclude Theorem \ref{thm:poly-mixing} in Section \ref{sec:poly-mixing:proof-main-result}. We start with the proof of Lemma \ref{lem:ergodicity:P(d_1(u_1,u_2)>t^-q)}, which relies on Strichartz estimates formulated in Lemma \ref{lem:Strichart:int_0^t|u|_infty^2sigma}.

\begin{proof}[Proof of Lemma \ref{lem:ergodicity:P(d_1(u_1,u_2)>t^-q)}] Setting $w=u_1-u_2$, recall from \eqref{ineq:|u_1-u_2|} that 
\begin{align*} 
\|w(t)\|^2_H \le \|u_0^1-u_0^2\|^2_H\exp\Big\{-2\lambda t +C\int_0^t \big(\|u_1(s)\|^{2\sigma}_{L^\infty}+\|u_2(s)\|^{2\sigma}_{L^\infty}\big)\d s  \Big\}.
\end{align*}
In view of Lemma \ref{lem:Strichart:int_0^t|u|_infty^2sigma}, cf. \eqref{ineq:Strichart:int_0^t|u|_infty^2sigma}, we infer the existence of positive constants $C_\sigma$ and $n_\sigma$ such that
\begin{align} \label{ineq:|u_1-u_2|^2_H}
&\|w(t)\|^2_H  \le \|u_0^1-u_0^2\|^2_H\exp\Big\{-2\lambda t + C_\sigma ( \|u_0^1\|^{2\sigma}_{H^1}+\|u_0^2\|^{2\sigma}_{H^1}+2+2t)\Big\} \notag \\
&\times\exp\Big\{ C_\sigma \int_0^t \Big[c_{1,n_\sigma}\big(\Phi_\alpha(u_1(s))^{n_\sigma}+\Phi_\alpha(u_2(s))^{n_\sigma}\big)+2n_\sigma\|\Gamma(s)\|^{n_\sigma}_H+2n_\sigma \|\triangle\Gamma(s)\|^{n_\sigma}_H\Big]\d s \Big\}.
\end{align} 

Turning back to \eqref{ineq:ergodicity:P(d_1(u_1,u_2)>t^-q)}, without loss of generality, we may assume $l=0$. Note that
\begin{align*}
&\big\{ \| u_1(t)-u_2(t) \|_{H}\mi 1\ge c_0 t^{-q}\big\} \cap \big\{\ell_{\theta,\beta}(k) =0 \big\}\\
& = \big\{ \| w(t) \|_{H}^2 \ge c_0^2 t^{-2q}\big\} \cap \big\{\ell_{\theta,\beta}(k)= 0 \big\}.
\end{align*}
It follows from Markov's inequality that
\begin{align*}
&\P\big( \big\{ \| u_1(t)-u_2(t) \|_{H}\mi 1\ge c_0 t^{-q}\big\} \cap \big\{\ell_{\theta,\beta}(k) =0 \big\} \big)\\
&\le \frac{t^{2q}}{c_0^2}\E [\|w(t)\|^2_H|\ell_{\theta,\beta}(k)=0].
\end{align*}
To further bound the above right-hand side, from Definition \ref{def:ell_beta}, given that $ \big\{\ell_{\theta,\beta}(k)= 0 \big\}$ has occurred, we get 
\begin{align*}
\Phi_\alpha(u_0^1)+\Phi_\alpha(u_0^2)\le \beta,
\end{align*}
and that for $t\in[0,kT]$,
\begin{align*}
&C_\sigma \int_0^t \Big[c_{1,n_\sigma}\big(\Phi_\alpha(u_1(s))^{n_\sigma}+\Phi_\alpha(u_2(s))^{n_\sigma}\big)+2n_\sigma\|\Gamma(s)\|^{n_\sigma}_H+2n_\sigma \|\triangle\Gamma(s)\|^{n_\sigma}_H\Big]\d s \\
&\qquad\le \frac{2C_\sigma}{\lambda}\Big(\theta+\beta^{n_\sigma} +K_{n_\sigma}(Q_{1,\alpha,n_\sigma}+1)t\Big).
\end{align*}
Also, the choice of $\Phi_\alpha$ in \eqref{form:Phi} implies that 
\begin{align*}
\|u_0^1-u_0^2\|^2_H \le c( \Phi_\alpha(u_0^1)+ \Phi_\alpha(u_0^2)+[\Phi_\alpha(u_0^1)+\Phi_\alpha(u_0^2)]^{m_\sigma}),
\end{align*}
where $m_\sigma>0$ is a positive constant.  In particular, conditioned on $ \big\{\ell_{\theta,\beta}(k)= 0 \big\}$, it is clear that
\begin{align*}
\|u_0^1-u_0^2\|^2_H \le c( \beta+\beta^{m_\sigma}).
\end{align*}
Likewise,
\begin{align*}
\|u_0^1\|^{2\sigma}_{H^1}+\|u_0^2\|^{2\sigma}_{H^1}\le c \beta^\sigma.
\end{align*}
Altogether with \eqref{ineq:|u_1-u_2|^2_H}, we infer the existence of positive constants $c$ and $C$ such that
\begin{align*}
&\E [\|w(t)\|^2_H|\ell_{\theta,\beta}(k)=0]\\
& \le c(\beta+\beta^{m_\sigma})\exp\Big\{-2\lambda t + C( \beta^\sigma+1+t)+\frac{C}{\lambda}\Big(\theta+\beta^{n_\sigma} +K_{n_\sigma}(Q_{1,\alpha,n_\sigma}+1)t\Big)\Big\}\\
& = c(\beta+\beta^{m_\sigma})\exp\Big\{ C(\beta^\sigma+1)+\frac{C}{\lambda}(\theta+\beta^{n_\sigma})  \Big\} \exp\Big\{ -2\lambda t +Ct+\frac{C}{\lambda}K_{n_\sigma}(Q_{1,\alpha,n_\sigma}+1)t  \Big\}.
\end{align*}
In the above, while $c$ and $C$ possibly depend on $\sigma$, we emphasize that they are both independent of $\lambda$, $\theta$, $\beta$ and $t$. Recalling the constants $K_{n_\sigma}$ and $Q_{1,\alpha,n_\sigma}$ from Lemma \ref{lem:P(sup[Phi+int.Phi])}, we see that by taking $\lambda=\lambda(Q)$ sufficiently large regardless of $\theta$ and $\beta$, the following holds
\begin{align*}
&\E [\|w(t)\|^2_H|\ell_{\theta,\beta}(k)=0]\\
& \le  c(\beta+\beta^{m_\sigma})\exp\Big\{ C(\beta^\sigma+1)+\frac{C}{\lambda}(\theta+\beta^{n_\sigma})  \Big\} \exp\{ -\lambda t \},
\end{align*}
whence
\begin{align*}
&\P\big( \big\{ \big\| u_1(t)-u_2(t) \big\|_{H}\mi 1\ge c_0 t^{-q}\big\} \cap \big\{\ell_{\theta,\beta}(k) =0 \big\} \big)\\
&\le \frac{t^{2q}}{c_0^2}c(\beta+\beta^{m_\sigma})\exp\Big\{ C(\beta^\sigma+1)+\frac{C}{\lambda}(\theta+\beta^{n_\sigma})  \Big\} \exp\{ -\lambda t \}\\
&\le \frac{1}{c_0^2 t^{q}}c(\beta+\beta^{m_\sigma})\exp\Big\{ C(\beta^\sigma+1)+\frac{C}{\lambda}(\theta+\beta^{n_\sigma})  \Big\}.
\end{align*}
At this point, we simply pick $c_0$ satisfying
\begin{align*}
c_0^3= c(\beta+\beta^{m_\sigma})\exp\Big\{ C(\beta^\sigma+1)+\frac{C}{\lambda}(\theta+\beta^{n_\sigma})  \Big\},
\end{align*}
to produce the bound
\begin{align*}
\P\big( \big\{ \big\| u_1(t)-u_2(t) \big\|_{H}\mi 1\ge c_0 t^{-q}\big\} \cap \big\{\ell_{\theta,\beta}(k) =0 \big\} \big)\le \frac{c_0}{t^q}.
\end{align*}
This establishes \eqref{ineq:ergodicity:P(d_1(u_1,u_2)>t^-q)}, thus finishing the proof.

\end{proof}

Next, we turn our attention to Lemma \ref{lem:ergodicity:P(ell(k+1)=k+1|l(k)=infty)>epsilon} and employ the irreducibility condition \eqref{ineq:irreducibility} to establish this result.

\begin{proof}[Proof of Lemma \ref{lem:ergodicity:P(ell(k+1)=k+1|l(k)=infty)>epsilon}]
First of all, we claim that for all $\beta>0$, the following holds
\begin{align}\label{ineq:ergodicity:P(Phi(u_1)+Phi(u_2)<beta)>1/2}
\P\big(  \Phi_\alpha (u(t_1;u_0^1)) + \Phi_\alpha (u(t_1;u_0^2))\le \beta \big)\ge\frac{1}{2},
\end{align}
for some positive constants $t_1=t_1(\beta)$ and $r_1=r_1(\beta)$, and for all $u_0^1$ and $u_0^2$ satisfying
\begin{align*}
\Phi_\alpha (u_0^1)+\Phi_\alpha (u_0^1)\le r_1.
\end{align*}
Indeed, let $t_1$ and $r_1$ be given and be chosen later. Recalling function $E_n$ as in \eqref{form:E_n}, for $\rho_1>0$, we introduce the stopping times $\tau_i$, $i=1,2$, defined as 
\begin{align*}
\tau_i = \inf\big\{ t\ge 0: E_n(t;u_0^i)-K_n(Q_{1,\alpha,n}+1)t\ge \Phi_\alpha (u_0^i)^n+\rho_1\sqrt{t_1}\big \}.
\end{align*}
In the above, $K_n$ and $Q_{1,\alpha,n}$ are the constants appearing in Lemma \ref{lem:P(sup[Phi+int.Phi])}, cf. \eqref{ineq:P(sup_[0,T][Phi+int.Phi])}. Given that 
\begin{align*}
\Phi_\alpha (u_0^1) + \Phi_\alpha (u_0^2) \le r_1,
\end{align*}
we observe that for all $t_1$, $r_1$ and $\rho_1$ satisfying
\begin{align} \label{cond:t_1-r_1}
2K_n(Q_{1,\alpha,n}+1)t_1+2\rho_1\sqrt{t_1}+r_1^n< \frac{1}{2^{n-1}} \beta^n,
\end{align}
the following holds
\begin{align*}
\cap_{i=1,2} &\big\{E_n(t_1;u_0^i)-K_n(Q_{1,\alpha,n}+1)t_1\le \Phi_\alpha (u_0^i)^n+\rho_1 \sqrt{t_1}   \big\}\\
&\subseteq \big\{ \Phi_\alpha (t;u_0^1)^n + \Phi_\alpha (t;u_0^2)^n \le \tfrac{1}{2^{n-1}}\beta^n \big\}\\
&\subseteq \big\{ \Phi_\alpha (t;u_0^1) + \Phi_\alpha (t;u_0^2) \le \beta \big\}.
\end{align*}
As a consequence, we deduce the inclusion
\begin{align*}
\{\tau_1\wedge \tau_2 \ge t_1\}\subseteq \big\{ \Phi_\alpha(t;u_0^1) + \Phi_\alpha(t;u_0^2) \le \beta \big\},
\end{align*}
whence
\begin{align*}
\P\big( \Phi(t;u_0^1) + \Phi_\alpha(t;u_0^2) \le \beta \big)\ge 1- \P(\tau_1\le t_1)-\P(\tau_2\le t_1).
\end{align*}
In view of Lemma \ref{lem:P(sup[Phi+int.Phi])}, cf. \eqref{ineq:P(sup_[0,T][Phi+int.Phi])}, for all $q>2$, we see that
\begin{align*}
\P(\tau_i\le t_1) & = \P\Big( \sup_{t\in[0,t_1]} E_{n}(t;u_0^i) -K_n(Q_{1,\alpha,n}+1)t\ge \Phi_\alpha(u_0^i)^n +\rho_1 \sqrt{t_1} \Big) \\
&\le \P\Big( \sup_{t\in[0,t_1]} E_{n}(t;u_0^i) -K_nQ_{1,\alpha,n}t\ge \Phi_\alpha(u_0^i)^n +\rho_1 \sqrt{t_1} \Big) \\
&\le \frac{K_{1,n,q}}{\rho^q}\big(\E \Phi_\alpha(u_0^i)^{nq}+Q^q_{2,\alpha,n} \big).
\end{align*}
It follows that
\begin{align*}
\P\big( \Phi(t;u_0^1) + \Phi_\alpha(t;u_0^2) \le \beta \big)\ge 1- \frac{K_{1,n,q}}{\rho_1 ^q}\big(r_1^{nq}+2Q^q_{2,\alpha,n} \big).
\end{align*}
In order to establish \eqref{ineq:ergodicity:P(Phi(u_1)+Phi(u_2)<beta)>1/2}, we pick $\rho_1$ satisfying
\begin{align*}
\rho_1 ^q = 2K_{1,n,q}\big(r_1^{nq}+2Q^q_{2,\alpha,n} \big).
\end{align*}
In addition, we simply choose $t_1=r_1^{2n}$ and take $r_1$ sufficiently small to guarantee the validity of \eqref{cond:t_1-r_1}. In turn, this allows one to conclude \eqref{ineq:ergodicity:P(Phi(u_1)+Phi(u_2)<beta)>1/2} as claimed.

Turning back to \eqref{ineq:ergodicity:P(ell(k+1)=k+1|l(k)=infty)>epsilon}, recalling Definition \ref{def:ell_beta} and Markov property, we observe that it suffices to establish the existence of a positive constants $T_1$ and $\varepsilon_1$ such that
\begin{align*}
\P\big( \Phi_\alpha (u(T_1;u_0^1))+\Phi_\alpha (u(T_1;u_0^2))\le \beta \big|  \Phi_\alpha (u_0^1)+\Phi_\alpha (u_0^2)\le R, \Phi_\alpha (u_0^1)+\Phi_\alpha (u_0^2)>\beta  \big)\ge \varepsilon_1.
\end{align*}
Let $t_1$ and $r_1$ be the constants from estimate \eqref{ineq:ergodicity:P(Phi(u_1)+Phi(u_2)<beta)>1/2}. Thanks to the irreducibility property from Lemma \ref{lem:irreducibility}, we may infer $T_*=T_*(R,r_1)$ and $\varepsilon_*=\varepsilon_*(R,r_1)$ such that
\begin{align*}
\P\big( \Phi_\alpha (u(T_*;u_0^1))+\Phi_\alpha (u(T_*;u_0^2))\le r_1 \big| \beta< \Phi_\alpha (u_0^1)+\Phi_\alpha (u_0^2)\le R \big)\ge \varepsilon_*.
\end{align*}
In view of estimate \eqref{ineq:ergodicity:P(Phi(u_1)+Phi(u_2)<beta)>1/2}, we set $T_1=T_*+t_1$ and observe that
\begin{align*}
& \P\big( \Phi_\alpha (u(T_1;u_0^1))+\Phi_\alpha (u(T_1;u_0^2))\le \beta \big| \beta< \Phi_\alpha (u_0^1)+\Phi_\alpha (u_0^2)\le R \big)\\
& = \P\big( \Phi_\alpha (u(T_1;u_0^1))+\Phi_\alpha (u(T_1;u_0^2))\le \beta \big|  \Phi_\alpha (u(T_*;u_0^1))+\Phi_\alpha (u(T_*;u_0^2))\le r_1 \big)\\
&\qquad\times\P\big( \Phi_\alpha (u(T_*;u_0^1))+\Phi_\alpha (u(T_*;u_0^2))\le r_1 \big| \beta< \Phi_\alpha (u_0^1)+\Phi_\alpha (u_0^2)\le R \big)\\
&\ge \frac{1}{2}\varepsilon_*=:\varepsilon_1.
\end{align*}
The proof is thus finished.

\end{proof}

Lastly, we provide the proof of Lemma \ref{lem:ergodicity:ell(k+1).neq.l|ell(k)=l)<1/2(1+(k-l)T)^-q}, which together with Lemma \ref{lem:ergodicity:P(d_1(u_1,u_2)>t^-q)} and \ref{lem:ergodicity:P(ell(k+1)=k+1|l(k)=infty)>epsilon} ultimately concludes Theorem \ref{thm:poly-mixing}.

\begin{proof}[Proof of Lemma \ref{lem:ergodicity:ell(k+1).neq.l|ell(k)=l)<1/2(1+(k-l)T)^-q}]

Without loss of generality, we may assume $l=0$ thanks to Markov property. From Definition \ref{def:ell_beta}, observe that 
\begin{align} \label{eqn:ell(k+1).neq.0|ell(k)=0}
\P\big(  \ell_{\theta,\beta}(k+1)\neq 0|\ell_{\theta,\beta}(k)=0  \big)& = \P\big(\ell_{\theta,\beta}(k+1)\neq 0|\ell_{\theta,\beta}(k)=0,\ell_{\theta,\beta}(0)=0  \big) \notag \\
&= \frac{  \P\big(\ell_{\theta,\beta}(k+1)\neq 0,\ell_{\theta,\beta}(k)=0|\ell_{\theta,\beta}(0)=0  \big) }{\P\big(\ell_{\theta,\beta}(k)=0|\ell_{\theta,\beta}(0)=0  \big)}.
\end{align} 
In order to produce inequality \eqref{ineq:ergodicity:ell(k+1).neq.l|ell(k)=l)<1/2(1+(k-l)T)^-q}, it suffices to derive an estimate on the numerator on the right-hand side of \eqref{eqn:ell(k+1).neq.0|ell(k)=0} while controlling the denominator away from below by zero. Regarding the former, we introduce the stopping times $\hat{\tau}_i$, $i=1,2,$
\begin{align*}
\hat{\tau}_i = 
\inf\big\{t\ge kT_1: E_n(t;u_0^i)-K_n(Q_{1,\alpha,n}+1)t\ge \theta+\beta^n \big\}.
\end{align*}
Observe that
\begin{align} \label{eqn:eqn:ell(k+1).neq.0|ell(k)=0:tauhat}
\big\{\ell_{\theta,\beta}(k+1)\neq 0,\ell_{\theta,\beta}(k)=0|\ell_{\theta,\beta}(0)=0  \big\} \subseteq \bigcup_{i=1,2}\{\hat{\tau}_i\le T_1|\ell_{\theta,\beta}(0)=0  \} .
\end{align}
Also, note that $\{\ell_{\theta,\beta}(0)=0 \} = \{ \Phi_\alpha(u_0^1) +\Phi_\alpha(u_0^2)  \le \beta \}$, by virtue of Definition \ref{def:ell_beta}. It follows that
\begin{align*}
&\P\big(\hat{\tau}_i\le T_1|\ell_{\theta,\beta}(0)=0  \big) \\
&= \P\big(  \sup_{t\in[kT_1,(k+1)T_1]} E_n(t;u_0^i)-K_n(Q_{1,\alpha,n}+1)t\ge \theta+\beta^n | \ell_{\theta,\beta}(0)=0 \big)\\
&\le  \P\big(  \sup_{t\in[kT_1,(k+1)T_1]} E_n(t;u_0^i)-K_n(Q_{1,\alpha,n}+1)t\ge \theta+\Phi_\alpha(u(t;u_0^i))^n | \ell_{\theta,\beta}(0)=0 \big).
\end{align*}
At this point, there are two cases to be considered depending on the value of $k$. On the one hand, when $k=0$, in light of Lemma \ref{lem:P(sup[Phi+int.Phi])}, cf. \eqref{ineq:P(sup_[0,T][Phi+int.Phi])} with $\rho=\theta/\sqrt{T_1}$, we obtain
\begin{align*}
&\P\big(\hat{\tau}_i\le T_1|\ell_{\theta,\beta}(0)=0  \big) \\
&\le  \P\big(  \sup_{t\in[0,T_1]} E_n(t;u_0^i)-K_nQ_{1,\alpha,n}t\ge \theta+\Phi_\alpha(u(t;u_0^i))^n | \ell_{\theta,\beta}(0)=0 \big)\\
&\le  \frac{K_{1,\alpha,n} T_1^{q/2}}{\theta^{q}}\big( \E\big[ \Phi_\alpha(u_0^i)^{nq}|    \ell_{\theta,\beta}(0)=0  \big]+Q^q_{2,\alpha,n} \big)\\
&\le  \frac{K_{1,\alpha,n} T_1^{q/2}}{\theta^{q}}\big( \beta^{nq}+Q^q_{2,\alpha,n} \big).
\end{align*}
It follows that for all $\theta=\theta(T_1,\beta,Q_{2,\alpha,n})$ large enough,
\begin{align*}
&\P\big(\hat{\tau}_i\le T_1|\ell_{\theta,\beta}(0)=0  \big) \le \frac{1}{8}.
\end{align*}
On the other hand, when $k\ge 1$, we invoke \eqref{ineq:P(sup_[t>T][Phi+int.Phi])} to deduce
\begin{align*}
&\P\big(\hat{\tau}_i\le T_1|\ell_{\theta,\beta}(0)=0  \big) \\
&\le  \P\big(  \sup_{t\in[kT_1,\infty)} E_n(t;u_0^i)-K_n(Q_{1,\alpha,n}+1)t\ge \theta+\Phi_\alpha(u(t;u_0^i))^n | \ell_{\theta,\beta}(0)=0 \big)\\
& \le \frac{K_{2,\alpha,n}}{(kT_1+\theta)^{q/2-1}}\big(\E\big[ \Phi_\alpha(u_0^i)^{nq}|    \ell_{\theta,\beta}(0)=0  \big]+Q^q_{2,\alpha,n} +1 \big)\\
&\le  \frac{K_{2,\alpha,n}}{(kT_1+\theta)^{q/2-1}}\big(\beta^{nq}+Q^q_{2,\alpha,n} +1 \big) .
\end{align*}
By taking $\theta$ further to infinity larger independent of $k$, we also infer
\begin{align*}
\P\big(\hat{\tau}_i\le T_1|\ell_{\theta,\beta}(0)=0  \big) \le \frac{1}{8(1+kT_1)^{q/2-2}}.
\end{align*}
From both cases, we arrive at the bound
\begin{align*}
\P\big(\hat{\tau}_i\le T_1|\ell_{\theta,\beta}(0)=0  \big) \le \frac{1}{8(1+kT_1)^{q/2-2}},\quad k=0,1,2,\dots
\end{align*}
Since the above estimate holds for arbitrary $q>2$, in light of expression \eqref{eqn:eqn:ell(k+1).neq.0|ell(k)=0:tauhat}, we get for all $q>4$,
\begin{align} \label{ineq:P(ell(k+1).neq.0,ell(k)=0|ell(0)=0)<1/4.T}
\P\big(\ell_{\theta,\beta}(k+1)\neq 0,\ell_{\theta,\beta}(k)=0|\ell_{\theta,\beta}(0)=0 \big)\le \frac{1}{4(1+kT_1)^{q+1}},\quad k=0,1,2,\dots
\end{align}

With regard to the denominator on the right-hand side of \eqref{eqn:ell(k+1).neq.0|ell(k)=0}, we claim that 
\begin{align} \label{ineq:P(ell(k)=0|ell(0)=0)>1/2}
\P\big(\ell_{\theta,\beta}(k)=0|\ell_{\theta,\beta}(0)=0  \big)\ge \frac{1}{2},\quad k\ge 0.
\end{align}
The case $k=0$ is trivial as the left-hand side probability is immediately one. Considering $k\ge 1$, we adopt a complement strategy as follows:
\begin{align*}
\big\{\ell_{\theta,\beta}(k)\neq 0|\ell_{\theta,\beta}(0)=0  \big\} \subseteq\bigcup_{i=0}^{k-1} \big\{\ell_{\theta,\beta}(i+1)\neq 0,\ell_{\theta,\beta}(i)=0|\ell_{\theta,\beta}(0)=0  \big\}.
\end{align*}
In turn, estimate \eqref{ineq:P(ell(k+1).neq.0,ell(k)=0|ell(0)=0)<1/4.T} implies that
\begin{align*}
\P\big(\ell_{\theta,\beta}(k)\neq 0|\ell_{\theta,\beta}(0)=0   \big) \le \sum_{i=0}^{k-1}\frac{1}{4(1+kT_1)^{q}}
&\le \frac{1}{4}+\sum_{i=1}^\infty \frac{1}{4(1+kT_1)^{q+1}} \\
&\le \frac{1}{4}+\int_0^\infty\frac{1}{(1+yT_1)^{q+1}}\d y\\
&= \frac{1}{4}+\frac{1}{qT_1}.
\end{align*}
Since $q\ge 4$ and $T_1\ge T_*\ge 1$, cf. \eqref{form:T^*}, we obtain
\begin{align*}
\P\big(\ell_{\theta,\beta}(k)\neq 0|\ell_{\theta,\beta}(0)=0   \big)\le \frac{1}{2},
\end{align*}
which implies \eqref{ineq:P(ell(k)=0|ell(0)=0)>1/2}, as claimed. 

Finally, we combine \eqref{ineq:P(ell(k+1).neq.0,ell(k)=0|ell(0)=0)<1/4.T} and \eqref{ineq:P(ell(k)=0|ell(0)=0)>1/2} together with expression \eqref{eqn:ell(k+1).neq.0|ell(k)=0} to produce the bound
\begin{align*}
\P\big(  \ell_{\theta,\beta}(k+1)\neq 0|\ell_{\theta,\beta}(k)=0  \big) \le  \frac{1}{2(1+kT_1)^{q}}, \quad k=0,1,2,\dots
\end{align*}
The proof is thus finished.

\end{proof}

\begin{remark} \label{remark:ergodicity:Girsanov} We note that the proofs of Lemmas \ref{lem:ergodicity:P(ell(k+1)=k+1|l(k)=infty)>epsilon} and \ref{lem:ergodicity:ell(k+1).neq.l|ell(k)=l)<1/2(1+(k-l)T)^-q} do not involve Girsanov Theorem that plays a central role in the coupling arguments of \cite{debussche2005ergodicity,gao2024polynomial,
nersesyan2024exponential,nersesyan2024polynomial, nguyen2024inviscid,
odasso2006ergodicity}. This is because in these works, different Foias-Prodi estimates are used to control the growth rate of the high modes assuming that the two solutions' low modes agree over time. Simultaneously, the Girsanov Theorem is employed to guarantee that one can tune the coupling argument so as to ensure the low modes staying the same. In our context, since we rely on the large damping condition, we are able to bypass the changing of measures, thereby significantly simplifying the proof while successfully achieving a similar coupling effect.
    
\end{remark}

\appendix

\section{Strichartz Estimates} \label{sec:Strichartz}

In this section, we collect useful pathwise estimates on the Schr\"odinger semigroup $S(t)$ that we have employed to establish the main results in the previous section. We start with Lemmas \ref{lem:Strichartz:S(t)u} and \ref{lem:Strichartz} giving deterministic Strichartz inequalities, whose proofs can be found in \cite{cazenave2003semilinear}.

\begin{lemma}{\cite[Corollary 2.2.6]{cazenave2003semilinear}} \label{lem:Strichartz:S(t)u}  If $t\neq 0$, the for all $s\in\rbb$, $p\in[2,\infty]$,
\begin{align} \label{ineq:Strichartz:S(t)u}
    \|S(t) u\|_{H^{s,p}}\le \frac{1}{(4\pi |t|)^{d(\frac{1}{2}-\frac{1}{p})}}\|u\|_{H^{s,p'}},
\end{align}
    where $p'$ denotes the Holder conjugate of $p$, i.e., $\frac{1}{p}+\frac{1}{p'}=1$.
\end{lemma}

\begin{definition}{\cite[Definition 2.3.1]{cazenave2003semilinear}} \label{def:admissble-pair} A pair $(\gamma,r)$ is called \textup{admissible} if
\begin{align*}
    \frac{2}{\gamma}+\frac{d}{r}=\frac{d}{2},\quad \text{and}\quad (\gamma,r)\neq (2,\infty),
\end{align*}
    and
    \begin{align*}
        \begin{cases}
            2\le r\le \infty, & d=1,\\
            2\le r<\infty,&d=2,\\
            2\le r \le \frac{2d}{d-2},& d\ge 3.
        \end{cases}
    \end{align*}
\end{definition}

\begin{lemma}{\cite[Theorem 2.3.3]{cazenave2003semilinear}} \label{lem:Strichartz}
Let $(p,r)$ be an admissible pair as in Definition \ref{def:admissble-pair}. Then, the following properties hold:

1. For every $\varphi\in L^2(\rbb^d) $, the function $t\mapsto S(t)=e^{-it\triangle}\varphi$ belongs to $L^p(\rbb;L^r(\rbb^d))\cap C(\rbb;L^2(\rbb^d)) $. Furthermore, there exists a positive constant $C$ such that
\begin{align} \label{ineq:Strichartz:u_0:deterministic}
\|S(\cdot)\varphi\|_{L^p(\rbb;L^r(\rbb^d))} \le C\|\varphi\|_{L^2(\rbb^d)}.
\end{align}

2. Let $I$ be an interval in $\rbb$ such that $0\in J= \bar{I}$. If $(\gamma,\rho)$ is an admissible pair and $f\in L^{\gamma'}(I;L^{\rho'}(\rbb^d))$, then the function $t\mapsto  G_f(t)=\int_0^t S(t-s)f(s)\textup{d} s$ belongs to $L^p(I;L^{r}(\rbb^d))\cap C(J;L^2(\rbb^d))$. Furthermore, there exists a positive constant $C$ independent of $I$ and $f$ such that
\begin{align} \label{ineq:Strichartz:int:deterministic}
    \|G_f\|_{ L^p(I;L^{r}(\rbb^d)) }\le C \|f\|_{L^{\gamma'}(I;L^{\rho'}(\rbb^d))}.
\end{align}
    
\end{lemma}

Next, we state and prove Lemma \ref{lem:Strichart:|u|_infty} providing useful pathwise estimates on the semigroup $S(t)$ on any time interval $[0,T]$. We note that the results of Lemma \ref{lem:Strichart:|u|_infty} appeared in the proofs of Lemma \ref{lem:Strichart:int_0^t|u|_infty^2sigma} and the irreducibility condition in Lemma \ref{lem:irreducibility}. Altogether, they are invoked to establish Theorem \ref{thm:poly-mixing} in Section \ref{sec:poly-mixing}.  

\begin{lemma} \label{lem:Strichart:|u|_infty} Denote $F(u)=|u|^{2\sigma}u$. In dimension $d=2,3$, for all $u_0\in H^1$, $u(\cdot)\in L^2((0,T);H^1)$, $T>0$, $\lambda\ge 1$, there exists a positive constant $C=C(\sigma)$ independent of $u_0$, $u$, $T$ and $\lambda$ such that the followings hold:

1. For all $\sigma>0$, 
\begin{align} \label{ineq:Strichart:|u_0|_infty}
\int_0^T \big\|e^{-\lambda t}S(t)u_0\big\|^{2\sigma}_{L^\infty}\d t\le C\|u_0\|^{2\sigma}_{H^1}.
\end{align}

2. Suppose that
\begin{align}
    \begin{cases}
        \sigma >0,& d=2,\\
        \sigma\in(1/6,3/2) ,& d=3.
    \end{cases}
\end{align}
There exists a positive constant $q_\sigma >2$ such that 
\begin{align} \label{ineq:Strichart:|u|_infty}
\int_0^T \Big\|\int_0^t e^{-\lambda(t-s)} S(t-s)F(u(s))\d s\Big\|_{L^\infty}^{2\sigma} \d t\le  C\Big(T+\int_0^T \|u(s)\|^{q_\sigma}_{H^1}\d s\Big).
\end{align}

\end{lemma}

\begin{remark} \label{rem:S(t)F}
    Following closely the proof of \cite[Proposition 3.1]{brzezniak2023ergodic}, one can also obtain the bound
    \begin{align} \label{ineq:Strichart:|u|_infty:BBZ}
        \int_0^T \Big\|\int_0^t S(t-s)F(u(s))\d s\Big\|_{L^\infty}^{2\sigma} \d t\le  C\Big(T^{1+\varepsilon}+\int_0^T \|u(s)\|^{q_\sigma}_{H^1}\d s\Big),
\end{align}
for some positive constant $\varepsilon>0$. However, as presented in Section \ref{sec:poly-mixing:proof-main-result}, the proof of the mixing rate requires a linear growth in time on the $L^\infty$ norm of the semigroup $S(t)$ acting on $F$. This turns out to be possible by exploiting the nature of the damping effect, hence the result of \eqref{ineq:Strichart:|u|_infty}. 

\end{remark}

In order to prove Lemma \ref{lem:Strichart:|u|_infty}, it is crucial to produce suitable bounds on the nonlinearity $F(u)=|u|^{2\sigma}u$. For this purpose, we assert a relation between $\||u|^{2\sigma}u\|_{H^{1,p}}$ and $\|u\|_{H^1}$ below in Lemma \ref{lem:F_alpha}, whose proof is a slightly rework of that of \cite[Lemma C.1]{brzezniak2023ergodic}

\begin{lemma} \label{lem:F_alpha} 

    1. In dimension $d=2$, for all $\sigma>0$ and $p$ satisfying
    \begin{align*}
        p\in[1,2)\quad \text{and} \quad p\ge \frac{2}{2\sigma+1},
    \end{align*}
the following holds
\begin{align} \label{ineq:F(u):d=2}
    \|F(u)\|_{H^{1,p}(\rbb^2)} \le C\|u\|^{2\sigma+1}_{H^1(\rbb^2)},\quad u\in H^1(\rbb^2),
\end{align}
for some positive constant $C$ independent of $u$.

2. In dimension $d=3$, for all $\sigma \in(0,\frac{3}{2}]$ and $p$ satisfying
\begin{align*}
        p\in[1,2)\quad \text{and} \quad  \frac{2}{2\sigma+1}\le p \le \frac{6}{2\sigma+3},
    \end{align*}
it holds that
\begin{align} \label{ineq:F(u):d=3}
    \|F(u)\|_{H^{1,\frac{6}{2\sigma+3}}(\rbb^3)} \le C\|u\|^{2\sigma+1}_{H^1(\rbb^3)},\quad u\in H^1(\rbb^3).
\end{align}
\end{lemma}
\begin{proof} 1. Considering dimension $d=2$, by a routine computation, we have
\begin{align*}
    \grad F(u) = (\sigma +1)|u|^{2\sigma}\grad u + \sigma |u|^{2\sigma-2}u^2\grad \bar{u},
\end{align*}
whence
\begin{align*}
    |\grad F(u)| \le (1+2\sigma)|u|^{2\sigma}|\grad u|.
\end{align*}
We invoke Holder inequality to deduce that for $p\in [1,2)$
\begin{align*}
    \|F(u)\|_{H^{1,p}}\le C( \|F(u)\|_{L^p} + \|\grad F(u)\|_{L^p} ) \le C\big( \|u\|^{1+2\sigma}_{L^{p(1+2\sigma)}}+ \|u\|^{2\sigma }_{L^{\frac{4\sigma p}{2-p}}}\|\grad u\|_{L^2}  \big).
\end{align*}
Recall that in dimension $d=2$, $H^1\subset L^q$ for all $q\ge 2$. It follows that 
\begin{align*}
    \|u\|^{1+2\sigma}_{L^{p(1+2\sigma)}}+ \|u\|^{2\sigma }_{L^{\frac{4\sigma p}{2-p}}}\|\grad u\|_{L^2} \le C \|u\|^{1+2\sigma}_{H^1},
\end{align*}
provided that
\begin{align*}
    p(1+2\sigma) \ge 2\quad \text{and}\quad  \frac{4\sigma p}{2-p} \ge 2.
\end{align*}
Note that the above condition are equivalent to 
\begin{align*}
    p\ge \frac{2}{1+2\sigma}.
\end{align*}
Altogether, we establish \eqref{ineq:F(u):d=2}, as claimed.

2. Turning to dimension $d=3$, recall the embedding $H^1\subset L^q$, $q\in[2,6]$. As a consequence, the following holds
\begin{align*}
    \|u\|^{1+2\sigma}_{L^{p(1+2\sigma)}}+ \|u\|^{2\sigma }_{L^{\frac{4\sigma p}{2-p}}}\|\grad u\|_{L^2} \le C \|u\|^{1+2\sigma}_{H^1},
\end{align*}
 provided that
 \begin{align*}
     2\le p(1+2\sigma)\le 6,\quad \text{and}\quad 2\le \frac{4\sigma p}{2-p} \le 6.
 \end{align*}
The above inequalities are equivalent to
\begin{align*}
    \frac{2}{2\sigma+1}\le p\le \frac{6}{2\sigma+1}\quad \text{and}\quad \frac{2}{2\sigma+1}\le p\le \frac{6}{2\sigma+3}.
\end{align*}
Since the latter estimate implies the former, we deduce that
\begin{align*}
    \frac{2}{2\sigma+1}\le p\le \frac{6}{2\sigma+3}.
\end{align*}
 Note that the existence of such a $p\in[1,2)$ verifying the above condition is only possible for all $\sigma\in(0,\frac{3}{2}]$. In turn, this produces \eqref{ineq:F(u):d=3}, thereby finishing the proof.  
\end{proof}

Now, we combine the auxiliary results from Lemmas \ref{lem:Strichartz} and \ref{lem:F_alpha} to conclude Lemma \ref{lem:Strichart:|u|_infty}.

\begin{proof}[Proof of Lemma \ref{lem:Strichart:|u|_infty}]

1. We start with dimension $d=2$ and consider an arbitrary admissible pair $(\gamma,r)$ satisfying $\gamma\in(2\sigma,\infty)$. In particular, Definition \ref{def:admissble-pair} implies $r>2$. Recalling the Sobolev embedding $H^{1,q}\subset L^\infty$ for all $q>2$, we employ Holder's inequality and \eqref{ineq:Strichartz:u_0:deterministic} to estimate the left-hand side of \eqref{ineq:Strichart:|u_0|_infty} as follows.
\begin{align*}
    \int_0^T \|e^{-\lambda t}S(t) u_0\|^{2\sigma}_{L^\infty}\d t& \le C\int_0^T e^{-2\sigma\lambda t}\|S(t) u_0\|^{2\sigma}_{H^{1,r}}\d t\\
    &\le C \Big(\int_0^T e^{-2\sigma\frac{\gamma}{\gamma-2\sigma} \lambda t}\d t\Big)^{\frac{\gamma-2\sigma}{\gamma}} \Big( \int_0^T \|S(t)u_0\|_{H^{1,r}}^{\gamma} \Big)^{\frac{2\sigma}{\gamma}}\\
    &\le C\Big(\frac{\gamma-2\sigma}{2\sigma\gamma\lambda}\Big)^{\frac{\gamma-2\sigma}{\gamma}}\|u_0\|_{H^1}^{2\sigma}\\
    &\le C \|u_0\|_{H^1}^{2\sigma}.
\end{align*}
Since $C$ is positive constant independent of $T>0$ and $\lambda\ge 1$, we obtain \eqref{ineq:Strichart:|u_0|_infty} in dimension $d=2$.

Turning to dimension $d=3$, we recall the embedding $H^{1,q}\subset L^\infty$ for all $q>3$. Let $(\gamma,r)$ be an admissible pair such that $r>3$ and $\gamma>2\sigma$. It is important to note that in view of Definition \ref{def:admissble-pair}, $r>3$ is equivalent to $\gamma<4$. Thanks to the assumption $\sigma\in(0,2)$, it follows that such a pair $(\gamma,r)$ is always guaranteed to exist. We then may employ the same argument as in the case of $d=2$ while making use of Holder's inequality and \eqref{ineq:Strichartz:u_0:deterministic} to produce the bound
\begin{align*}
    \int_0^T \|e^{-\lambda t}S(t) u_0\|^{2\sigma}_{L^\infty}\d t
    &\le C \|u_0\|_{H^1}^{2\sigma}.
\end{align*}
This conclude \eqref{ineq:Strichart:|u_0|_infty} in dimension $d=3$, thereby completing part 1 of Lemma \ref{lem:Strichart:|u|_infty}.

2. With regard to \eqref{ineq:Strichart:|u|_infty}, for notational convenience, we denote 
\begin{align*}
    G(t) =\int_0^t  e^{-\lambda(t-s)}S(t-s)F(u(s))\d s.
\end{align*} 
There are two cases to be considered depending on the dimension $d$.

Case 1: $d=2$. 
Let $(\gamma,r)$ be a given admissible pair and be chosen later. We claim that one can tune $(\gamma,r)$  appropriately so as to produce a constant $q_\sigma>0$ satisfying \eqref{ineq:Strichart:|u|_infty}. To see this, we recall the Sobolev embedding $H^{1,r}\subset L^\infty$ for all $r>d=2$. Then, we invoke Lemma \ref{lem:Strichartz:S(t)u} (with $p=r>2$) to estimate $G(t)$ as follows.
\begin{align*}
    \|G(t)\|_{L^\infty}& = \Big\|\int_0^t  e^{-\lambda(t-s)}S(t-s)F(u(s))\d s\Big\|_{L^\infty}\\
    &\le \int_0^t  e^{-\lambda(t-s)}\|S(t-s)F(u(s))\|_{L^\infty}\d s\\
    &\le C \int_0^t  e^{-\lambda(t-s)} \|S(t-s)F(u(s))\|_{H^{1,r}}\d s\\
    &\leq  C\int_0^t  e^{-\lambda(t-s)} \frac{1}{|t-s|^{\frac{2}{\gamma}}}\|F(u(s))\|_{H^{1,r'}}\d s,
\end{align*}
whence,
\begin{align*}
     \|G(t)\|_{L^\infty}^{2\sigma}\le C\Big(\int_0^t  e^{-\lambda(t-s)} \frac{1}{|t-s|^{\frac{2}{\gamma}}}\|F(u(s))\|_{H^{1,r'}}\d s\Big)^{2\sigma}.
\end{align*}
At this point, there are two sub cases depending on the values of $\sigma$. 

Case 1a: $2\sigma\le 1$. In this case, we employ Young's and Holder's inequalities to infer
\begin{align*}
    &\Big(\int_0^t  e^{-\lambda(t-s)} \frac{1}{|t-s|^{\frac{2}{\gamma}}}\|F(u(s))\|_{H^{1,r'}}\d s\Big)^{2\sigma}\\
    &\le C+C\int_0^t  e^{-\lambda(t-s)} \frac{1}{|t-s|^{\frac{2}{\gamma}}}\|F(u(s))\|_{H^{1,r'}}\d s\\
    &\le C+ C\int_0^t  e^{-\lambda(t-s)} \frac{1}{|t-s|^{\frac{4}{\gamma}}}\d s+C\int_0^t e^{-\lambda(t-s)}\|F(u(s))\|^2_{H^{1,r'}}\d s.
\end{align*}
On the one hand, note that if $\gamma>4$, we have the estimate (recalling $\lambda\ge 1$)
\begin{align*}
    \int_0^t  e^{-\lambda(t-s)} \frac{1}{|t-s|^{\frac{4}{\gamma}}}\d s \le \int_0^\infty  e^{-\lambda s} \frac{1}{|s|^{\frac{4}{\gamma}}}\d s\le   \frac{1}{\lambda}+\frac{1}{1-\frac{4}{\gamma}}\le 1+\frac{1}{1-\frac{4}{\gamma}}.
\end{align*}
On the other hand, if $r'\in[1,2)$ satisfying $r'\ge 2/(1+2\sigma)$, we may invoke \eqref{ineq:F(u):d=2} to infer
\begin{align*}
    \int_0^t e^{-\lambda(t-s)}\|F(u(s))\|^2_{H^{1,r'}}\d s \le \int_0^t e^{-\lambda(t-s)}\|u(s)\|^{2(1+2\sigma)}_{H^{1}}\d s.
\end{align*}
It follows that
\begin{align*}
   \|G(t)\|^{2\sigma}_{L^\infty} &\le C\Big(\int_0^t  e^{-\lambda(t-s)} \frac{1}{|t-s|^{\frac{2}{\gamma}}}\|F(u(s))\|_{H^{1,r'}}\d s\Big)^{2\sigma} \\
    &\le C\Big(1+\frac{1}{1-\frac{4}{\gamma}}\Big)+C\int_0^t e^{-\lambda(t-s)}\|u(s)\|^{2(1+2\sigma)}_{H^{1}}\d s,
\end{align*}
provided the following conditions are met
\begin{align*}
    \gamma>4,\quad r>2,\quad \frac{1}{\gamma}+\frac{1}{r}=\frac{1}{2},\quad r'\in[1,2),\quad r'\ge \frac{2}{2\sigma+1}.
\end{align*}

\noindent

\noindent 
Since $2\sigma\in(0,1]$, observe that the choice
\begin{align*}
    r'=\frac{1}{2}\Big( \frac{2}{2\sigma+1}+2 \Big)=\frac{2\sigma+2}{2\sigma+1},\quad r=2\sigma+2,\quad \gamma = 4+\frac{2}{\sigma}
\end{align*}
verify these conditions.

Case 1b: $2\sigma>1$. In this case, we employ Holder's inequality to infer
\begin{align*}
    & \Big(\int_0^t  e^{-\lambda(t-s)} \frac{1}{|t-s|^{\frac{2}{\gamma}}}\|F(u(s))\|_{H^{1,r'}}\d s\Big)^{2\sigma}\\
     &=\Big(\int_0^t  e^{-\frac{1}{2}\lambda(t-s)} \frac{1}{|t-s|^{\frac{2}{\gamma}}}\cdot e^{-\frac{1}{2}\lambda(t-s)} \|F(u(s))\|_{H^{1,r'}}\d s\Big)^{2\sigma}\\
    & \le \Big(\int_0^t e^{-\frac{\sigma}{2\sigma-1}\lambda(t-s) }\frac{1}{|t-s|^{\frac{4\sigma}{\gamma(2\sigma-1)}}}\d s\Big)^{2\sigma-1}\int_0^t e^{-\sigma\lambda(t-s)}\|F(u(s))\|^{2\sigma}_{H^{1,r'}}\d s.
\end{align*}
Similar to Case 1a, on the one hand, if $\frac{4\sigma}{\gamma(2\sigma-1)}<1$,
\begin{align*}
    \Big(\int_0^t e^{-\frac{\sigma}{2\sigma-1}\lambda(t-s) }\frac{1}{|t-s|^{\frac{4\sigma}{\gamma(2\sigma-1)}}}\d s\Big)^{2\sigma-1}\le \Big( 
 \frac{2\sigma-1}{\sigma\lambda}+\frac{\gamma(2\sigma-1)}{4\sigma}\Big)^{2\sigma-1}.
\end{align*}
On the other hand, from \eqref{ineq:F(u):d=2}, if $r'\in[1,2)$ such that $r'\ge \frac{2}{2\sigma+1}$,
\begin{align*}
    \int_0^t e^{-\sigma\lambda(t-s)}\|F(u(s))\|^{2\sigma}_{H^{1,r'}}\d s\le \int_0^t e^{-\sigma\lambda(t-s)}\|u(s)\|^{2\sigma(1+2\sigma)}_{H^{1}}\d s.
\end{align*} 
It follows that
\begin{align*}
   \|G(t)\|^{2\sigma}_{L^\infty} &\le C\Big(\int_0^t  e^{-\lambda(t-s)} \frac{1}{|t-s|^{\frac{2}{\gamma}}}\|F(u(s))\|_{H^{1,r'}}\d s\Big)^{2\sigma} \\
    &\le C\Big( 
 \frac{2\sigma-1}{\sigma}+\frac{\gamma(2\sigma-1)}{4\sigma}\Big)^{2\sigma-1}\int_0^t e^{-\sigma\lambda(t-s)}\|u(s)\|^{2\sigma(1+2\sigma)}_{H^{1}}\d s,
\end{align*}
provided the following conditions are met
\begin{align*}
    \gamma>\frac{4\sigma}{2\sigma-1},\quad r>2,\quad \frac{1}{\gamma}+\frac{1}{r}=\frac{1}{2},\quad r'\in[1,2),\quad r'\ge \frac{2}{2\sigma+1}.
\end{align*} 
Since $2\sigma>1$, observe that the above holds for the instance 
\begin{align*}
    \gamma= \frac{4\sigma+1}{2\sigma-1},\quad r=\frac{2(4\sigma+1)}{3},\quad r'=\frac{2(4\sigma+1)}{8\sigma-1}.
\end{align*}
Now, we combine two cases to infer a positive constant $q_1=q_1(\sigma)>0$ and $q_2=q_2(\sigma)>2$ such that
\begin{align*}
     \|G(t)\|^{2\sigma}_{L^\infty}\le C+C\int_0^t e^{-q_1\lambda(t-s)}\|u(s)\|^{q_2}_{H^1}\d s.
\end{align*}
As a consequence, it holds that
\begin{align*}
    \int_0^T \|G(t)\|^{2\sigma}_{L^\infty}\d t &\le CT+C\int_0^T \int_0^t e^{-q_1\lambda(t-s)}\|u(s)\|^{q_2}_{H^1}\d s\d t\\
    & = CT + \frac{C}{q_1\lambda}\int_0^T\|u(s)\|^{q_2}_{H^{1}}\d s\\
    &\le CT + C\int_0^T\|u(s)\|^{q_2}_{H^{1}}\d s.
\end{align*}
In the above, we emphasize that the constant $C$ does not depend on either $T$ or $\lambda$. This produces estimate \eqref{ineq:Strichart:|u|_infty} in dimension $d=2$, and thus completes Case 1.

Case 2: $d=3$. In this case, recall that $H^{\theta,q}\subset L^\infty$ for all $\theta q>3$. Similarly to the previous situation, there are three sub cases to be considered depending on the values of $\sigma\in(1/6,3/2)$. 

Case 2a: $\frac{1}{3}\le \sigma\le 1$. Firstly, we employ Holder's inequality and $H^{1,6}\subset L^\infty$ to infer
\begin{align*}
    \int_0^T \|G(t)\|^{2\sigma}_{L^\infty}\d t & \le T^{1-\sigma}\Big(\int_0^T\|G(t)\|^{2}_{H^{1,6}}\d t\Big)^{\sigma}.
\end{align*}
Since $(2,6)$ is already an admissible pair, we employ \noindent 
\eqref{ineq:Strichartz:int:deterministic} with $(\gamma,r)=(2,6)$ to obtain the bound
\begin{align*}
    \Big(\int_0^T\|G(t)\|^{2}_{H^{1,6}}\d t\Big)^{\sigma}\le \Big(\int_0^T\Big\|\int_0^t S(t-s)F(u(s))\d s\Big\|^{2}_{H^{1,6}}\d t\Big)^{\sigma}\le C\Big(\int_0^T\|F(u(t))\|^{2}_{H^{1,\frac{6}{5}}}\d t\Big)^{\sigma}.
\end{align*}
In view of Lemma \ref{lem:F_alpha}, c.f., estimate \eqref{ineq:F(u):d=3}, 
it holds that
\begin{align*}
    \|F(u(t))\|^{2}_{H^{1,\frac{6}{5}}} \le C\|u(t)\|^{2(1+2\sigma)}_{H^1},
\end{align*}
provided that
\begin{align*}
    \frac{2}{2\sigma+1}\le \frac{6}{5}\le \frac{6}{2\sigma+3}.
\end{align*}
Note that the above condition is equivalent to $1/3\le \sigma\le 1$.
It follows that
\begin{align*}
    \int_0^T \|G(t)\|^{2\sigma}_{L^\infty}\d t\le CT^{1-\sigma}\Big(\int_0^T\|u(t)\|^{2(1+2\sigma)}_{H^1}\d t\Big)^{\sigma}\le CT+C\int_0^T\|u(t)\|^{2(1+2\sigma)}_{H^1}\d t.
\end{align*}
In the last implication above, we employed Young's inequality thanks to the fact that $\sigma\le 1$. This produces \eqref{ineq:Strichart:|u|_infty} for Case 2a, $1/3\le \sigma\le 1$.

Case 2b: $1<\sigma<3/2$. Let $\theta\in(0,1)$ be given and be chosen later. Given $(\gamma,\rho)$ an admissible pair satisfying $\theta r>3$ and $(\gamma,r)\neq (2,6)$, we invoke Lemma \ref{lem:Strichartz:S(t)u} again to infer
\begin{align*}
    \|G(t)\|_{L^\infty}&\le C\int_0^t e^{-\lambda(t-s)}\|S(t-s)F(u(s))\|_{H^{\theta,r}}\d s\\
    &\le C\int_0^t e^{-\lambda(t-s)}\frac{1}{|t-s|^{\frac{2}{\gamma}}}\|F(u(s))\|_{H^{\theta,r'}}\d s.
\end{align*}

\noindent
To further estimate the above right-hand side, we recall the Sobolev embedding
\begin{align*}
    \frac{1}{p}-\frac{1}{r'}=\frac{1-\theta}{3},\quad H^{1,p}\subset H^{\theta,r'}.
\end{align*}
This together with \eqref{ineq:F(u):d=3} implies that
\begin{align*} 
    \|F(u(s))\|_{H^{\theta,r'}} \le C\|F(u(s))\|_{H^{1,p}}\le C\|u\|_{H^1}^{1+2\sigma},
\end{align*}
for $p\in[1,2)$ satisfying 
\begin{align*}
    \frac{2}{2\sigma+1}\le p\le \frac{6}{2\sigma+3}.
\end{align*}
As a consequence, we obtain
\begin{align} \label{ineq:G(t):2b:1<sigma<3/2}
   \|G(t)\|_{L^\infty} \le   C\int_0^t e^{-\lambda(t-s)}\frac{1}{|t-s|^{\frac{2}{\gamma}}}\|u(s)\|^{1+2\sigma}_{H^{1}}\d s.
\end{align}
We now proceed to tune the parameters $\theta$, $r'$ and $p$ so as to verify \eqref{ineq:G(t):2b:1<sigma<3/2}. To this end, for simplicity, we pick 
\begin{align*}
    p=\frac{6}{2\sigma+3},
\end{align*}
which satisfies $p\in[1,6/5)$ thanks to the assumption $\sigma\in(1,3/2)$. Also, 
\begin{align*}
    \frac{1}{r'}=\frac{1}{p}-\frac{1-\theta}{3}=\frac{2\sigma+3}{6}-\frac{1-\theta}{3}.
\end{align*}
Note that the requirement $3/\theta<r<6$ is equivalent to \begin{align*}
    \frac{3-\theta}{3}<\frac{1}{r'}<\frac{5}{6}.
\end{align*}
In other words,
\begin{align*}
   \frac{3-\theta}{3}<  \frac{2\sigma+3}{6}-\frac{1-\theta}{3}<\frac{5}{6},
\end{align*}
which is simplified to
\begin{align*}
   \frac{1}{2}\Big( \frac{3}{2}-\sigma\Big)<\theta -\frac{1}{2}<\frac{3}{2}-\sigma.
\end{align*}
Since $\sigma\in(1,3/2)$, we may choose 
\begin{align*}
    \theta=\frac{1}{2}+\frac{3}{4}\Big(\frac{3}{2}-\sigma\Big)\in(0,1),
\end{align*}
and set $r'$, $r$ and $\gamma$ accordingly through the relations
\begin{align*}
    \frac{1}{r'}=\frac{2\sigma+3}{6}-\frac{1-\theta}{3},\quad \frac{1}{r}+\frac{1}{r'}=1,\quad\text{and}\quad \frac{2}{\gamma}+\frac{3}{r}=\frac{3}{2}.
\end{align*}
In turn, this allows us to establish the validity of \eqref{ineq:G(t):2b:1<sigma<3/2}. Next, since $\gamma>2$, let $\varepsilon=\varepsilon(\gamma)>0$ be sufficiently small such that
\begin{align*}
  \frac{2}{\gamma}(1+\varepsilon)<1  \quad\text{and}\quad 2\sigma\varepsilon<1.
\end{align*}
We employ Holder inequality to further deduce that
\begin{align*}
     &  C\Big(\int_0^t e^{-\lambda(t-s)}\frac{1}{|t-s|^{\frac{2}{\gamma}}}\|u(s)\|^{1+2\sigma}_{H^{1}}\d s\Big)^{2\sigma}\\
     &\le C\Big(\int_0^t e^{-\frac{1}{2}(1+\varepsilon)\lambda(t-s)}\frac{1}{|t-s|^{\frac{2}{\gamma}(1+\varepsilon)}} \d s\Big)^{\frac{2\sigma}{1+\varepsilon}}\Big(\int_0^t e^{-\frac{1+\varepsilon}{2\varepsilon}\lambda(t-s)}
 \|u(s)\|^{\frac{1+\varepsilon}{\varepsilon}(1+2\sigma)}_{H^{1}}\d s\Big)^{\frac{2\sigma\varepsilon}{1+\varepsilon}}\\
 &\le C\Big(\frac{2}{(1+\varepsilon)\lambda}+\frac{1}{1-\frac{2}{\gamma}(1+\varepsilon)}\Big)^{\frac{2\sigma}{1+\varepsilon}}\Big(\int_0^t e^{-\frac{1+\varepsilon}{2\varepsilon}\lambda(t-s)}
 \|u(s)\|^{\frac{1+\varepsilon}{\varepsilon}(1+2\sigma)}_{H^{1}}\d s\Big)^{\frac{2\sigma\varepsilon}{1+\varepsilon}}.
\end{align*}
This together with \eqref{ineq:G(t):2b:1<sigma<3/2} implies the bound
\begin{align*}
    \int_0^T \|G(t)\|^{2\sigma}_{L^\infty}\d t & \le C \int_0^T \Big(\int_0^t e^{-\frac{1+\varepsilon}{2\varepsilon}\lambda(t-s)}
 \|u(s)\|^{\frac{1+\varepsilon}{\varepsilon}(1+2\sigma)}_{H^{1}}\d s\Big)^{\frac{2\sigma\varepsilon}{1+\varepsilon}}\d t\\
 &\le CT +C\int_0^T \int_0^t e^{-\frac{1+\varepsilon}{2\varepsilon}\lambda(t-s)}
 \|u(s)\|^{\frac{1+\varepsilon}{\varepsilon}(1+2\sigma)}_{H^{1}}\d s\d t\\
 &= CT + C\frac{2\varepsilon}{(1+\varepsilon)\lambda} \int_0^T
 \|u(t)\|^{\frac{1+\varepsilon}{\varepsilon}(1+2\sigma)}_{H^{1}}\d t.
\end{align*}
In the above, we emphasize that $C$ is a positive constant independent of $\lambda$ and $T$. We therefore establish \eqref{ineq:Strichart:|u|_infty} in Case 2b, $1<\sigma<3/2$.

Case 2c: $1/6<\sigma<1/3$. We employ an argument similarly to Case 2b and observe that \eqref{ineq:G(t):2b:1<sigma<3/2} holds provided that
\begin{align*}
    p\in[1,2),\quad \frac{2}{2\sigma+1}\le p\le \frac{6}{2\sigma+3},\quad \frac{3-\theta}{3}<\frac{1}{r'}=\frac{1}{p}-\frac{1-\theta}{3}<\frac{5}{6}.
\end{align*}
Since $\sigma\in(1/6,1/3)$, we firstly pick 
\begin{align}
    p=\frac{2}{2\sigma+1}\in \Big(\frac{6}{5},2\Big).
\end{align}
So, it is required that
\begin{align*}
    \frac{3-\theta}{3}< \frac{2\sigma+1}{2}-\frac{1-\theta}{3}<\frac{5}{6},
\end{align*}
which is equivalent to
\begin{align*}
    \frac{2}{3}(1-\theta)<\sigma-\frac{1}{6},\quad\text{and}\quad \sigma+\frac{\theta}{3}<\frac{2}{3}.
\end{align*}
Note that the latter condition is immediately satisfied since $\theta\in(0,1)$ and $\sigma<1/3$ while the former is only valid provided $\sigma>1/6$. So, for simplicity, we pick $\theta$ given by
\begin{align*}
    1-\theta=\sigma-\frac{1}{6},
\end{align*}
and choose $r'$, $r$ and $\gamma$ in accordance with
\begin{align*}
    \frac{1}{r'}=\frac{2\sigma+3}{6}-\frac{1-\theta}{3},\quad \frac{1}{r}+\frac{1}{r'}=1,\quad\text{and}\quad \frac{2}{\gamma}+\frac{3}{r}=\frac{3}{2}.
\end{align*}
This allows us to establish \eqref{ineq:G(t):2b:1<sigma<3/2}. Then, we proceed in a similar fashion as in Case 2b to ultimately deduce
\begin{align*}
    \int_0^T \|G(t)\|^{2\sigma}_{L^\infty}\d t & \le CT + C\frac{2\varepsilon}{(1+\varepsilon)\lambda} \int_0^T
 \|u(t)\|^{\frac{1+\varepsilon}{\varepsilon}(1+2\sigma)}_{H^{1}}\d t,
\end{align*}
for suitable positive constants $\varepsilon=\varepsilon(\sigma)$ and $C=C(\sigma)$ independent of $T>0$ and $\lambda>1$.

Altogether, we  conclude the desired estimate \eqref{ineq:Strichart:|u|_infty} for some constant $q_\sigma>2$. The proof is thus finished.

\end{proof}

\section{Auxiliary Results} \label{sec:aux-results}

Let $\Gamma(t)$ denote the stochastic convolution solving the equation
\begin{equation} \label{eqn:Schrodinger:Gamma}
\d \Gamma (t) + \i \triangle \Gamma(t)\d t + \lambda \Gamma(t)\d t = Q \d W(t),\quad \Gamma(0)=0.
\end{equation}

\begin{lemma} \label{lem:moment-bound:Gamma}
Let $\Gamma(t)$ be the stochastic convolution solving \eqref{eqn:Schrodinger:Gamma}. Then, the followings hold for all $n\ge 1$:
\begin{align} \label{ineq:E.int.|Gamma|^2n_H}
\E\int_0^t\|\Gamma(s)\|^{2n}_{H}\d s \le c \frac{\|Q\|^{2n}_{L_{\HS}(U;H)}}{\lambda^{n}}t,
\end{align}
\begin{align}\label{ineq:E.int.|triangle.Gamma|^2n_H}
\E\int_0^t\|\triangle\Gamma(s)\|^{2n}_{H}\d s \le c \frac{\|Q\|^{2n}_{L_{\HS}(U;H^2)}}{\lambda^{n}}t,
\end{align}
where $c$ is a positive constant independent of $t$ and $Q$. Furthermore,
\begin{align}\label{ineq:E.int.|Gamma|^2sigma_L^infty}
\E\int_0^t\|\Gamma(s)\|^{2\sigma}_{L^\infty}\d s \le c \|Q\|^{2\sigma}_{L_{\HS}(U;H^1)}t,
\end{align}
where $\sigma$ is as in Assumption \ref{cond:sigma}.

\end{lemma}
\begin{proof}
Firstly, from \eqref{eqn:Schrodinger:Gamma}, we apply It\^o's formula to $\|\Gamma(t)\|_{H}^2$ and obtain the identity
\begin{align*}
\d \|\Gamma\|^2_H +2\lambda \|\Gamma\|^2_H\d t= 2\Re\big(\la \Gamma,Q\d W\ra_H  \big)+\|Q\|^2_{\L_{\HS}(U;H)}\d t.
\end{align*}
Similarly, for $n>1$, we have
\begin{align*}
\d \|\Gamma\|^{2n}_H +2n \lambda \|\Gamma\|^{n}_H\d t&=n\|\Gamma\|^{2(n-1)}\|Q\|^2_{L_{\HS}(U;H)}\d t+ 2n \|\Gamma\|^{2(n-1)}_H\Re\big(\la \Gamma,Q\d W\ra_H  \big)\\
&\qquad + 2n(n-1) \|\Gamma\|^{2(n-2)}_H\sum_{k\ge 1}\big| \Re\big( \la \Gamma,Qe_k\ra_H \big) \big|^2\d t.
\end{align*}
Using a similar argument to \eqref{ineq:d|u|^2n_H}, we infer the estimate
\begin{align} \label{ineq:d|Gamma|^2n_H}
\d \|\Gamma\|^{2n}_H \le  -n\lambda \|\Gamma\|^{2n}_H\d t+ c\frac{\|Q\|^{2n}_{L_{\HS}(U;H)}}{\lambda^{n-1}}\d t + 2n\|\Gamma\|^{2(n-1)}_H\Re\big(\la \Gamma,Q\d W\ra_H\big).
\end{align}
In turn, this produces \eqref{ineq:E.int.|Gamma|^2n_H} by taking expectation on both sides of the above inequality.

Likewise, the following holds
\begin{align} \label{ineq:d|triangle.Gamma|^2n_H}
\d \|\triangle\Gamma\|^{2n}_H \le  -n\lambda \|\triangle \Gamma\|^{2n}_H\d t+ c\frac{\|Q\|^{2n}_{L_{\HS}(U;H^2)}}{\lambda^{n-1}}\d t + 2n\|\triangle \Gamma\|^{2(n-1)}_H\Re\big(\la \triangle \Gamma,\triangle Q\d W\ra_H\big).
\end{align}
By taking expectation, we obtain
\begin{align*}
\E\|\triangle \Gamma(t)\|^{2n}_H+n\lambda\int_0^t \E\|\triangle \Gamma(s)\|^{2n}_H\d s \le c \frac{\|Q\|^{2n}_{L_{\HS}(U;H^2)}}{\lambda^{n}}t,
\end{align*}
which implies \eqref{ineq:E.int.|triangle.Gamma|^2n_H}, as claimed.

Finally, the argument of \eqref{ineq:E.int.|Gamma|^2sigma_L^infty} can be found in the proof of \cite[Proposition 3.1]{brzezniak2023ergodic}, which employs the stochastic Strichartz estimate of \cite[Proposition 2]{hornung2018nonlinear}.

\end{proof}

Lastly, we collect a useful inequality in $\cbb$ through Lemma \ref{lem:|z|^(2-2sigma)z^2-|x|^(2-2sigma)<|x-z|^2sigma}. Particularly, we invoked Lemma \ref{lem:|z|^(2-2sigma)z^2-|x|^(2-2sigma)<|x-z|^2sigma} to establish the irreducibility condition in Lemma \ref{lem:irreducibility}.

\begin{lemma} \label{lem:|z|^(2-2sigma)z^2-|x|^(2-2sigma)<|x-z|^2sigma}
Let $z,x\in\cbb$ and $\sigma\ge 0$. Then, there exists a positive constant $c(\sigma)$ independent of $x$ and $z$ such that the followings hold

1. $\sigma\in[0,1/2]$:
\begin{align} \label{ineq:|z|^(2-2sigma)z^2-|x|^(2-2sigma)<|x-z|^2sigma:sigma<1/2}
|z|^{2\sigma-2}z^2-|x|^{2\sigma-2}x^2 \le c(\sigma) |z-x|^{2\sigma}.
\end{align}

2. $\sigma>1/2$:
\begin{align} \label{ineq:|z|^(2-2sigma)z^2-|x|^(2-2sigma)<|x-z|^2sigma:sigma>1/2}
|z|^{2\sigma-2}z^2-|x|^{2\sigma-2}x^2 \le c(\sigma)|z-x|( |x|^{2\sigma-1}+|z|^{2\sigma-1} ).
\end{align}
\end{lemma}
\begin{proof} Inequalities \eqref{ineq:|z|^(2-2sigma)z^2-|x|^(2-2sigma)<|x-z|^2sigma:sigma<1/2}-\eqref{ineq:|z|^(2-2sigma)z^2-|x|^(2-2sigma)<|x-z|^2sigma:sigma>1/2} are trivial if at least $x$ or $z$ is zero. Considering $x\neq 0$ and $z\neq 0$, let $\theta_z$ and $\theta_x$ respectively be the angles in $\cbb$ corresponding to $z$ and $x$, that is, $z=|z|e^{\i \theta_z}$ and $x=|x|e^{\i\theta_x}$. With regard to the case $\sigma\in[0,1/2]$, observe that
\begin{align*}
 \big||z|^{2\sigma-2}z^2-|x|^{2\sigma-2}x^2\big|^2 &=\big| |z|^{2\sigma}e^{\i 2\theta_z}-|x|^{2\sigma} e^{\i 2\theta_x} \Big|^2\\
& = |z|^{4\sigma}+|x|^{4\sigma}-2|x|^{2\sigma}|z|^{2\sigma}\cos\big(2(\theta_z-\theta_x)\big)\\
&= \big(|z|^{2\sigma}+|x|^{2\sigma}\big)^2-4|x|^{2\sigma}|z|^{2\sigma}\cos^2(\theta_z-\theta_x)
\end{align*}
Similarly, we have
\begin{align*}
|z-x |^{4\sigma}= \big| |z|e^{\i\theta_z}-|x|e^{\i \theta_x}  \big|^{4\sigma}= \big(|z|^2+|x|^2-2|x|\,|z|\cos(\theta_z-\theta_x)\big)^{2\sigma}.
\end{align*}
To establish \eqref{ineq:|z|^(2-2sigma)z^2-|x|^(2-2sigma)<|x-z|^2sigma:sigma<1/2}, it suffices to prove that
\begin{align*}
\big(|z|^{2\sigma}+|x|^{2\sigma}\big)^2-4|x|^{2\sigma}|z|^{2\sigma}\cos^2(\theta_z-\theta_x) \le 2^{3-2\sigma} \big(|z|^2+|x|^2-2|x|\,|z|\cos(\theta_z-\theta_x)\big)^{2\sigma},
\end{align*}
which is equivalent to
\begin{align} \label{ineq:|z|^(2-2sigma)z^2-|x|^(2-2sigma)<|x-z|^2sigma:sigma<1/2:a}
&\frac{(|z|^{2\sigma}+|x|^{2\sigma})^2}{(|z|^2+|x|^2)^{2\sigma}}-2^{2-2\sigma}\Big(\frac{2|x||z|}{|z|^2+|x|^2}\Big)^{2\sigma}\cos^2(\theta_z-\theta_x) \notag \\
&\qquad\le 2^{3-2\sigma} \Big(1-\frac{2|x||z|}{|z|^2+|x|^2}\cos(\theta_z-\theta_x)\Big)^{2\sigma}.
\end{align}
To see this, we note that since $0\le\sigma\le 1/2<1$, the following holds for all $a\in[0,1]$
\begin{align*}
a^{\sigma}+(1-a)^\sigma\le 2^{1-\sigma},
\end{align*}
whence
\begin{align*}
\frac{(|z|^{2\sigma}+|x|^{2\sigma})^2}{(|z|^2+|x|^2)^{2\sigma} }= \Big( \Big(\frac{|z|^2}{|z|^2+|x|^2}\Big)^{\sigma}+\Big(\frac{|x|^2}{|z|^2+|x|^2}\Big)^{\sigma} \Big)^2\le 2^{2-2\sigma}.
\end{align*}
It follows that \eqref{ineq:|z|^(2-2sigma)z^2-|x|^(2-2sigma)<|x-z|^2sigma:sigma<1/2:a} is reduced to 
\begin{align*}
1\le \Big(\frac{2|x||z|}{|z|^2+|x|^2}\Big)^{2\sigma}\cos^2(\theta_z-\theta_x) + 2 \Big(1-\frac{2|x||z|}{|z|^2+|x|^2}\cos(\theta_z-\theta_x)\Big)^{2\sigma}.
\end{align*}
To this end, we invoke the hypothesis that $\sigma\le 1/2$ to infer
\begin{align*}
&\Big(\frac{2|x||z|}{|z|^2+|x|^2}\Big)^{2\sigma}\cos^2(\theta_z-\theta_x) + 2 \Big(1-\frac{2|x||z|}{|z|^2+|x|^2}\cos(\theta_z-\theta_x)\Big)^{2\sigma} \\
&\ge \Big(\frac{2|x||z|}{|z|^2+|x|^2}\Big)^{2}\cos^2(\theta_z-\theta_x)+2 \Big(1-\frac{2|x||z|}{|z|^2+|x|^2}\cos(\theta_z-\theta_x)\Big) \\
&= \Big(\frac{2|x||z|}{|z|^2+|x|^2}\cos(\theta_z-\theta_x)-1\Big)^2+1\ge 1.
\end{align*}
This establishes \eqref{ineq:|z|^(2-2sigma)z^2-|x|^(2-2sigma)<|x-z|^2sigma:sigma<1/2:a}, thereby producing \eqref{ineq:|z|^(2-2sigma)z^2-|x|^(2-2sigma)<|x-z|^2sigma:sigma<1/2} in the case $\sigma\in[0,1/2]$.

Turning to the case $\sigma>1/2$, we recast the left-hand side of \eqref{ineq:|z|^(2-2sigma)z^2-|x|^(2-2sigma)<|x-z|^2sigma:sigma>1/2} as follow
\begin{align*}
& |z|^{2\sigma-2}z^2-|x|^{2\sigma-2}x^2  \\
&= |z|^{2\sigma-1}\Big( \frac{z^2}{|z|}-\frac{x^2}{|x|}\Big)+|z|^{2\sigma}\frac{x^2}{|x|}\Big(\frac{1}{|z|}-\frac{1}{|x|}\Big)+\big(|z|^{2\sigma}-|x|^{2\sigma}\big)\frac{x^2}{|x|^2}\\
&= I_1+I_2+I_3.
\end{align*}
Concerning $I_1$, we invoke the previous case $\sigma\le 1/2$ to infer 
\begin{align*}
I_1 =|z|^{2\sigma-1}\Big( |z|\frac{z^2}{|z|^2}-|x|\frac{x^2}{|x|^2}\Big) \le c\, |z|^{2\sigma-1}|z-x|.
\end{align*}
Regarding $I_2$, an application of the triangle inequality gives
\begin{align*}
I_2=|z|^{2\sigma}\frac{x^2}{|x|}\frac{|x|-|z|}{|x|\,|z|}\le |z|^{2\sigma-1}|z-x|.
\end{align*}
Lastly, to estimate $I_3$, since $2\sigma>1$, we employ the mean value theorem and obtain
\begin{align*}
I_3 = \big(|z|^{2\sigma}-|x|^{2\sigma}\big)\frac{x^2}{|x|^2} \le c\, |z-x|(|x|^{2\sigma-1}+|z|^{2\sigma-1}).
\end{align*}
Altogether, we deduce
\begin{align*}
|z|^{2\sigma-2}z^2-|x|^{2\sigma-2}x^2\le c\, |z-x|(|x|^{2\sigma-1}+|z|^{2\sigma-1}).
\end{align*}
This produces inequality \eqref{ineq:|z|^(2-2sigma)z^2-|x|^(2-2sigma)<|x-z|^2sigma:sigma>1/2} for the case $\sigma>1/2$. The proof is thus finished.

\end{proof}

\bibliographystyle{abbrv}
{\footnotesize\bibliography{wave-bib}}

@article{barbu2014stochastic,
  title={{Stochastic nonlinear Schr{\"o}dinger equations with linear multiplicative noise: rescaling approach}},
  author={Barbu, Viorel and R{\"o}ckner, Michael and Zhang, Deng},
  journal={J. Nonlinear Sci.},
  volume={24},
  number={3},
  pages={383--409},
  year={2014},
  publisher={Springer}
}

@article{barbu2016stochastic,
  title={{Stochastic nonlinear Schr{\"o}dinger equations}},
  author={Barbu, Viorel and R{\"o}ckner, Michael and Zhang, Deng},
  journal={Nonlinear Anal. Theory Methods Appl.},
  volume={136},
  pages={168--194},
  year={2016},
  publisher={Elsevier}
}

@article{barbu2017stochastic,
  title={{Stochastic nonlinear Schr{\"o}dinger equations: no blow-up in the non-conservative case}},
  author={Barbu, Viorel and R{\"o}ckner, Michael and Zhang, Deng},
  journal={J. Differ. Equ.},
  volume={263},
  number={11},
  pages={7919--7940},
  year={2017},
  publisher={Elsevier}
}

@article{barbu2017stochasticLOG,
  title={{The stochastic logarithmic Schr{\"o}dinger equation}},
  author={Barbu, Viorel and R{\"o}ckner, Michael and Zhang, Deng},
  journal={J. Math. Pures Appl.},
  volume={107},
  number={2},
  pages={123--149},
  year={2017},
  publisher={Elsevier}
}

@book{bergh2012interpolation,
  title={{Interpolation Spaces: an Introduction}},
  author={Bergh, J{\"o}ran and L{\"o}fstr{\"o}m, J{\"o}rgen},
  volume={223},
  year={2012},
  publisher={Springer Science \& Business Media}
}

@article{bessaih2020invariant,
  title={{Invariant measures for stochastic damped 2D Euler equations}},
  author={Bessaih, Hakima and Ferrario, Benedetta},
  journal={Commun. Math. Phys.},
  volume={377},
  number={1},
  pages={531--549},
  year={2020},
  publisher={Springer}
}

@article{brzezniak2019martingale,
  title={{Martingale solutions for the stochastic nonlinear Schr{\"o}dinger equation in the energy space}},
  author={Brze{\'z}niak, Zdzis{\l}aw and Hornung, Fabian and Weis, Lutz},
  journal={Probab. Theory Relat. Fields},
  volume={174},
  pages={1273--1338},
  year={2019},
  publisher={Springer}
}

@article{brzezniak2023invariant,
  title={{Invariant measures for a stochastic nonlinear and damped 2D Schr{\"o}dinger equation}},
  author={Brze{\'z}niak, Zdzis{\l}aw and Ferrario, Benedetta and Zanella, Margherita},
  journal={Nonlinearity},
  volume={37},
  number={1},
  pages={015001},
  year={2023},
  publisher={IOP Publishing}
}

@article{brzezniak2023ergodic,
  title={{Ergodic results for the stochastic nonlinear Schr{\"o}dinger equation with large damping}},
  author={Brze{\'z}niak, Zdzislaw and Ferrario, Benedetta and Zanella, Margherita},
  journal={J. Evol. Equ.},
  volume={23},
  number={1},
  pages={19},
  year={2023},
  publisher={Springer}
}

@article{brzezniak2025global,
  title={{Global well posedness and ergodic results in regular Sobolev spaces for the nonlinear Schr{\"o}dinger equation with multiplicative noise and arbitrary power of the nonlinearity}},
  author={Brze{\'z}niak, Zdzis{\l}aw and Ferrario, Benedetta and Maurelli, Mario and Zanella, Margherita},
  journal={Discrete Contin. Dyn. Syst.},
  volume={45},
  number={9},
  pages={3217--3257},
  year={2025},
  publisher={Discrete and Continuous Dynamical Systems}
}

@book{cazenave2003semilinear,
  title={{Semilinear Schrodinger equations}},
  author={Cazenave, T},
  publisher={New York University, Courant Institute of Mathematical Sciences},
  year={2003}
}

@article{cerrai2020convergence,
  title={On the convergence of stationary solutions in the Smoluchowski-Kramers approximation of infinite dimensional systems},
  author={Cerrai, Sandra and Glatt-Holtz, Nathan},
  journal={J. Funct. Anal.},
  volume={278},
  number={8},
  pages={108421},
  year={2020},
  publisher={Elsevier}
}

@article{chen2025exponential,
  title={{Exponential mixing for the randomly forced NLS equation}},
  author={Chen, Yuxuan and Xiang, Shengquan and Zhang, Zhifei and Zhao, Jia-Cheng},
  journal={arXiv preprint arXiv:2506.10318},
  year={2025}
}

@article{cui2019global,
  title={{On global existence and blow-up for damped stochastic nonlinear Schr{\"o}dinger equation}},
  author={Cui, Jianbo and Hong, Jialin and Sun, Liying},
  journal={Discrete Continuous Dyn. Syst. Ser. B},
  volume={25},
  number={12},
  year={2019}
}

@article{de1999stochastic,
  title={{A stochastic nonlinear Schr{\"o}dinger equation with multiplicative noise}},
  author={de Bouard, Anne and Debussche, Arnaud},
  journal={Commun. Math. Phys.},
  volume={205},
  pages={161--181},
  year={1999},
  publisher={Springer}
}

@article{de2003stochastic,
  title={{The Stochastic Nonlinear Schr{\"o}dinger Equation in $H^1$}},
  author={de Bouard, A and Debussche, A},
  journal={Stoch. Anal. Appl.},
  volume={21},
  number={1},
  pages={97--126},
  year={2003},
  publisher={Informa UK Limited}
}

@article{de2010nonlinear,
  title={{The nonlinear Schr{\"o}dinger equation with white noise dispersion}},
  author={De Bouard, Anne and Debussche, Arnaud},
  journal={J. Funct. Anal.},
  volume={259},
  number={5},
  pages={1300--1321},
  year={2010},
  publisher={Elsevier}
}

@article{debussche2005ergodicity,
  title={{Ergodicity for a weakly damped stochastic non-linear Schr{\"o}dinger equation}},
  author={Debussche, Arnaud and Odasso, Cyril},
  journal={J. Evol. Equ.},
  volume={5},
  number={3},
  pages={317--356},
  year={2005},
  publisher={Springer}
}

@article{debussche20111d,
  title={{1D quintic nonlinear Schr{\"o}dinger equation with white noise dispersion}},
  author={Debussche, Arnaud and Tsutsumi, Yoshio},
  journal={J. Math. Pures Appl.},
  volume={96},
  number={4},
  pages={363--376},
  year={2011},
  publisher={Elsevier}
}

@article{dong2023ergodicity,
  title={Ergodicity for stochastic conservation laws with multiplicative noise},
  author={Dong, Zhao and Zhang, Rangrang and Zhang, Tusheng},
  journal={Commun. Math. Phys.},
  volume={400},
  number={3},
  pages={1739--1789},
  year={2023},
  publisher={Springer}
}

@article{ekren2017existence,
  title={{Existence of invariant measures for the stochastic damped Schr{\"o}dinger equation}},
  author={Ekren, Ibrahim and Kukavica, Igor and Ziane, Mohammed},
  journal={Stoch. Partial Differ. Equ.: Anal. Comput.},
  volume={5},
  pages={343--367},
  year={2017},
  publisher={Springer}
}

@article{ferrario2023uniqueness,
  title={{Uniqueness of the invariant measure and asymptotic stability for the 2D Navier Stokes equations with multiplicative noise}},
  author={Ferrario, Benedetta and Zanella, Margherita},
  journal={arXiv preprint arXiv:2307.03483},
  year={2023}
}

@article{ferrario2025stationary,
  title={{Stationary solutions for the nonlinear Schr{\"o}dinger equation}},
  author={Ferrario, Benedetta and Zanella, Margherita},
  journal={Stoch. Partial Differ. Equ.: Anal. Comput.},
  pages={1--53},
  year={2025},
  publisher={Springer}
}

@article{gao2024polynomial,
  title={{Polynomial mixing for white-forced Kuramoto-Sivashinsky equation on the whole line}},
  author={Gao, Peng},
  journal={arXiv preprint arXiv:2408.00592},
  year={2024}
}

@article{gao2024polynomialWAVE,
  title={Polynomial mixing for the white-forced wave equation on the whole line},
  author={Gao, Peng},
  journal={arXiv preprint arXiv:2412.13230},
  year={2024}
}

@article{glatt2021long,
  title={{On the long-time statistical behavior of smooth solutions of the weakly damped, stochastically-driven KdV equation}},
  author={Glatt-Holtz, Nathan and Martinez, Vincent R and Richards, Geordie H},
  journal={Trans. Amer. Math. Soc. },
  year={2025}
}

@article{hornung2018nonlinear,
  title={{The nonlinear stochastic Schr{\"o}dinger equation via stochastic Strichartz estimates}},
  author={Hornung, Fabian},
  journal={J. Evol. Equ.},
  volume={18},
  pages={1085--1114},
  year={2018},
  publisher={Springer}
}

@book{karatzas2012brownian,
  title={{Brownian Motion and Stochastic Calculus}},
  author={Karatzas, Ioannis and Shreve, Steven},
  volume={113},
  year={2012},
  publisher={Springer Science \& Business Media}
}

@article{kim2006invariant,
  title={{Invariant measures for a stochastic nonlinear Schr{\"o}dinger equation}},
  author={Kim, Jong Uhn},
  journal={Indiana Math. J.},
  pages={687--717},
  year={2006},
  publisher={JSTOR}
}

@article{kim2008stochastic,
  title={On the stochastic wave equation with nonlinear damping},
  author={Kim, Jong Uhn},
  journal={Appl. Math. Optim.},
  volume={58},
  number={1},
  pages={29--67},
  year={2008},
  publisher={Springer}
}

@article{liu2024exponential,
  title={Exponential mixing for random nonlinear wave equations: weak dissipation and localized control},
  author={Liu, Ziyu and Wei, Dongyi and Xiang, Shengquan and Zhang, Zhifei and Zhao, Jia-Cheng},
  journal={arXiv preprint arXiv:2407.15058},
  year={2024}
}

@article{martirosyan2014exponential,
  title={Exponential mixing for the white-forced damped nonlinear wave equation},
  author={Martirosyan, Davit},
  journal={Evol. Equ. Control Theory.},
  volume={3},
  number={4},
  pages={645},
  year={2014},
  publisher={American Institute of Mathematical Sciences}
}

@article{nersesyan2008polynomial,
  title={{Polynomial mixing for the complex Ginzburg--Landau equation perturbed by a random force at random times}},
  author={Nersesyan, Vahagn},
  journal={J. Evol. Equ.},
  volume={8},
  number={1},
  pages={1--29},
  year={2008},
  publisher={Springer}
}

@article{nersesyan2024exponential,
  title={{Exponential mixing for the white-forced complex Ginzburg--Landau equation in the whole space}},
  author={Nersesyan, Vahagn and Zhao, Meng},
  journal={SIAM J. Math. Anal.},
  volume={56},
  number={3},
  pages={3646--3678},
  year={2024},
  publisher={SIAM}
}

@article{nersesyan2024polynomial,
  title={{Polynomial mixing for the white-forced Navier-Stokes system in the whole space}},
  author={Nersesyan, Vahagn and Zhao, Meng},
  journal={arXiv preprint arXiv:2410.15727},
  year={2024}
}

@article{nguyen2023small,
  title={The small mass limit for long time statistics of a stochastic nonlinear damped wave equation},
  author={Nguyen, Hung D},
  journal={J. Diff. Equ.},
  volume={371},
  pages={481--548},
  year={2023},
  publisher={Elsevier}
}

@article{nguyen2024polynomial,
  title={Polynomial mixing of a stochastic wave equation with dissipative damping},
  author={Nguyen, Hung D},
  journal={Appl. Math. Opt.},
  volume={89},
  pages={1--31},
  year={2024},
  publisher={Springer}
}

@article{nguyen2024inviscid,
  title={{The inviscid limit for long time statistics of the one-dimensional stochastic Ginzburg-Landau equation}},
  author={Nguyen, Hung D},
  journal={arXiv preprint arXiv:2403.08951},
  year={2024}
}

@article{odasso2006ergodicity,
  title={{Ergodicity for the stochastic complex Ginzburg-Landau equations}},
  author={Odasso, Cyril},
  journal={Ann. Inst. Henri Poincare (B) Probab. Stat.},
  volume={42},
  number={4},
  pages={417--454},
  year={2006}
}

@article{roy2025large,
  title={{Large deviation principle for a stochastic nonlinear damped Schr\"odinger equation}},
  author={Roy, Sandip and Mukherjee, Debopriya and Mohan, Manil Thankamani},
  journal={arXiv preprint arXiv:2510.06110},
  year={2025}
}

@book{triebel1978interpolation,
  title={{Interpolation Theory, Function Spaces, Differential Operators}},
  author={Triebel, Hans},
  publisher={North
Holland Mathematical Library, 18. North Holland, Amsterdam-New York},
  year={1978}
}

\end{document}